\documentclass[12pt]{amsart}

\usepackage{tikz}
\usepackage{amsmath}
\usepackage{amssymb,amsfonts,mathrsfs}
\usepackage[colorlinks=true,linkcolor=blue,citecolor=blue]{hyperref}
\usepackage[driver=pdftex,margin=3cm,heightrounded=true,centering]{geometry}

\usepackage{mllocal} 
\newcommand{\cH}{\mathcal H}

\begin{document}

\title[Analytic Torsion on Manifolds with Conical Singularities]
{The Metric Anomaly of Analytic Torsion on Manifolds with Conical 
Singularities}

\date{\today}

\author{Werner M\"uller}
\address{Mathematisches Institut,
Universit\"at Bonn,
53115 Bonn,
Germany}
\email{mueller@math.uni-bonn.de}
\urladdr{www.math.uni-bonn.de/people/mueller}

\author{Boris Vertman}
\address{Mathematisches Institut,
Universit\"at Bonn,
53115 Bonn,
Germany}
\email{vertman@math.uni-bonn.de}
\urladdr{www.math.uni-bonn.de/people/vertman}

\thanks{Both authors were partially supported by the 
Hausdorff Center for Mathematics}


\begin{abstract}
{In this paper we study the analytic torsion of an 
odd-dimensional manifold with isolated conical 
singularities. First we show that the analytic torsion is invariant under
 deformations of the metric which are of higher order near the singularities. 
Then we identify the metric anomaly of analytic torsion for a bounded 
generalized cone at its regular boundary in terms of spectral information of 
the cross-section. In view of previous computations of analytic torsion on 
cones, this leads to a detailed geometric 
identification of the topological and spectral contributions to analytic 
torsion, arising from the conical 
singularity. The contribution exhibits a torsion-like spectral invariant of 
the cross-section of the cone, which we study under scaling of the metric on 
the cross-section.}
\end{abstract}

\maketitle
\tableofcontents

\section{Introduction} 

In this paper we study the Ray-Singer analytic torsion of a compact 
odd-dimensional Riemannian manifold with an isolated conical singularity. This means
the following. We consider an open Riemannian manifold $(M,g^M)$ 
of odd dimension which admits a decomposition 
\[
M=\cCN\cup_N X,
\]
into  a compact Riemannian manifold $X$ with boundary $N$ and an open
truncated generalized cone $\cCN$ over $N$. Here
$\cCN=(0,1]\times N$, and the metric on $\cCN$ takes the 
special form
\begin{align*}
g^M|_{\cCN}:=dx^2\oplus x^2g^N(x), \ x\in (0,1),
\end{align*}
where $g^N(x)$ is a family of metrics on $N$ which is smooth up to $0$. 
Our goal is to study the analytic torsion of such manifolds. 
We mention that the case of an even-dimensional cone has been 
studied by the second named author \cite{Ver:TMA} and 
Hartmann-Spreafico \cite{HarSpr:TAT}. This case is simpler than the odd-dimensional case.

The analytic torsion of a compact Riemannian manifold, equipped with a flat
Hermitian vector bundle,  was introduced by
Ray and Singer \cite{RaySin:RTA} as an analytic counterpart of the 
Reidemeister-Franz torsion. The latter is defined combinatorial, whereas the 
analytic torsion is 
defined as a weighted product of regularized determinants of the
Laplacian on $p$-forms with values in the flat bundle. The conjecture of Ray 
and Singer that the two invariants are equal was proved independently by
Cheeger \cite{Che:ATA} and M\"uller \cite{Mue:ATA}.  L\"uck
\cite{Lue:AAT} and Vishik \cite{Vis:GRS} extended this result to compact
Riemannian manifolds with boundary under the assumption that the metric is
a product near the boundary. Equality still holds up to a term which is given
by the Euler characteristic of the boundary. If the metric is not a product
near the boundary, additional terms, called anomaly, arise in the difference 
between the  analytic and the combinatorial torsion (see \cite{DaiFan:ATA}
and \cite{BruMa:AAF}).
 
A natural problem is to determine whether the equality of analytic
torsion and Reidemeister torsion for compact Riemannian manifolds has any
analogue for manifolds with singularities. To begin with one needs to define
both invariants for such spaces. On the combinatorial site, Dar \cite{Dar:IRT} 
has defined the intersection R-torsion for pseudo-manifolds. Furthermore
he has  shown that the analytic torsion is  well-defined for compact 
manifolds with conical singularities. Therefore it makes sense
to study the relation between analytic torsion and intersection torsion for
compact manifolds with conical singularities. The present paper is a 
contribution to this goal. As already mentioned, we study the analytic torsion
of a manifold with conical singularities.

First we derive a variational formula which shows that we may assume
that the metric in a neighborhood of the cone tip is an exact cone metric. 
Using a splitting formula (which we don't prove in this paper) together with
\cite{BruMa:AAF}, it follows that
one can replace $M$ by an exact cone $C(N)$, i.e., $C(N)=(0,1]\times N$,  
equipped with the metric $dx^2+x^2 g^N(0)$, and study the problem for this 
case. 

In \cite{Ver:ATO} the second named author has derived a formula
which expresses the analytic torsion of $C(N)$ in terms of spectral invariants
of $N$. In the present paper we will identify some of these terms. More precisely
we will show that the analytic torsion of $C(N)$ is the sum of three terms.
The first one is a linear combination of Betti numbers. The second term 
equals the metric anomaly of Br\"uning and Ma \cite{BruMa:AAF}.
The remaining term is a torsion-like spectral invariant of the cross section 
$N$.

The paper is organized as follows. In  Section \ref{statement}
we will summarize the basic facts about analytic torsion of manifolds with
conical singularities and state the main results. 
In Section \ref{section-variation} we will establish invariance of the 
analytic torsion under higher order 
deformations of the conical metric, assuming essential self-adjointness of 
the Gau\ss-Bonnet operator. The result is a special case of a more 
general investigation of analytic torsion on manifolds with incomplete
edge singularities by Mazzeo and the second author in \cite{MazVer:ATO}. 
Essential self-adjointness of the Gau\ss-Bonnet operator is posed only within the limits of 
Section \ref{section-variation} and can always be achieved by appropriate rescaling 
of $g^N(0)$.

In Section \ref{section-scaling} we discuss the metric anomaly formula for 
Ray-Singer analytic torsion on manifolds with boundary,
established by Br\"uning-Ma in \cite{BruMa:AAF}. Especially we prove that the 
metric anomaly is scaling invariant.

We proceed in Section \ref{section-truncated} with computing analytic torsion 
of a truncated cone.  This leads in Section \ref{section-anomaly} to an 
explicit identification of the metric anomaly at the regular boundary 
of the cone in terms of the spectral information of the cross-section $N$. 
This shows that the main contribution of the conical singularity to the
analytic torsion is given by  a  torsion-like spectral invariant of the 
cross section.

In Section \ref{section-torsion-like} we study the behavior of this 
invariant under scaling of the  metric on the base of the cone
 and compute its value for some examples.

\section*{Acknowledgments}
It is a pleasure for the second author to acknowledge the helpful discussions 
with Rafe Mazzeo and Jeff Cheeger. Furthermore both authors would like to 
thank Jochen Br\"uning and Xiaonan Ma for explaining important aspects of the 
metric anomaly for analytic torsion on manifolds with boundary, 
as well as Matthias Lesch. The second author would like 
to express his gratitude to Stanford University where the research has 
partially 
been conducted for hospitality, and Hausdorff Center for Mathematics in Bonn.

\section{Preliminaries and statement of main results}
\label{statement}

In this section we summarize some facts about the analytic torsion for
manifolds with conical singularities and we state the main results.

\subsection{Definition and variation of Analytic Torsion}

Let $(M^{m},g^M)$ be an odd-dimensional oriented Riemannian 
manifold 
with an isolated conical singularity. Thus $M=\cCN\cup_NX$, where 
$(X,g|_X)$ is a 
compact Riemannian manifold with boundary $\partial X=N$ and $\cCN$
is a truncated generalized cone, which means that $\cCN=(0,1)
\times N$ and  the metric on $\cCN$ is given by
\begin{align}\label{metric1}
g^M|_{\cCN}:=dx^2\oplus x^2g^N(x), 
\end{align}
where $\ g^N\in C^{\infty}\left([0,1], \textup{Sym}^2(T^*N)\right)$. Actually
the assumptions on the the behavior of $g^N(x)$ at $x=0$ can be weakened
(see \cite[Section 2]{BruLes:KHT}). For simplicity we will work with the 
stronger
assumption that $g^N(x)$ is smooth up to $0$.  

Fix a base point $x_0\in M$ and let 
\begin{equation}\label{repr}
\rho:\pi_1(M,x_0)\to U(n,\C)
\end{equation}
be a unitary representation of the fundamental group $\pi_1(M,x_0)$. 
Let $(E,\nabla,h^E)$ be the flat Hermitian vector bundle over $M$, associated
to the representation $\rho$. Here $h^E$ is the flat Hermitian metric 
on $E$, induced by the standard Hermitian inner product on $\C^n$. We may 
assume that $x_0=(t,y_0)\in\cCN$. Then $E|_{\cCN}$ is the flat bundle 
associated to the composition $\rho_N$ of $\pi_1(\cCN,x_0)\to \pi_1(M,x_0)$ and
$\rho$. Note that $\pi_1(N,y_0)\cong\pi_1(\cCN,x_0)$. So we may regard $\rho_N$
as a representation of $\pi_1(N,y_0)$. Let $(E_N,\nabla_N,h_N)$ be the 
associated flat Hermitian bundle. Then we have
\begin{equation}\label{flatbd}
E|_{\cCN}\cong p^*(E_N)\cong (0,1]\times E_N,
\end{equation}
where $p\colon (0,1]\times N\to N$ is the canonical projection. Moreover the
isomorphism \eqref{flatbd} is compatible with the Hermitian metrics and the
flat connections. The flat covariant derivatives $\nabla$ and $\nabla_N$ 
are related as follows. For $s\in C^{\infty}([0,1], C^{\infty}(E_N))\cong
\Gamma(E|_{\cCN})$ we have
\begin{align}
\nabla s= \frac{\partial s}{\partial x}\otimes dx + \nabla_N s.
\end{align}
Let $(\Omega^*(M,E),d_*)$ be the associated twisted de Rham complex, 
where $\Omega^*(M,E)$ denotes the space of $E$-valued differential forms. 
Denote by $\Omega^*_c(M,E)$ the subspace of differential forms with 
compact support. Let $\Delta_k$ be the corresponding Laplace operator on
$k$-forms. We regard $d_k$ and $\Delta_k$ as unbounded operators in 
$L^2\Omega^k(M,E)$ with domain $\Omega^k_c(M,E)$.

Consider the maximal extension $d_{k,\max}$ of $d_k$ in 
$L^2\Omega^k(M,E)$ with domain $\dom (d_{k,\max})$.
The minimal extension $d_{k,\min}$ of $d_k$ with domain 
$\dom (d_{k,\min}) \subset \dom (d_{k,\max})$ is the 
graph closure of $d_k$. Ideal boundary conditions for the 
de Rham complex $(\Omega^*(M,E),d_*)$ is a choice 
of closed extensions $D_k$ of $d_k$ for each $k=0,...,\dim M$ with
\[
\dom (d_{k,\min}) \subseteq \dom (D_k) \subseteq \dom (d_{k,\max}),
\] 
which combine into a Hilbert complex in the sense of 
\cite{BruLes:KHT}. Ideal boundary conditions for the de Rham complex 
induce a self-adjoint extension for each $\Delta_k$. The two special 
cases are the relative and absolute boundary conditions
\begin{equation}
\begin{split}
\Delta_{k, \textup{rel}} &= 
d^*_{k,\min}d_{k,\min} + d_{k-1,\min}d^*_{k-1,\min}, \\
\Delta_{k, \textup{abs}} &= 
d^*_{k,\max}d_{k,\max} + d_{k-1,\max}d^*_{k-1,\max}. \\
\end{split}
\end{equation}
We recall from \cite{Che:SGO} and \cite[Theorem 3.8]{BruLes:KHT}
the following facts. Ideal boundary conditions are unique in degrees 
$k\neq (m-1)/2$, 
$m=\dim M$. If we assume the Witt-condition $H^p(N,E_N)=0$,
$p=(m-1)/2$, then ideal boundary conditions are unique also in degree $p$. 

However, the Witt condition which implies uniqueness of the ideal boundary 
conditions, 
does not imply essential self-adjointness of $\Delta_k$ and the Laplacian 
corresponding to 
the (unique) ideal boundary conditions need not be given by the Frederich's 
extension in degree 
$k=(m\pm 1)/2$. In fact, by \cite[Corollary 3.5]{BruLes:KHT} the extensions 
$\Delta_{k, \textup{rel}}$ and $\Delta_{k, \textup{abs}}$
are given by the Friedrich extension in any degree if and only if the 
Gau\ss-Bonnet operator is essentially self-adjoit. This can always be 
achieved by scaling 
$g^N(0)$ to $c^2g^N(0)$ with $c>0$ sufficiently small.

Choose either the relative or the absolute self-adjoint extension of 
$\Delta_k$, which we denote again by the same letter for the moment.
It is known that {$\Delta_{k}$ has pure point spectrum \cite{Les:OOF}
and by \cite[Corollary 2.4.3]{Les:OOF} it follows that for every $t>0$,
$\exp(-t\Delta_{k})$ is a trace class operator.
Moreover we have the following theorem.
\begin{thm}\label{heat-trace}
Let $(M,g)$ be a compact odd-dimensional Riemannian manifold of dimension $m$
with a
conical singularity.  Let $(E,\nabla, h^E)$ be a flat Hermitian vector bundle 
over $M$ and $(E_N,\nabla_N,h_N)$ the associated bundle over $N$.
If $g^N(x)$ is not constant in $x\in [0,1]$,  
we assume essential self-adjointness of the Gau\ss-Bonnet operator. Then for 
the absolute or relative self-adjoint extension 
$\Delta_k$ of the Laplacian on $\Omega^k_c(M,E)$, $k=0,...,m,$ we have
\begin{align}\label{heatasymp}
\textup{Tr}\, \left(e^{-t\Delta_{k}} \right)\sim_{t\to 0+} 
\sum_{j=0}^{\infty}A_kt^{j-\frac{m}{2}}
+\sum_{j=0}^{\infty}C_jt^{\frac{j}{2}}+\sum_{j=1}^{\infty}G_jt^{\frac{j}{2}}\log t.
\end{align}
\end{thm}
For $g^N\equiv g^N(0)$ the result is well-known (see \cite{Che:SGO}). 
Otherwise, 
asymptotic expansions with $G_0=0$ are known only for the Frederich's extension 
$\Delta^{\mathscr{F}}$ (see \cite{Moo:HKA}) and essential self-adjointness of 
the Gau\ss-Bonnet operator is imposed above to guarantee 
$\Delta=\Delta^{\mathscr{F}}$.

\begin{remark}
In fact it suffices to impose the Witt condition $H^p(N,E_N)=0$,
$p=(m-1)/2$ instead of essential self-adjointness of the 
Gau\ss-Bonnet operator for Theorem \ref{heat-trace} to hold in case
$g^N(x)$ is not constant in $x\in [0,1]$. Under the 
Witt condition the heat kernel of $\Delta_k$ is still polyhomogeneous 
on a certain parabolic blowup of $\R^+\times M^2$ (see \cite{Moo:HKA})
and the heat trace expansion follows by the Push-forward theorem. We confine 
ourselves to the case of essentially self-adjoint Gau\ss-Bonnet operator
for simplicity.
\end{remark}

Let $\lambda_0\le \lambda_1\le\cdots $ denote the eigenvalues of $\Delta_k$. 
It follows from \eqref{heatasymp} that Weyl's law holds for the counting
function of the eigenvalues. This implies that the zeta function
\[
\zeta(s;\Delta_k):=\sum_{\lambda_j>0}\lambda_j^{-s}
\]
converges in the half-plane $\Re(s)>m/2$ and it can be expressed in terms of 
the trace of the heat operator by
\begin{align}
\zeta(s, \Delta_k)=\frac{1}{\Gamma(s)}\int_0^{\infty}t^{s-1}
\textup{Tr}(e^{-t\Delta_k}-P_k) \;dt, \quad \textup{Re}(s)>m/2,
\end{align}
where $P_k$ denotes the orthogonal projection 
of $L^2\Omega^k(M,E)$ onto the subspace $\mathscr{H}^k_{(2)}(M,E)$ of harmonic 
$k$-forms. 
Then the asymptotic expansion \eqref{heatasymp} yields the meromorphic 
extension of the right hand side and hence, of the zeta function, to the 
whole complex plane. Furthermore it also follows from \eqref{heatasymp} that
$\zeta (s, \Delta_k)$ is regular at $s=0$.
We can now define the analytic torsion of $(M,g^M)$ and $(E,\nabla, h^E)$ as
in the compact case. Let $\Delta_{k,\textup{abs}}$ (resp. 
$\Delta_{k,\textup{rel}}$) denote the self-adjoint extension of $\Delta_k$ with
respect to the choice of absolute (resp. relative) boundary conditions. Recall
that boundary conditions need to be imposed only if $k=(m\pm 1)/2$. Then we define
\begin{align}\label{analtor}
\log T_{\textup{abs}}(M,E;g^M):=\frac{1}{2} \sum_{k=0}^m(-1)^k\cdot k \cdot 
\frac{d}{ds}\zeta(s,\Delta_{k,\textup{abs}})\big|_{s=0},
\end{align}
and define $T_{\textup{rel}}(M,E;g^M)$ similarly. 
Since $h^E$ is the canonical metric on $E$, induced by the unitary 
representation, we suppress it from the notation. Observe that 
$\ast \Delta_{k,\textup{abs}}=\Delta_{m-k,\textup{rel}}\ast$. 
Denote the Laplacians on closed or coclosed differential forms 
by an additional subscript $cl$ or $ccl$, respectively.

Since the dimension of $M$ is odd, it follows by Poincar\'e duality that
\begin{equation}\label{poincare}
\begin{split}
T_{\textup{\textup{abs}}}(M,E;g^M)
&=\sum_{k=0}^{m-1} (-1)^{k+1}\frac{d}{ds}\zeta(s,\Delta_{k,ccl,\textup{abs}})
\big|_{s=0} \\
&=\sum_{k=0}^{m-1} (-1)^{k+1}\frac{d}{ds}\zeta(s,\Delta_{m-k,cl,\textup{rel}})
\big|_{s=0} \\
&=\sum_{k=0}^{m-1} (-1)^{k+1}\frac{d}{ds}\zeta(s,\Delta_{k,ccl,\textup{rel}})\big|_{s=0}
=T_{\textup{rel}}(M,E;g^M).
\end{split}
\end{equation}
For the second equality we have used that $\ast\Delta_{k,ccl,\textup{abs}}=
\Delta_{m-k,cl,\textup{rel}}\ast$.  The third equality follows from the fact 
that for an eigenvalue $\lambda\neq0$, $d$ is an isomorphism between the 
$\lambda$-eigenspace 
of the Laplacian on co-closed forms, satisfying relative boundary conditions,
 and the $\lambda$-eigenspace of the Laplacian on closed forms, satisfying  
relative boundary conditions. So we will denote the torsion 
 simply by $T(M,E;g^M)$. 

\begin{remark}
Dar \cite{Dar:IRT} also discusses the analytic torsion of a manifold with
conical singularities, where the  cross section $N$ is allowed to have a 
non-empty boundary. In this case the zeta-functions may have a pole at $s=0$.
However, due to residue cancellations, the analytic torsion is still
well-defined.  If $\partial N=\emptyset$, the case treated in the present 
paper, the zeta-functions are regular at zero individually. 
\end{remark}

For a finite-dimensional vector space $F$, set
$\det F:=\Lambda^{\textup{max}}F$ and denote by $(\det F)^{-1}$ the dual line
$\det F^*$. Let $\mathscr{H}^*_{(2)}(M,E)$ be the space of 
square integrable harmonic forms $\mathscr{H}^*_{(2)}(M,E)$ which satisfy
absolute boundary conditions.
The determinant line  of $(M,E)$ is defined as
\begin{equation}\label{det-line}
\det \mathscr{H}^*_{(2)}(M,E) :=\bigotimes_{k=0}^{m}
\left[\det \mathscr {H}^k_{(2)}(M,E)\right]^{(-1)^{k+1}}.
\end{equation}
 
\begin{defn}\label{t-norm-def}
The Ray-Singer metric $\|\cdot \|^{RS}_{(M,E;g^M)}$  on 
$\det\mathscr {H}^k_{(2)}(M,E) $ is defined by
\begin{equation}
\|\cdot \|^{RS}_{(M,E;g^M)}:=T(M,E;g^M)\cdot 
\|\cdot \|_{\det \mathscr {H}^k_{(2)}(M,E)},
\end{equation}
where $\|\cdot \|_{\det\mathscr {H}^k_{(2)}(M,E)}$ is the norm on 
$\det \mathscr {H}^k_{(2)}(M,E)$ induced by the $L^2$-norm on the space of 
harmonic forms $\mathscr{H}^*_{(2)}(M,E)$.
\end{defn}

Using methods developed by Melrose \cite{Mel:TAP}, we study the variation of
the analytic torsion with respect to the variation of the metric $g^M$.
Our main result is the invariance of analytic torsion of manifolds with 
isolated conical singularities under higher order deformations of the metric 
near the conical tip, assuming essential self-adjointness of the Gau\ss-Bonnet operator.
More precisely, we have
\begin{thm}\label{main1}
Let $(M^m,g^M)$ be a compact odd-dimensional Riemannian manifold with an 
isolated conical singularity, 
and a neighborhood $\mathscr{U}\cong (0,1)\times N$ of the singularity with
\begin{align*}
g^M\restriction \mathscr{U}=dx^2\oplus x^2g^N(x), \ g^N\in C^{\infty}
\left([0,1], \textup{Sym}^2(T^*N)\right).
\end{align*}
Let $(E,\nabla, h^E)$ be a flat Hermitian vector bundle and 
$(E_N,\nabla_N,h_N)$ its restriction over  $N$.
Assume essential self-adjointness of the Gau\ss-Bonnet operator. 
Then the Ray-Singer metric $\|\cdot \|^{RS}_{(M,E;g^M)}$ is 
invariant under all deformations of $g^M$ that fix $g^N(0)$.
\end{thm}

\subsection{Analytic Torsion of a bounded Cone}

It follows from Theorem \ref{main1} that for spaces with isolated conical 
singularities with appropriately rescaled $g^N \equiv g^N(0)$, so that the Gau\ss-Bonnet 
operator is essentially self-adjoint, we may assume that
the neighborhood of the singularity is a bounded  exact cone $C(N)$, i.e., 
\begin{equation}\label{excone}
C(N)=(0,1]\times N, \quad g=dx^2+x^2 g^N, \ g^N \in \textup{Sym}^2(T^*N).
\end{equation}
Now recall that for a 
compact manifold there is a splitting formula for the analytic torsion
\cite{Vis:GRS}. It is very likely that the splitting formula continuous to
hold for compact manifolds with conical singularities. Assuming this, the
study of the analytic torsion for such manifolds can be reduced to the study
of analytic torsion of a bounded cone. In \cite{Ver:ATO}, the second author
has derived a formula which expresses the analytic torsion of a bounded cone
in terms of spectral and topological data of the cross section. The computation 
uses the double-summation 
method developed by Spreafico in \cite{Spr:ZFA}, \cite{Spr:ZIF}, and a 
symmetry observation by 
Lesch \cite{Les:TAT}. We recall the result in odd dimensions. 

\begin{thm}\label{BV-Theorem}\textup{(Vertman, \cite{Ver:ATO})}
Let $(C(N),g)$ be a bounded
cone over a closed even-dimensional oriented Riemannian manifold 
$(N^n,g^N)$. Let $(E,\nabla, h^E)$ be a flat Hermitian vector bundle over 
$C(N)$and $(E_N,\nabla_N,h_N)$ the associated flat Hermitian vector bundle over
$N$. 
Let $b_k=\dim H^k(N,E_N)$, $k=0,\dots,n$, and $\chi(N,E_N)$ the 
Euler characteristic of $(N,E_N)$. Denote by $\Delta_{k,ccl,N}$ the 
Laplacian on coclosed $E_N$-valued $k$-forms on $N$ and define
\begin{align*}
&\A_k:=\frac{(n-1)}{2}-k, \quad 
\nu(\eta)=\sqrt{\eta + \A_k^2}, \ \textup{for} \ \eta \in E_k=\textup{Spec}\Delta_{k,ccl,N}\backslash \{0\}, \\
&\zeta_{k,N}(s)=\sum_{\eta \in E_k} \textup{m}(\eta) \ \nu(\eta)^{-s},\quad \zeta_{k,N}(s, \pm \A_k):=
\sum_{\eta\in E_k} \textup{m}(\eta) \ (\nu(\eta)\pm \A_k)^{-s}, \quad Re(s)\gg0,
\end{align*}
where $\textup{m}(\eta)$ denotes the multiplicity of $\eta \in E_k$.
Then the logarithm of the analytic torsion of $C(N)$ 
is given by a sum of three terms:
\[
\log T(C(N), E; g)= \textup{Top}(N,E_N;g^N)+ 
\textup{Tors}(N,E_N;g^N)+ \textup{Res}(N,E_N;g^N),
\]
where the first term is a combination of Betti numbers 
\begin{equation}\label{topol}
\begin{split}
\textup{Top}(N,E_N;g^N)= &\frac{\log 2}{2}\chi(N,E_N)\\
&-\sum_{k=0}^{\frac{n}{2}-1}(-1)^kb_k
\left( \frac{1}{2}\log (n-2k+1) + \sum_{l=0}^{\frac{n}{2}-k-1}
\log (2l+1)\right).
\end{split} 
\end{equation}
The second term, which is a torsion-like term, is given by 
\begin{align}\label{torlike}
 \textup{Tors}(N,E_N;g^N)=\frac{1}{2}\sum_{k=0}^{n-1}(-1)^k \zeta_{k,N}'(0,\A_k).
\end{align}
The third term, the residual term, is a combination of residues 
of $\zeta_{k,N}(s)$: 
\begin{align}\label{stern}
\textup{Res}(N,E_N;g^N)= \sum_{k=0}^{n/2-1}\frac{(-1)^k}{4}
\sum_{r=1}^{n/2}\underset{s=2r}{\textup{Res}}\, \zeta_{k,N}(s)
\sum_{b=0}^{2r}A_{r,b}(\A_k)\frac{\Gamma'(b+r)}{\Gamma (b+r)},
\end{align} 
where the coefficients $A_{r,b}(\A_k)$ are determined by certain 
recursive formulas, associated to combinations of special functions. 
\end{thm}

\begin{remark}
The computations in \cite{Ver:ATO} are only done in the case of the trivial 
bundle $E$. However, using \eqref{flatbd}, everything in \cite{Ver:ATO} 
can be extended to the twisted case without difficulties. 
\end{remark}

Our next result identifies $\textup{Res}(N,E_N; g^N)$ with the metric anomaly of
the boundary $\{1\}\times N$ of $(C(N),g)$. The metric anomaly has been studied
in \cite{BruMa:AAF}. We recall the basic facts. 
Let $(X,g^X)$ be an oriented compact Riemannian manifold of odd dimensions, 
with boundary $\partial X$. 
Let $(E,\nabla^E,h^E)$ be flat Hermitian vector bundle over $X$.  
Let $\nabla^{TX}$ be the Levi-Civita connection  associated to $g^X$. 
Br\"uning and Ma define a secondary class $B(\nabla^{TX})\in 
\Omega^*(\partial X,E|_{\partial X})$  
(see \cite[(1.19)]{BruMa:AAF}) which  describes the metric anomaly in the
following sense. Let $g^X_i$, $i=1,2$, be two  Riemannian 
metrics  on $X$  and let $\nabla^{TX}_i$  denote the corresponding Levi-Civita 
connections. Denote by $\|\cdot \|^{RS}_{(X,E;g^X_i)}$, $i=1,2$ the Ray-Singer 
metrics on $\det H^*(X,E)$. Then 
\begin{align}\label{BM-thm}
\log \left(\frac{\|\cdot \|^{RS}_{(X,E;g^X_1)}}{\|\cdot \|^{RS}_{(X,E;g^X_2)}}
\right)=\frac{\textup{rank}(E)}{2}
\left[\int_{\partial X}B(\nabla^{TX}_2)-\int_{\partial X}B(\nabla^{TX}_1)\right].
\end{align}

Now we can state our next result.
\begin{thm}\label{main2}
Let $(C(N),g)$ be the cone over a closed oriented Riemannian 
manifold $(N^n,g^N)$ of even dimension. 
Let $(E,\nabla, h^E)$ be a flat  Hermitian vector bundle over $C(N)$
and $(E_N,\nabla_N,h_N)$ the associated flat vector bundle over $N$. 
Let $\nabla^{TC(N)}$ be the Levi-Civita connection of $g$. 
Then the residual term of Theorem \ref{BV-Theorem} equals
\begin{align*}
\textup{Res}(N,E_N; g^N)=-\frac{\textup{rank}(E)}{2}\int _{N}B(\nabla^{TC(N)}),
\end{align*}
where $N$ is identified with $\{1\}\times N$.
\end{thm}

We note that De Melo, Hartmann and Spreafico \cite{HarSpr:TAT} 
and \cite{MHS:RTA} computed the analytic torsion in the special 
case of a cone  over the sphere $S^n$ and verified that
 in even dimensions the residual term $\textup{Res}(N,E_N; g^N)$ 
equals the anomaly term by direct comparison. 

\begin{remark}
Consider the example of a cone over the $n$-dimensional flat torus $T^n$ 
and assume that $E$ is the trivial line bundle. Then from \eqref{stern} 
one finds that the integral of $B(\nabla^{TC(T^n)})$ is given in terms of residues of the 
shifted zeta function $\zeta_{k,N}(s)$. These are rational multiples of 
$(\pi)^{-n/2}\textup{Vol}(T^n)$. 
This agrees with the explicit formula in \cite[(4.43)]{BruMa:AAF}.
\end{remark}

\begin{remark}
In general,  the metric anomaly in \cite{BruMa:AAF} is expressed in terms of 
the curvature tensor,
whereas our result provides an expression in terms of residues of the shifted 
zeta function $\zeta_{k,N}(s)$. The residues can be computed in terms of the 
coefficients of the asymptotic expansion of the trace of the heat operator
of the Laplacian of $N$ on forms. It is well-known that these coefficients 
are given by the integral of universal polynomials in the  covariant 
derivatives of the curvature tensor.
\end{remark}

An immediate corollary is the identification of the Ray-Singer metric for
a bounded generalized cone with exact 
cone metric near the singularity and product metric near the 
regular boundary.

\begin{cor}\label{main3}
Let $(C(N),g)$ be a cone over a closed oriented Riemannian manifold 
$(N^n,g^N)$ of even dimension. Let $(E,\nabla, h^E)$ be a flat complex 
Hermitian vector bundle over $C(N)$ and let $(E_N,\nabla_N,h_N)$ be
the associated flat vector bundle over $N$. 
Let $g_0$ be a metric on $C(N)$ which coincides 
with $dx^2+x^2 g^N$ near the singularity at $x=0$ and with the product metric 
$dx^2\oplus g^N$ near the boundary $\{1\}\times N$. Then
\begin{equation*}
\log \left(\frac{\|\cdot \|^{RS}_{(C(N),E, g_0)}}{\|\cdot \|_{\det H^*(C(N),E)}}
     \right)= \textup{Top}(N,E_N;g^N)+\textup{Tors}(N,E_N;g^N),
\end{equation*}
where $\textup{Top}(N,E_N;g^N)$ and $\textup{Tors}(N,E_N;g^N)$ are given by 
\eqref{topol} and 
\eqref{torlike}, respectively.
\end{cor}

Corollary \ref{main3} shows that the main contribution of the 
conical singularity to the analytic torsion is given by
 $\textup{Tors}(N,E_N; g^N)$. This result will be of significance in regard of the
splitting formula.

Next we study its
asymptotic behavior  under scaling of the metric of $N$. 

\begin{prop}\label{main4}
Let $(N,g^N)$ be a closed oriented Riemannian manifold and let
$g^N_{\mu}:=\mu^{-2}g^N, \mu>0$. Then
\begin{align}
\textup{Tors}(N,E_N;g^N_{\mu})= 
O\left(\frac{\log \mu}{\mu}\right), \quad \mu \to \infty.
\end{align}
\end{prop}

\section{Variation of Analytic Torsion under Higher Order 
Deformations}
\label{section-variation}

In this section we recall some facts about the asymptotic properties of the 
heat kernel on a manifold with conical singularities. Then we will show that
the Ray-Singer metric is invariant under higher order deformations of the 
metric near the tip of the cone.

Let $(M,g^M)$ be an odd-dimensional compact Riemannian manifold with an 
isolated conical singularity with $M=\cCN\cup_N X$ and metric on $\cCN$ given
by \eqref{metric1}.
Let $(E,\nabla, h^E)$ be a flat Hermitian vector bundle over $M$
and let $(E_N,\nabla_N,h_N)$ be the flat Hermitian vector bundle over $N$
associated to $E$ via \eqref{flatbd}. 
 
Let $(\Omega^*(\mathscr{U},E),d_*)$ be the twisted de Rham complex over 
$\mathscr{U}$. 
In this section we assume essential self-adjointness of the Gau\ss-Bonnet operator,
so that the Laplacian $\Delta_k$ defined with respect to relative or absolute 
boundary conditions coincides with the Frederich's extension $\Delta_k^{\mathscr{F}}$ in all degrees
$k=0,\dots,m$. Essential self-adjointness of the Gau\ss-Bonnet operator can 
always be achieved by scaling $g^N(0)$ to $c^2g^N(0)$ with $c>0$ sufficiently 
small, see \cite{BruLes:KHT}. 

\begin{remark}
The variational result of this section holds in fact for any 
self-adjoint extension of the Laplacian satisfying Theorem 
\ref{friedrichs-blowup}, which does hold for the Frederich's extension 
$\Delta^{\mathscr{F}}$. However it suffices to impose the Witt condition $H^p(N,E_N)=0$,
$p=(m-1)/2$ instead of essential self-adjointness of the 
Gau\ss-Bonnet operator for Theorem \ref{friedrichs-blowup} to hold for 
the relative and absolute self-adjoint extension. We confine 
ourselves here to the case of essentially self-adjoint Gau\ss-Bonnet operator
for simplicity.
\end{remark}

We use the unitary rescaling transformation over 
$\mathscr{U}$, introduced in \cite[(5.2)]{BruSee:AIT}
\begin{equation}
\label{separation2}
\begin{split}
\Psi_k : C_c^{\infty}((0,1),\Omega^{k-1}(N,E_N)\oplus \Omega^k(N,E_N))\to 
\Omega_c^k(\mathscr{U},E) \\
(\phi_{k-1},\phi_k)\mapsto x^{k-1-n/2}\phi_{k-1}\wedge dx + x^{k-n/2}\phi_k, 
\end{split}
\end{equation}
where $\phi_k,\phi_{k-1}$ are identified with their pullback to $\mathscr{U}$ under the projection 
$\pi: \mathscr{U}\cong (0,1)\times N\to N$ onto the second factor, and $x\in (0,1)$ is the radial coordinate.

\begin{prop}\label{unitary} 
\textup{\cite[(5.4)]{BruSee:AIT}}
The rescaling map $\Phi_k$ extends to a unitary transformation
\begin{align*}
L^2((0,1), L^2(\Omega^{k-1}(N,E_N)\oplus \Omega^k(N,E_N), g^N(x), h_N)) \rightarrow 
L^2(\Omega^k(\mathscr{U},E), g^M|_{\mathscr{U}}, h^E|_{\mathscr{U}}). 
\end{align*}
\end{prop}
Throughout this paper we will use  the unitary transformation 
\eqref{separation2} and 
denote the heat operator under the unitary transformation again 
by $e^{-t\Delta_k}$. 
By a small abuse of notation we denote the heat operator and the 
corresponding heat kernel both by $e^{-t\Delta_k}$.

\subsection{The Heat Kernel on Manifolds with Conical Singularities}

In order to describe the asymptotic properties of the heat kernel 
$e^{-t\Delta_k}$
accurately, we consider $M$ as the interior of a compact 
manifold $\overline{M}$ with boundary.

The heat kernel of $e^{-t\Delta_k}$ is a priori a $k$-form valued distribution 
on 
a manifold with corners $M^2_h=\R^+\times \overline{M}^2$. Its singular 
structure can be conveniently described by lifting $e^{-t\Delta_k}$ to the heat 
space $\mathscr{M}^2_h$, which is a resolution 
of $M^2_h$ obtained by blowing up certain submanifolds of its boundary, 
see (\cite{Moo:HKA}). We 
begin by a blowup of 
\begin{align}
A = \{t=0\}\times (\partial \overline{M})^2 \subset M^2_h,
\end{align}
where the $\R^+-$direction is scaled parabolically, according to the parabolic 
homogeneity of the problem. The resulting space 
$[M^2_h, A]$ is defined as the union of $(M^2_h\backslash A)$ with the 
interior spherical normal bundle of $A$ in $M^2_h$.
The blowup $[M^2_h, A]$ is endowed with the unique minimal differential 
structure, including $\sqrt{t}$, smooth functions 
in the interior of $M^2_h$ and polar coordinates around $A$. The blowup 
introduces an additional boundary hypersurface, 
which we refer to as the \emph{front face} ff. We refer to \cite{Mel:TAP} for 
a careful definition and general discussion of parabolic blowups.  

The actual heat-space $\mathscr{M}^2_h$ is obtained by blowing up $[M^2_h, A]$ 
at the lift of the diagonal in $\overline{M}^2$ 
at $t=0$, again parabolically in the time $\R^+-$direction. This blowup 
introduces a boundary hypersurface, which we refer to 
as the \emph{temporal diagonal} td. The heat space comes with a canonical 
blowdown map $\beta:\mathscr{M}^2_h\to M^2_h$ 
and may be visualized as follows.

\begin{figure}[h]
\begin{center}
\begin{tikzpicture}
\draw (0,0.7) -- (0,2);
\draw (-0.7,-0.5) -- (-2,-1);
\draw (0.7,-0.5) -- (2,-1);

\draw (0,0.7) .. controls (-0.5,0.6) and (-0.7,0) .. (-0.7,-0.5);
\draw (0,0.7) .. controls (0.5,0.6) and (0.7,0) .. (0.7,-0.5);
\draw (-0.7,-0.5) .. controls (-0.5,-0.6) and (-0.4,-0.7) .. (-0.3,-0.7);
\draw (0.7,-0.5) .. controls (0.5,-0.6) and (0.4,-0.7) .. (0.3,-0.7);

\draw (-0.3,-0.7) .. controls (-0.3,-0.3) and (0.3,-0.3) .. (0.3,-0.7);
\draw (-0.3,-1.4) .. controls (-0.3,-1) and (0.3,-1) .. (0.3,-1.4);

\draw (0.3,-0.7) -- (0.3,-1.4);
\draw (-0.3,-0.7) -- (-0.3,-1.4);

\node at (1.2,0.7) {\large{rf}};
\node at (-1.2,0.7) {\large{lf}};
\node at (1.1, -1.2) {\large{tf}};
\node at (-1.1, -1.2) {\large{tf}};
\node at (0, -1.7) {\large{td}};
\node at (0,0.1) {\large{ff}};
\end{tikzpicture}
\end{center}
\label{heat-incomplete}
\caption{Heat-space Blowup $\mathscr{M}^2_h$.}
\end{figure}

We fix coordinate chart $(x,z)$ in the singular neighborhood $\mathscr{U}\subset M$. 
Two copies of this chart together with the time coordinate yield a coordinate system $(t, x, z, \wx, \wz)$ on $M^2_h$. 
The projective coordinates on $\mathscr{M}^2_h$ are then given as follows. 
Near the top corner of the front face ff, the projective coordinates are given by
\begin{align}\label{top-coord}
\rho=\sqrt{t}, \  \xi=x / \rho, \ \widetilde{\xi}=\wx / \rho, \ z, \ \widetilde{z},
\end{align}
where in these coordinates $\rho, \xi, \widetilde{\xi}$ are the defining functions of the boundary faces ff, lf and rf respectively. 
For the right hand side bottom corner of the front face projective coordinates are given by
\begin{align}\label{right-coord}
\tau=t / \wx^2, \ s=x / \wx, \ z, \ \wx, \ \widetilde{z},
\end{align}
where in these coordinates $\tau, s, \widetilde{x}$ are the defining functions of tf, rf and ff respectively. 
For the left hand side bottom corner of the front face projective coordinates are obtained by interchanging 
the roles of $x$ and $\widetilde{x}$. Projective coordinates on $\mathscr{M}^2_h$ near temporal diagonal are given by 
\begin{align}\label{d-coord}
\eta=\frac{\sqrt{t}}{\wx}, \ S =\frac{(x-\wx)}{\sqrt{t}}, \ Z =\frac{\wx (z-\wz)}{\sqrt{t}}, \  \wx, \ \widetilde{z}.
\end{align}
In these coordinates tf is the face in the limit $|(S, Z)|\to \infty$, ff and td are defined by $\widetilde{x}, \eta$, respectively. 
The blow-down map $\beta: \mathscr{M}^2_h\to M^2_h$ is in local coordinates simply the coordinate change back to 
$(t, (x,z), (\widetilde{x},\widetilde{z}))$. The heat kernel lifts to a 
polyhomogeneous conormal distribution on $\mathscr{M}^2_h$
with product type expansions at the corners of the heat space, and with 
coefficients depending smoothly on the tangential variables.
More precisely we have the following definition.

\begin{defn}\label{phg}
Let $\mathfrak{W}$ be a manifold with corners, with all boundary faces embedded, and $\{(H_i,\rho_i)\}_{i=1}^N$ an enumeration 
of its boundaries and the corresponding defining functions. For any multi-index $b= (b_1,
\ldots, b_N)\in \C^N$ we write $\rho^b = \rho_1^{b_1} \ldots \rho_N^{b_N}$.  Denote by $\mathcal{V}_b(\mathfrak{W})$ the space
of smooth vector fields on $\mathfrak{W}$ which lie
tangent to all boundary faces. A distribution $\w$ on $\mathfrak{W}$ is said 
to be conormal,
if $w\in \rho^b L^\infty(\mathfrak{W})$ for some $b\in \C^N$ and $V_1 \ldots V_\ell u \in \rho^b L^\infty(\mathfrak{W})$
for all $V_j \in \mathcal{V}_b(\mathfrak{W})$ and for every $\ell \geq 0$. An index set 
$E_i = \{(\gamma,p)\} \subset {\mathbb C} \times {\mathbb N}$ 
satisfies the following hypotheses:
\begin{enumerate}
\item $\textup{Re}(\gamma)$ accumulates only at plus infinity,
\item For each $\gamma$ there is $P_{\gamma}\in \N_0$, such 
that $(\gamma,p)\in E_i$ for $p \leq P_\gamma < \infty$,
\item If $(\gamma,p) \in E_i$, then $(\gamma+j,p') \in E_i$ for all $j \in {\mathbb N}$ and $0 \leq p' \leq p$. 
\end{enumerate}
An index family $E = (E_1, \ldots, E_N)$ is an $N$-tuple of index sets. 
Finally, we say that a conormal distribution $w$ is polyhomogeneous on 
$\mathfrak{W}$ 
with index family $E$ (denoted by 
$\w\in \mathscr{A}_{\textup{phg}}^E(\mathfrak{W})$), 
if $w$ is conormal and if in addition, near each $H_i$, 
\[
w \sim \sum_{(\gamma,p) \in E_i} a_{\gamma,p} \rho_i^{\gamma} (\log \rho_i)^p, \ 
\textup{as} \ \rho_i\to 0,
\]
with coefficients $a_{\gamma,p}$ conormal on $H_i$, polyhomogeneous with index $E_j$
at any $H_i\cap H_j$. 
\end{defn}

The asymptotic properties of the lift $\beta^*e^{-t\Delta_k}$ of the heat kernel to $\mathscr{M}^2_h$ 
in an open neighborhood of the front face have been described by Mooers in \cite{Moo:HKA}.

\begin{thm}\label{friedrichs-blowup}\textup{(\cite{Moo:HKA})}
The lift $\beta^*e^{-t\Delta_k}$ is polyhomogeneous conormal distribution on $\mathscr{M}^2_h$ 
of leading order $(-1)$ at the front face ff, leading order $(-\dim M)$ at the temporal diagonal td, and of 
leading order $1/2$ at the left and right boundary faces. $\beta^*e^{-t\Delta_k}$ vanishes to infinite order 
at the temporal face tf and does not admit logarithmic terms in its asymptotic expansion at ff and td. 
\end{thm}

The asymptotic properties of the heat kernel $e^{-t\Delta_k}$ has been discussed by various authors, 
with major contributions by Br\"uning-Seeley in \cite{BruSee:RSA}, Lesch \cite{Les:OOF} and Mooers \cite{Moo:HKA}. 
The heat calculus on manifolds with edges by Mazzeo-Vertman \cite{MazVer:ATO} contains the setup of isolated 
conical singularities as a special case.

\subsection{Analytic Torsion under Metric Variation at the Cone Tip}

We now study analytic torsion of $(M,g^M)$ under higher order deformations of the Riemannian metric $g^M$. 
Let $\mathscr{I}\subset \R$ be an open interval and let $(g^M_{\mu}, \mu \in \mathscr{I})$ be a smooth family of 
Riemannian metrics on $M$ with
\begin{align}
g^M_{\mu}\restriction \mathscr{U}=dx^2 \oplus x^2g^N(x,\mu), \ 
g^N \in C^{\infty}\left([0,1]\times \mathscr{I}, \textup{Sym}^2(T^*N)\right).
\end{align}
The dependence of $\Delta_k$ on the parameter is denoted by $\Delta_k(\mu)$. 
For $\mu\neq \mu_0$
the corresponding operators $\Delta_k(\mu), \Delta_k(\mu_0)$ act on different 
Hilbert spaces $L^2\Omega^k(M,g^M_{\mu})$ 
and $L^2\Omega^k(M,g^M_{\mu_0})$, respectively. We fix some 
$\mu_0 \in \mathscr{I}$ and employ the natural isometry 
$$
T_{\mu}:L^2\Omega^k(M,g^M_{\mu})\to L^2\Omega^k(M,g^M_{\mu_0}),
$$
to define self-adjoint operators in the fixed Hilbert space  $L^2\Omega^k(M,g^M_{\mu_0})$
$$
H_k(\mu):=T_{\mu} \circ \Delta_k(\mu) \circ T_{\mu}^{-1}.
$$ 
As a simple consequence of the semi-group property we have for any $\mu, \mu_0\in \mathscr{I}$
\begin{align*}
\frac{e^{-tH_k(\mu)}-e^{-tH_k(\mu_0)}}{\mu-\mu_0}&=
\int_0^t\frac{\partial}{\partial s}\left(\frac{e^{-(t-s)H_k(\mu_0)}e^{-sH_k(\mu)}}{\mu-\mu_0}\right)ds \\
&=\int_0^t e^{-(t-s)H_k(\mu_0)}\left(\frac{H_k(\mu_0)-H_k(\mu)}{\mu-\mu_0}\right)e^{-sH_k(\mu)}ds.
\end{align*}
Taking the limit $\mu \to \mu_0$ this leads to 
\begin{align*}
\left. \frac{\partial}{\partial \mu}\textup{Tr}\left(e^{-tH_k(\mu)}\right)\right|_{\mu=\mu_0}=
&-\int_0^t\textup{Tr}\left(e^{-(t-s)H_k(\mu_0)}\left(\dot{H}_k(\mu_0)\right) e^{-sH_k(\mu_0)}\right)\\
=& -t \cdot \textup{Tr}\left(\left(\dot{H}_k(\mu_0)\right)
 e^{-tH_k(\mu_0)}\right),
\end{align*}
where the upper-dot denotes the derivative with respect to $\mu$. 
Evaluating $\dot{H}_k(\mu)$ explicitly in terms of the isometry $T_{\mu}$, we find
\begin{equation}\label{trace-game}
\begin{split}
\left. \frac{\partial}{\partial \mu}\textup{Tr}\left(e^{-tH_k(\mu)}\right)\right|_{\mu=\mu_0}=
& -t \cdot \textup{Tr}\left(\dot{T}_{\mu_0}\circ \Delta_k(\mu_0)\circ e^{-t\Delta_k(\mu_0)}\circ T^{-1}_{\mu_0}\right) \\
& -t \cdot \textup{Tr}\left(\Delta_k(\mu_0)\circ \dot{T}^{-1}_{\mu_0}T_{\mu_0}\circ e^{-t\Delta_k(\mu_0)}\right) \\
& -t \cdot \textup{Tr}\left(\dot{\Delta}_k(\mu_0)\circ e^{-t\Delta_k(\mu_0)}\right).
\end{split}
\end{equation}
For the second summand in \eqref{trace-game} we employ commutativity of bounded operators under the trace and find
\begin{align*}
&\textup{Tr}\left(\Delta_k(\mu_0)\circ \dot{T}^{-1}_{\mu_0}T_{\mu_0}\circ e^{-t\Delta_k(\mu_0)}\right) 
= \textup{Tr}\left(e^{-t/2\Delta_k(\mu_0)}\Delta_k(\mu_0)\circ 
\dot{T}^{-1}_{\mu_0}T_{\mu_0}\circ e^{-t/2\Delta_k(\mu_0)}\right) \\
= \, &\textup{Tr}\left(\dot{T}^{-1}_{\mu_0}T_{\mu_0}\circ \Delta_k(\mu_0) 
\circ e^{-t\Delta_k(\mu_0)}\right) 
= - \textup{Tr}\left(T^{-1}_{\mu_0}\dot{T}_{\mu_0}\circ \Delta_k(\mu_0) 
\circ e^{-t\Delta_k(\mu_0)}\right). 
\end{align*}
Consequently the first and the second summands in \eqref{trace-game} cancel 
and we get
\begin{align}
\left. \frac{\partial}{\partial \mu}\textup{Tr}
\left(e^{-t\Delta_k(\mu)}\right)\right|_{\mu=\mu_0}=
 -t \cdot \textup{Tr}\left(\dot{\Delta}_k(\mu_0)\circ e^{-t\Delta_k(\mu_0)}\right).
\end{align}
Let $P_k(\mu), \mu \in \mathscr{I}$ denote the orthogonal projection onto the 
kernel of $\Delta_k(\mu)$. 
Let $*_{\mu}$ denote the Hodge-star operator associated to 
 $g^M_{\mu}, \mu \in \mathscr{I}$,  and put 
$\A^k_{\mu}:=*_{\mu}^{-1}\dot{*}_{\mu},$ with the upper-index $k$ denoting the 
restriction to forms of degree $k$.
Repeating the arguments of Ray-Singer in \cite[p. 152-153]{RaySin:RTA} we 
get
\begin{equation}\begin{split}\label{RS-identity}
&\frac{\partial }{\partial \mu}\sum_{k=0}^{\dim M}(-1)^k\cdot k \cdot 
\textup{Tr}\left(e^{-t\Delta_k(\mu)}-P_k(\mu)\right) \\
=t\, &\frac{\partial }{\partial t}\sum_{k=0}^{\dim M}(-1)^k
\textup{Tr}\left(\A_{\mu}^k \left(e^{-t\Delta_k(\mu)}-P_k(\mu)\right)\right)dt.
\end{split}\end{equation}
Put
\begin{align*}
f(\mu,s):= \frac{1}{2}\sum_{k=0}^{\dim M}(-1)^k\cdot k \cdot 
\frac{1}{\Gamma(s)}\int_0^{\infty}t^{s-1}
\textup{Tr}\left(e^{-t\Delta_k(\mu)}-P_k(\mu)\right)dt.
\end{align*} 
Then by definition we have
\begin{align}\label{t-f}
\log T(M, E, g^M_{\mu})=\left.\frac{\partial}{\partial s}\right|_{s=0}f(\mu, s).
\end{align}
Since the heat trace is exponentially decaying, we can differentiate 
$f(\mu,s)$ with respect to $\mu\in I$ in the half-plane $\textup{Re}(s)\gg 0$, 
by differentiating under the integral sign. Choosing $\textup{Re}(s)$
large enough so that boundary terms vanish and using \eqref{RS-identity} we
get
\begin{equation}\begin{split}\label{f}
\frac{\partial }{\partial\mu}f(\mu, s)&=\frac{1}{2}
\sum_{k=0}^{\dim M}(-1)^{k}\frac{1}{\Gamma(s)}
\int_0^{\infty}t^{s}\frac{d}{dt}\textup{Tr}\left(\A_{\mu}^k 
\left(e^{-t\Delta_k(\mu)}-P_k(\mu)\right)\right)dt\\
&=\frac{1}{2}\, s\sum_{k=0}^{\dim M}(-1)^{k+1}\frac{1}{\Gamma(s)}
\int_0^{\infty}t^{s-1}\textup{Tr}
\left(\A_{\mu}^k \left(e^{-t\Delta_k(\mu)}-P_k(\mu)\right)\right)dt.
\end{split}\end{equation}
Let us henceforth assume that the metric family $g^M_{\mu}$ differs only in 
its higher order terms on 
the singular neighborhood $\mathscr{U}$. More precisely, assume that
 for all $\mu\in \mathscr{I}$
\begin{align}\label{g-first-order}
g^N(0,\mu)\equiv g^N(0).
\end{align}
Under the transformation \eqref{separation2}, 
$\A^k_{\mu}|_{\mathscr{U}}\in C^{\infty}\left([0,1], 
\textup{End}\left(\Omega^{k-1}(N,E_N)\oplus \Omega^k(N,E_N)\right)\right)$ 
for any $\mu\in\mathscr{I}$, and by \eqref{g-first-order}, it follows that
\[
\A^k_{\mu}(x)=O(x),\quad x\to 0.
\]
Let $\pi:\R^+\times \overline{M}^2\to \overline{M}$ be the projection onto a 
copy of $\overline{M}$. 
Then, employing the projective coordinates on the heat space $\mathscr{M}^2_h$, 
the lift $(\pi \circ \beta)^*\A^k_{\mu}$ is a polyhomogeneous function on 
$\mathscr{M}^2_h$ with 
\begin{align}
(\pi \circ \beta)^*\A^k_{\mu} =\rho_{\textup{ff}}\rho_{\textup{rf}} \cdot 
\mathscr{B}
\end{align}
where $\rho_{\textup{ff}}$ and  $\rho_{\textup{rf}}$ are boundary defining 
functions 
of the front and the right boundary face, respectively; $\mathscr{B}$ is a 
polyhomogeneous bounded function on 
$\mathscr{M}^2_h$. If we compose $(\pi \circ \beta)^*\A^k_{\mu}$ 
pointwise with 
$\beta^*e^{-t\Delta_k(\mu)}$ and use Theorem \ref{friedrichs-blowup}, we obtain 
the following  lemma.

\begin{lemma}\label{aH}
The lift $\beta^*\left(\A^k_{\mu}e^{-t\Delta_k(\mu)}\right)$ is a polyhomogeneous 
conormal distribution 
on $\mathscr{M}^2_h$ of zero leading order at the front face ff, leading order 
$(-\dim M)$ at the temporal diagonal td, and of 
leading order $1/2$ and $3/2$ at the left and right boundary faces, 
respectively. $\beta^*\left(\A^k_{\mu}e^{-t\Delta_k(\mu)}\right)$ 
vanishes to infinite order at the temporal face tf and does not admit 
logarithmic terms in its asymptotic expansion at ff and td. 
\end{lemma}
For any fixed $t>0$
\begin{align}\label{point-trace}
\A^k_{\mu}e^{-t\Delta_k(\mu)}\in \wedge^k T^*M \boxtimes \wedge^k T^*M \cong (\wedge^k T^*M)^* \boxtimes \wedge^k T^*M.
\end{align}
The right hand side of \eqref{point-trace} admits on the diagonal an invariantly defined pointwise trace $\textup{tr}$, and 
$\textup{tr}\left(\A^k_{\mu}e^{-t\Delta_k(\mu)}\right)$ defines a polyhomogeneous conormal distribution on the lift $\mathscr{D}$ 
of the diagonal in $\overline{M}^2$ to the heat space $\mathscr{M}^2_h$. The lifted diagonal $\mathscr{D}$ is a resolution 
of $\R^+\times \overline{M}$, obtained by blowing up $\{0\}\times \partial \overline{M}$, with the $\R^+-$direction 
scaled parabolically; $\mathscr{D}$ may be visualized as in Figure \ref{lifted-diagonal}. 

\begin{figure}[h]
\begin{center}
\begin{tikzpicture}
\draw (0,0) -- (0,2);

\draw (0,0) .. controls (0.5,-0.5) and (1,-1) .. (1,-2);
\draw (0,0) .. controls (0.5,-0.1) and (1.2,-0.5) .. (1.6,-1.5);
\draw (0,0) .. controls (0.5,0) and (1.5,0) .. (2,-1.3);

\draw (1,-2) -- (3,-3);
\draw (2,-1.3) -- (4,-1.3);

\draw (1.7,-1.5) .. controls (1.5,-1.5) and (1.4,-1.6) .. (1.4,-1.8);
\draw (1.7,-1.5) -- (3.6,-1.9);
\draw (1.4,-1.8) -- (3.4,-2.3);

\draw (3.6,-1.9) .. controls (3.5,-1.9) and (3.4,-2.2) .. (3.4,-2.3);
\draw (3.6,-1.9) .. controls (3.7,-1.9) and (3.8,-2) .. (3.8,-2.1);

\draw (3.6,-1.9) -- (3.6,0.1);
\draw (3.6,0.1) -- (0,0.7);

\draw (1,-2) -- (1.4,-1.8);
\draw (2,-1.3) -- (1.8,-1.5);

\node at (4.5,-2) {\large{td}};
\node at (-0.5,1) {\large{lf}};
\node at (0.1,-1) {\large{ff}};
\node at (4,1) {$\mathscr{D}\subset \mathscr{M}^2_h$};

\end{tikzpicture}
\end{center}
\label{lifted-diagonal}
\caption{The lifted diagonal $\mathscr{D}$ of the heat-space $\mathscr{M}^2_h$.}
\end{figure}
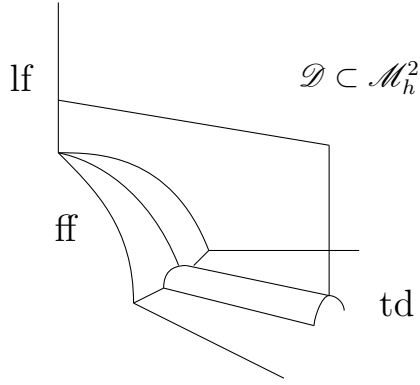
By a small abuse of notation, 
we refer to the boundary faces of the lifted diagonal $\mathscr{D}\subset \mathscr{M}^2_h$ 
again as temporal diagonal td, front face ff and left face lf, respectively. The projective coordinates \eqref{top-coord} and \eqref{d-coord} on $\mathscr{M}^2_h$ 
yield projective coordinates on $\mathscr{D}$.  Near the top corner of $\mathscr{D}$ the projective 
coordinates are given by
\begin{align}
\rho=\sqrt{t}, \  \xi=x/\rho, \ z, \label{D-coord1}
\end{align} 
where in these coordinates $\rho$ and $\xi$ are the defining functions of the front and the left 
boundary face, respectively. Near the lower corner of $\mathscr{D}$ the projective 
coordinates are given by
\begin{align}
\eta=\sqrt{t}/x, \ x, \ z, \label{D-coord2}
\end{align}
where in these coordinates $x$ and $\eta$ are the defining functions of the front and the 
temporal diagonal boundary face, respectively. $\mathscr{D}$ comes with a canonical blowdown map 
$\beta_{\mathscr{D}}:\mathscr{D}\to \R^+\times \overline{M}$, which in local coordinates is simply the coordinate change 
back to $(t,x,z)$. 

The identification of the exterior bundle $\wedge^k T^*M$ with its dual $ (\wedge^k T^*M)^*$ via $g^M$ in fact does 
not induce any shift in the polyhomogeneous expansion of $\A^k_{\mu}e^{-t\Delta_k(\mu)}$, since we consider the setup 
after unitary transformation in Proposition \ref{unitary} to a product $L^2$-space. By Lemma
\ref{aH} we arrive at the following

\begin{prop}\label{d-expansion1}
The lift $\beta^*_{\mathscr{D}}\textup{tr}\left(\A^k_{\mu}e^{-t\Delta_k(\mu)}\right)$ is a polyhomogeneous conormal 
distribution on $\mathscr{D}$ of zero leading order at the front face ff, leading order $(-\dim M)$ at the temporal diagonal td, and of 
leading order $2$ at the left boundary face lf. Moreover the asymptotic expansion at ff does not admit logarithmic terms and 
the expansion at td ($\rho_{\textup{td}}$ denotes the defining function of td) is of the form
\begin{align}
\beta^*_{\mathscr{D}}\textup{tr}\left(\A^k_{\mu}e^{-t\Delta_k(\mu)}\right)
\sim \rho_{\textup{td}}^{-\dim M}\left(\sum\limits_{k=0}^{\infty}a_{2k}\rho_{\textup{td}}^{2k}\right), 
\ \textup{as} \ \rho_{\textup{td}}\to 0.
\end{align}
\end{prop}

\begin{proof}
The polyhomogeneity and the leading orders in the asymptotic expansions of 
$\beta^*_{\mathscr{D}}\textup{tr}\left(\A^k_{\mu}e^{-t\Delta_k(\mu)}\right)$ 
at the boundary faces of $\mathscr{D}$ follow by Lemma \ref{aH}. Moreover, at the 
temporal diagonal Lemma \ref{aH} implies
\begin{align}\label{lifted-ptw-trace}
\beta^*_{\mathscr{D}}\textup{tr}\left(\A^k_{\mu}e^{-t\Delta_k(\mu)}\right)
\sim \rho_{\textup{td}}^{-\dim M}\left(\sum\limits_{k=0}^{\infty}a_{k}\rho_{\textup{td}}^{k}\right), 
\ \textup{as} \ \rho_{\textup{td}}\to 0.
\end{align}
We have the following classical pointwise trace expansion on Riemannian manifolds
\begin{align}\label{ptw-trace}
\textup{tr}\left(\A^k_{\mu}e^{-t\Delta_k(\mu)}\right) \sim 
(\sqrt{t})^{-\dim M} \sum_{k=0}^{\infty} \A_{2k}(\sqrt{t})^{2k}, \ 
\textup{as} \ t\to 0.
\end{align}
Employing the projective coordinates \eqref{D-coord1} and \eqref{D-coord2} we find 
$\beta^*_{\mathscr{D}}t=\rho_{\textup{ff}}^2\rho_{\textup{td}}^2$. Hence, lifting \eqref{ptw-trace} 
to $\mathscr{D}$, we deduce for \eqref{lifted-ptw-trace} that $a_k\equiv 0$ for $k$ odd.
\end{proof}

Suppose $\mu$ is a density on $\mathscr{D}$, which is smooth up to all 
boundary faces and nowhere vanishing. A smooth $b$-density $\mu_b$ is by 
definition any density of the form 
$(\rho_{\textup{ff}}\rho_{\textup{lf}}\rho_{\textup{td}})^{-1}\mu$.
Let $\pi_{\mathscr{D}}:\R^+\times \overline{M}\to \R^+$ is the projection onto the first factor. 
We can push forward densities in a natural way and a straightforward computation 
in local coordinates shows
\begin{align}\label{push}
(\pi_{\mathscr{D}}\circ \beta_{\mathscr{D}})_* 
\beta^*_{\mathscr{D}}\left( x\cdot \textup{tr}\left(\A^k_{\mu}e^{-t\Delta_k(\mu)}\right)\right)=
\textup{Tr} \left( \A^k_{\mu}e^{-t\Delta_k(\mu)} \right) f(t) \frac{dt}{t},
\end{align}
where $f$ is some smooth function on $\R^+\cup \{0\}$.
We derive the short-time asymptotics of $\textup{Tr} \left( \A^k_{\mu}e^{-t\Delta_k(\mu)} \right)$ 
employing the Push-forward theorem of Melrose \cite{Mel:COC}.
We also refer to \cite{Mel:COC} for the exact definition of b-fibrations and b-densities.

\begin{thm}\label{melrose}\textup{(Melrose, \cite{Mel:COC})}
Let $\mathfrak{W}_1, \mathfrak{W}_2$ be manifolds with corners, and 
let $\{(H_i,\rho_i)\}_{i\in I}$ and $\{(K_j,\delta_j)\}_{j\in J}$ be the respective indexing 
of the boundaries and the corresponding defining functions. Let $f:\mathfrak{W}_1\to \mathfrak{W}_2$ 
be a b-fibration with 
\begin{align}
f^*\delta_j=h \cdot \prod_{i\in I} \rho_j^{e(j,i)},
\end{align}
where $h\in C^{\infty}(\mathfrak{W}_1)$ is strictly positive and the exponents $e(j,i)\in \N_0$ are 
such that for any $j\in J$ there is at most one $i\in I$ with $e(j,i)\neq 0$. If 
$\w\in \mathscr{A}_{\textup{phg}}^E(\mathfrak{W}_1)$ for an index set $E=(E_i)_{i\in I}$ such that $f_*\w$
exists, then $f_*\w\in \mathscr{A}_{\textup{phg}}^{f_*E}(\mathfrak{W}_2)$, with 
the index family $f_*E=(f_*E)_{j\in J}$, where the index sets $(f_*E)_j=\{(\gamma, p)\}\in \C\times \N_0$ 
are comprised of all pairs $(\gamma, p)$ such that 
\begin{enumerate}
\item $(\gamma,p)=(s/e(j,i), l)$, for $(s,l)\in E_i$ with $e(j,i)\neq 0$,
\item $(\gamma,p)=(s/e(j,i), l+l'+1)$, for $(s,l)\in E_{i}$ and $(s',l')\in E_{i'}$ \\
with $e(j,i), e(j,i')\neq 0$ and $s/e(j,i)=s'/e(j,i')$.
\end{enumerate}
\end{thm}

In the projective coordinates \eqref{D-coord1} and \eqref{D-coord2} we 
directly verify that the b-fibration $(\pi_{\mathscr{D}}\circ \beta_{\mathscr{D}})$ satisfies
$(\pi_{\mathscr{D}}\circ \beta_{\mathscr{D}})^*t=\rho_{\textup{ff}}^2\rho_{\textup{td}}^2$. Hence, 
by Push-forward Theorem \ref{melrose} we infer from \eqref{push} and Proposition \ref{d-expansion1}
\begin{align}
\textup{Tr}\left(\A^k_{\mu}e^{-t\Delta_k(\mu)}\right) \sim 
\sum_{j=0}^{\infty}A_jt^{j-\frac{m}{2}}+\sum_{j=1}^{\infty}C_jt^{\frac{j}{2}}+
\sum_{j=1}^{\infty}G_jt^{\frac{j}{2}}\log t, \ \textup{as} \ t\to 0.
\end{align}
Consequently
\begin{align}
\underset{s=0}{\textup{Res}}\,\int_0^{\infty}t^{s-1}\textup{Tr}\left(\A_{\mu}^k \left(e^{-t\Delta_k(\mu)}-P_k(\mu)\right)\right)dt=-\textup{Tr}(\A_{\mu}^kP_k(\mu)).
\end{align}
In view of the additional $s$-factor in \eqref{f}, we find by \eqref{t-f}
\begin{equation}\begin{split}
\frac{d}{d\mu}\log T(M, E, g^M_{\mu})&=
\frac{1}{2}\sum_{k=0}^{\dim M}(-1)^{k}\textup{Tr}\left(\A_{\mu}^k P_k(\mu)\right)\\
&=\frac{d}{d\mu}\log \|\cdot \|^{-1}_{\det H^*(M,E), g^M_{\mu}}.
\end{split}
\end{equation}
We have proved the following main result (Theorem \ref{main1}) of this section.

\begin{thm}\label{higher-order}
Let $(M,g^M_{\mu}), \mu \in \mathscr{I}$ be a compact odd-dimensional  
Riemannian manifold with an isolated conical singularity 
and a singular neighborhood $\mathscr{U}\cong (0,1)\times N$, endowed with a 
family of conical metrics, parametrized over an open $\mathscr{I}\subset \R$, 
such that 
\begin{align}
g^M_{\mu}\restriction \mathscr{U}=dx^2 \oplus x^2g^N(x,\mu), \ 
g^N \in C^{\infty}\left([0,1]\times \mathscr{I}, \textup{Sym}^2(T^*N)\right).
\end{align}
Let $(E,\nabla, h^E)$ be a flat Hermitian vector bundle over $M$ and let 
$(E_N,\nabla_N,h_N)$ be the flat vector bundle  over $N$, which is associated
to $E|_{\mathscr{U}}$. 
Assume that the Gau\ss-Bonnet operator is essential self-adjoint.
Then the  Ray-Singer norm 
$\|\cdot \|^{RS}_{(M,E;g^M_{\mu})}$ is independent of $\mu$, i.e. 
\begin{align}
\frac{d}{d\mu} \|\cdot \|^{RS}_{(M,E;g^M_{\mu})}=0.
\end{align}
\end{thm}

\section{Scaling Invariance of the Metric Anomaly for Analytic Torsion}
\label{section-scaling}

To begin with we need to introduce some notation. Let $\mathscr{A}$,
$\mathscr{B}$ be two $\Z_2$-graded algebras with identity and let 
$\mathscr{A}\widehat\otimes\mathscr{B}$ denote their $\Z_2$-graded tensor
product. Identify $\mathscr{A}$ with $\mathscr{A}\widehat\otimes I$ and let
$\widehat{\mathscr{B}}:=I\widehat\otimes\mathscr{B}$. Write $\wedge:=\widehat
\otimes$ so that $\mathscr{A}\widehat\otimes\mathscr{B}=\mathscr{A}\wedge
\mathscr{B}$. 

Let $(X,g^X)$ be a compact oriented Riemannian manifold of odd dimension $m$ 
with boundary $\partial X$. In 
\cite[(1.15)]{BruMa:AAF} elements
\begin{equation}
\dot{R}^{T\partial X}\in\Lambda T^*\partial X \, \widehat{\otimes} 
\widehat{\Lambda T^*\partial X}\quad \textup{and}\quad \dot{S}\in
\Lambda T^*\partial X\widehat\otimes \widehat{\Lambda T^*\partial X}
\end{equation}
have been introduced. They are defined as follows. 
Let $R^{T\partial X}$ be the curvature tensor of $(\partial X, g^{\partial X})$ and 
denote by $\{e_k\}_{k=1}^m$ 
a local orthonormal frame field on $(TX,g^{\partial X})$. We assume  
that near 
the boundary $e_m$ is the inward-pointing unit normal vector at every boundary 
point. Let 
$\{e^*_k\}_{k=1}^m$ be the dual orthonormal frame field of $T^*X$ and let 
$\widehat{e^*_k}$ be the canonical identification with an element
of $\widehat{\Lambda T^*X}$. Let $j\colon Y
\hookrightarrow X$ be the canonical embedding. Then 
\begin{equation}\label{RS3}
\dot{R}^{T\partial X}:= \frac{1}{2}\sum_{1\le k,j\le m-1}\langle e_k, R^{T\partial X} 
e_j\rangle \widehat{e^*_k} \wedge \widehat{e^*_j}\quad\textup{and}\quad
\dot{S}:=\frac{1}{2}j^*\nabla^{TX}\widehat{e_m^*}.
\end{equation}
$\dot{R}^{T\partial X}$ is defined in terms of the curvature of the Levi-Civita 
connection, 
induced by $g^{\partial X}$. $\dot{S}$ measures the deviation from a metric 
product structure near the boundary. $\dot{R}^{T\partial X}$ and $\dot{S}^2$ are 
both homogeneous of degree two. Let
\begin{align}
\int^{B_{\partial X}}:\Lambda T^*\partial X \, \widehat{\otimes} \, 
\widehat{\Lambda T^*\partial X} \to \Lambda T^*\partial X,
\end{align}
be the Berezin integral (see (\cite[Section 1.1]{BruMa:AAF}). It is 
non-trivial only on elements which are homogeneous of degree $m$. Then the 
secondary class $B(\nabla^{TX})$, 
introduced in  \cite[(1.17)]{BruMa:AAF} is defined as follows
\begin{align}\label{RS2}
B(\nabla^{TX})=\int^{B_{\partial X}} \exp \left(-\frac{1}{2}\dot{R}^{T\partial X}\right) 
\sum_{k=1}^{\infty} \frac{(-\dot{S}^2)^k}{4k\Gamma(k+1)}.
\end{align}
By(\cite[Theorem 0.1]{BruMa:AAF} the metric anomaly of analytic torsion is 
expressed in terms of $B(\nabla^{TX})$.
For our purpose we will need that $B(\nabla^{TX})$ is scaling invariant. This
is verified in the next proposition.
\begin{prop}\label{Scaling}
Let $g^X_1$ be a Riemannian metrics on a compact manifold $(X,\partial X)$ 
with boundary. Let $s>0$ and put $g^X_2=s\cdot g^X_1$.
Let $\nabla^{TX}_1, \nabla^{TX}_2$ denote the corresponding 
Levi-Civita connections. Then 
\begin{align}
 B(\nabla^{TX}_1)=B(\nabla^{TX}_2).
\end{align}
\end{prop}
\begin{proof}
Denote by $(\dot{R}^{T\partial X}_i$ and $\dot{S}_i$ the elements \eqref{RS3} with
respect to the metric $g^{TX}_i$, $i=1,2$. It follows from \eqref{RS3} that
\begin{equation}\label{scaling} 
\dot{R}^{\partial T}(g^X_2)=s\dot{R}^{\partial T}(g^X_1), 
\quad \dot{S}(g^X_2)=\sqrt{s}\dot{S}(g^X_1).
\end{equation}
The Berezin integral also depends on the metric. We denote it by 
$\int^{B_{\partial X}}_j,j=1,2$. 
Using the definition of the Berezin integral \cite[(1.1)]{BruMa:AAF} we obtain
\begin{align}
\int^{B_{\partial X}}_2=s^{-\dim \partial X /2} \int^{B_{\partial X}}_1.
\end{align}
Since $\dot{R}^{T\partial X}$ and $\dot{S}^2$ are both of degree two and the 
Berezin integral is non-trivial only on terms homogeneous of degree 
$\dim \partial X$, it follows from \eqref{scaling} that
\begin{align*}
B(\nabla^{TX}_2)&=\int^{B_{\partial X}}_2 \exp \left(-\frac{1}{2}
\dot{R}^{T\partial X}_2\right) 
\sum_{k=1}^{\infty} \frac{(-\dot{S}_2^2)^k}{4k\Gamma(k+1)}\\
&= \int^{B_{\partial X}}_1 \exp \left(-\frac{1}{2}\dot{R}^{T\partial X}_2\right)
\sum_{k=1}^{\infty} \frac{(-\dot{S}_2^2)^k}{4k\Gamma(k+1)} \cdot 
s^{-\dim \partial X /2}\\
&= \int^{B_{\partial X}}_1 \exp \left(-\frac{1}{2}\dot{R}^{T\partial X}_1\right) 
\sum_{k=1}^{\infty} \frac{(-\dot{S}_1^2)^k}{4k\Gamma(k+1)}=B(\nabla^{TX}_1).
\end{align*} 
\end{proof}

In the special case of a cone with a flat cross-section $N$ the scaling 
invariance of the metric anomaly follows also  from 
\cite[(4.43)]{BruMa:AAF}.

\section{Analytic Torsion of the Truncated Cone}
\label{section-truncated}

\subsection{Decomposition of the de Rham Complex of the Truncated Cone}\label{decomposition}

The Analytic torsion of an exact cone can be evaluated by decomposing 
the de Rham complex into short subcomplexes. This decomposition 
 has been used in the computation of analytic torsion of cones 
in \cite{Ver:ATO}. The same 
decomposition and its symmetry are in fact also valid for truncated cones. 
Let $I$ be either the open interval $(0,1)$ or the closed interval 
$[\varepsilon,1]$, $\varepsilon>0$. Let $(N^n,g^N)$ be an even-dimensional 
closed Riemannian manifold and let $C_I(N):=I\times N$, equipped with the 
metric
\[
g = dx^2 \oplus x^2g^N, x\in I.
\]
Let $(E,\nabla, h^E)$ be a flat complex Hermitian vector bundle over $C_I(N)$ 
and $(E_N,\nabla_N,h_N)$ its restriction to the cross-section $N$. 
Let $(\Omega^*_c(C_I(N),E),d_*)$ be the associated twisted de Rham complex, 
where $\Omega^*_c(C_I(N),E)$ are the $E$-valued differential forms on 
$C_I(N)$ with compact support and $d_*$ denotes the differential.  
As in Section \ref{section-variation} we consider the map given by 
separation of variables
\begin{align}\label{separation}
\Psi_k : C^{\infty}_c(I,\Omega^{k-1}(N,E_N)\oplus \Omega^k(N,E_N))\to 
\Omega_c^k(C_I(N),E) \\
(\w_{k-1},\w_k)\mapsto x^{k-1-n/2}\w_{k-1}\wedge dx + x^{k-n/2}\w_k, \nonumber
\end{align}
where $\w_k,\w_{k-1}$ are identified with their pullback to $C_I(N)$ under the 
projection $\pi: I\times N\to N$ onto the second factor, and $x\in I\subset 
\R^+$. 
As in Proposition \ref{unitary}, the map $\Psi_k$ extends to an isometry 
\begin{align*}
\Psi_k: L^2(I, L^2(\Omega^{k-1}(N,E_N)\oplus \Omega^k(N,E_N), g^N, h_N), dx)\to 
L^2(\Omega_c^k(C_I(N),E), g,h^E).
\end{align*}
Let
\[
c_k:=(-1)^k\left(k-\frac{n}{2}\right),\quad k=0,...,m.
\]
As in \cite[(5.5)]{BruSee:AIT} we obtain
\begin{equation}\label{derivative} 
\begin{split}
\Psi_{k+1}^{-1} d_k \Psi_k= \left( \begin{array}{cc}0&(-1)^k\partial_x\\
0&0\end{array}\right)+\frac{1}{x}
\left( \begin{array}{cc}d_{k-1,N}&c_k\\0&d_{k,N}\end{array}\right), \\
\Psi_k^{-1} d_k^* \Psi_{k+1}= \left( \begin{array}{cc}0&0\\
(-1)^{k+1}\partial_x&0\end{array}\right)+
\frac{1}{x}\left( \begin{array}{cc}d_{k-1,N}^*&0\\c_k&d_{k,N}^*\end{array}\right), 
\end{split}
\end{equation}
where $d_{k,N}$ is the exterior differential on $\Omega^k(N,E_N)$.
Following a suggestion of M. Lesch,
we decompose the de Rham complex $(\Omega^*_c(C_I(N),E),d_*)$ into a direct sum 
of subcomplexes of two types. 
The first type of the subcomplexes is given as follows. Let 
$\psi \in \Omega^k(N,E_N)$ be a coclosed 
eigenform of the Laplacian $\Delta_{k,N}$ on $\Omega^k(N,E_N)$ with eigenvalue 
$\eta >0$. 
We consider the following four associated pairs
\begin{equation}\label{xi1}
\begin{split}
&\xi_1:=(0,\psi)\in \Omega^{k-1}(N,E_N)\oplus \Omega^{k}(N,E_N), \\ 
&\xi_2:=(\psi,0)\in 
\Omega^{k}(N,E_N)\oplus \Omega^{k+1}(N,E_N), \\
&\xi_3:=(0,d_N\psi/\sqrt{\eta})\in \Omega^{k}(N,E_N)\oplus \Omega^{k+1}(N,E_N), \\ 
&\xi_4:=(d_N\psi/\sqrt{\eta},0)\in \Omega^{k+1}(N,E_N)\oplus \Omega^{k+2}(N,E_N).
\end{split}
\end{equation}
Denote by $\langle \xi_1,\xi_2,\xi_3,\xi_4\rangle$ the span of 
$\xi_1,\dots,\xi_4$. Then
$C^{\infty}_c(I,\langle \xi_1,\xi_2,\xi_3,\xi_4\rangle)$ is invariant under 
$d,d^*$ and we obtain a subcomplex
\begin{align}\label{complex-1}
0 \rightarrow C_c^{\infty}(I,\left< \xi_1\right>) \xrightarrow{d_0} 
C_c^{\infty}(I,\left<\xi_2,\xi_3\right>) \xrightarrow{d_1}
C_c^{\infty}(I,\left<\xi_4\right>)
\rightarrow 0,
\end{align}
where $d_0,d_1$ take the following form with respect to the chosen 
basis:
\begin{align*}
d_0=\left(\begin{array}{c}(-1)^k\partial_x+\frac{c_k}{x}\\ 
x^{-1}\sqrt{\eta}\end{array}\right), 
\quad d_1=\left(x^{-1}\sqrt{\eta}, \ (-1)^{k+1}\partial_x
+\frac{c_{k+1}}{x}\right).
\end{align*}
The associated Laplacians are of the following form
\begin{align}\label{psi-laplacians}
\Delta_{0,\eta}:= d_0^*d_0=-\partial_x^2+
\frac{1}{x^2}\left[\eta +
\left(k+\frac{1}{2}-\frac{n}{2}\right)^2
-\frac{1}{4}\right]=d_1d_1^*=:
\Delta_{1,\eta}.
\end{align}
with respect to the identification of $\phi =f\cdot 
\xi_{i}\in C^{\infty}_c(I,\langle \xi_{i}\rangle)$, $i=1,\dots,4$, 
with its scalar part $f \in C^{\infty}_c(I)$.
The subcomplexes \eqref{complex-1} always come in pairs on oriented cones. The 
twin subcomplex is 
constructed by considering $\phi:=*_N\psi \in \Omega^{n-k}(N,E_N)$. Then 
$d^*_N\phi/\sqrt{\eta}$ is 
again a coclosed eigenform of the Laplacian on $\Omega^{n-k-1}(N,E_N)$ with 
eigenvalue $\eta$, and we put 
\begin{equation}\label{xi2}
\begin{split}
&\widetilde{\xi_1}:=(0,d^*_N\phi/\sqrt{\eta})\in \Omega^{n-k-2}(N,E_N)\oplus 
\Omega^{n-k-1}(N,E_N), 
\\ &\widetilde{\xi_2}:=(d^*_N\phi/\sqrt{\eta},0)\in \Omega^{n-k-1}(N,E_N)\oplus 
\Omega^{n-k}(N,E_N), \\
&\widetilde{\xi_3}:=(0,\phi)\in \Omega^{n-k-1}(N,E_N)\oplus \Omega^{n-k}(N,E_N), 
\\ &\widetilde{\xi_4}
:=(\phi,0)\in \Omega^{n-k}(N,E_N)\oplus \Omega^{n-k+1}(N,E_N).
\end{split}
\end{equation}
Denote by $\langle \widetilde{\xi_1},\widetilde{\xi_2},\widetilde{\xi_3},
\widetilde{\xi_4}\rangle$ 
the vector space, spanned by these vectors.
$C_c^{\infty}(I,\langle \widetilde{\xi_1},\widetilde{\xi_2},\widetilde{\xi_3},
\widetilde{\xi_4}\rangle)$ 
is invariant under the action of $d,d^*$ and we obtain a subcomplex
\begin{align}\label{complex-2}
0 \rightarrow C_c^{\infty}(I,\langle \widetilde{\xi_1}\rangle) 
\xrightarrow{\widetilde{d_0}} 
C_c^{\infty}(I,\langle\widetilde{\xi_2},\widetilde{\xi_3}\rangle) 
\xrightarrow{\widetilde{d_1}}
C_c^{\infty}(I,\langle\widetilde{\xi_4}\rangle)\rightarrow 0.
\end{align}
Computing explicitly the action of the exterior derivative \eqref{derivative} 
on the basis elements $\widetilde{\xi_i}$ we find
\begin{align*}
\widetilde{d_0}=\left(\begin{array}{c}(-1)^{n-k-1}\partial_x+\frac{c_{n-k-1}}{x}\\ 
x^{-1}\sqrt{\eta}\end{array}\right), 
\quad \widetilde{d_1}=\left(x^{-1}\sqrt{\eta}, \ (-1)^{n-k}\partial_x
+\frac{c_{n-k}}{x}\right).
\end{align*}
As before we compute the corresponding Laplacians and find
\begin{align}\label{phi-laplacians}
\widetilde{\Delta}_{0,\eta}=\widetilde{\Delta}_{1,\eta}=-\partial_x^2+\frac{1}{x^2}\left[\eta 
+\left(k+\frac{1}{2}-
\frac{n}{2}\right)^2-\frac{1}{4}\right]=\Delta_{0,\eta}=\Delta_{1,\eta}, 
\end{align}
where the operators are again identified with their scalar actions.

The second type of the subcomplexes comes from the harmonic forms 
$\cH^k(N,E_N)$ on $N$. Fix an orthonormal basis $\{u_i\}$ of $\cH^k(N,E_N)$
and observe that any subspace 
$C^{\infty}_c(I,\langle0\oplus u_i,u_i\oplus 0\rangle)$ 
is invariant under $d,d^*$. Consequently we obtain a subcomplex of the de Rham 
complex
\begin{equation}\label{complex-3}\begin{split}
0\to C^{\infty}_c(I,\langle 0\, \oplus \, &u^k_i\rangle)\xrightarrow{d^H_k} 
C^{\infty}_c(I,\langle u^k_i\oplus 0\rangle) \to 0, \\
&d^H_k=(-1)^k\partial_x+\frac{c_k}{x},
\end{split}
\end{equation}
where the action of $d^H_k$ is identified with its scalar action, as before. 
The Laplacians of the complex are given by
\begin{equation}
\begin{split}
H^k_0:&=(d^H_k)^*d^H_k=-\partial_x^2+\frac{1}{x^2}
\left(\left(\frac{(n-1)}{2}-k\right)^2-\frac{1}{4}\right), \\ 
H^k_1:&=d^H_k(d^H_k)^*=-\partial_x^2+\frac{1}{x^2}
\left(\left(\frac{(n+1)}{2}-k\right)^2-\frac{1}{4}\right).
\end{split}
\end{equation}  

\subsection{The Relative Ideal Boundary Conditions}

By \eqref{poincare} the analytic torsion  with respect to absolute or relative 
boundary conditions coincide. We will work with relative boundary conditions.

Let $d_{k,\max}$ denote the maximal extension of $d_k$ in 
$L^2(\Omega_c^*(C_I(N),E),g,h^E)$. Furthermore, let
$d_{k,\min}$ denote the graph closure of $d_k$ in 
$L^2(\Omega_c^*(C_I(N),E),g,h^E)$. Then we have
$\dom (d_{k,\min}) \subset \dom (d_{k,\max})$, where both spaces are the Hilbert 
spaces equipped with the graph-norm. 
Despite the fact that the differential $d_{k}$ is not elliptic, there is still 
a well-defined trace on its maximal domain by the trace theorem of Paquet 
\cite{Paq:PMP}.
\begin{thm}\label{trace-theorem} \textup{\cite[Theorem 1.9]{Paq:PMP}} 
Let $(X,g^X)$ be a compact oriented Riemannian manifold with isolated 
conical singularities and with smooth boundary $\partial X$. Let 
$\iota: \partial X \hookrightarrow X$ be the natural inclusion. 
Let $(E,\nabla, h^E)$ be a flat complex Hermitian vector bundle over $X$ and 
$(E_{\partial X},\nabla_{\partial X},h_{\partial X})$
its restriction to the boundary.
Then the pullback $\iota^*:\Omega^k(X,E) \to \Omega^k(\partial X, E_{\partial X})$
 with 
$\Omega^k(\partial X, E_{\partial X})=\{0\}$ for $k=\dim X$, extends continuously 
to a linear surjective map 
\begin{align}
\iota^*:\dom (d_{k,\max})\rightarrow H^{-1/2}(d_{k,\partial X}),
\end{align} 
where $d_{k,\partial X}$ is the  differential on 
$\Omega^k(\partial X, E_{\partial X})$,  
$H^{-1/2}(\Omega^k(\partial X, E_{\partial X}))$ the $(-1/2)$-th Sobolev space on 
$\partial X$ and
\begin{align*}
 H^{-1/2}(d_{k,\partial X}):=\{\w \in H^{-1/2}(\Omega^k(\partial X, E_{\partial X})) 
\mid 
d_{k,\partial X} \w \in H^{-1/2}(\Omega^{k+1}(\partial X, E_{\partial X}))\},
\end{align*}
is a Hilbert space with respect to the graph-norm.
\end{thm}

\begin{remark}
The trace theorem \cite[Theorem 1.9]{Paq:PMP} is stated for the untwisted 
case on compact (non-singular) 
Riemannian manifolds. Extension to flat Hermitian vector bundles is 
straightforward. Moreover, the analysis 
localizes to an open neighborhood of the boundary $\partial X$, so the trace 
theorem carries over to compact Riemannian manifolds with singular structure 
away from $\partial X$.
\end{remark}

Fix the relative extension of the Laplacian, induced by $d^*_{k,\min}$ 
\begin{align}
\Delta_k^{\textup{rel}}:=d^*_{k,\min}d_{k,\min}+d_{k-1,\min}
d^*_{k-1,\min}.
\end{align}
The minimal extension $d_{k,\min}$ is defined as the graph closure of the de
 Rham differential $d_k$ 
on $\Omega^*_c(C_I(N),E)$ in $L^2(\Omega^*(C_I(N),E),g,h^E)$. Hence, 
Theorem \ref{trace-theorem} implies 
\begin{equation}\label{rel-bc}
\begin{split}
&\dom (d_{k,\min})\subseteq \{\w\in \dom (d_{k,\max})|\iota^*\w=0\}, \\
&\dom (\Delta^{\textup{rel}}_k)\subseteq \{\w \in \dom (\Delta_{k,\max})| 
\iota^*\w=0, \iota^*(d^*_{k-1}\w)=0\}.
\end{split}
\end{equation}
By the Hodge decomposition of $\Omega^*(N,E_N)$, the de Rham complex 
$(\Omega_c^*(C_I(N)),d)$ 
decomposes completely into subcomplexes of the three types \eqref{complex-1}, 
\eqref{complex-2} and \eqref{complex-3}. 
It has been observed in \cite[Theorem 3.5]{Ver:ATO} that in each degree $k$ 
this induces a compatible decomposition 
for the relative extension of the Laplacian. In the classical language of 
\cite{Wei:LOI} we have a decomposition into reducing subspaces of the 
Laplacians. 
Hence the Laplacians $\Delta_k^{\textup{rel}}$ induce self-adjoint relative 
extensions of the 
Laplacians $\Delta_{j}^{\psi}, \Delta_{j}^{\phi}$, $j=0,1$, and $H^k_0,H^k_1$. 

In order to discuss the corresponding relative boundary conditions explicitly, 
note that by the classical theory of linear differential equations for any 
element $f$ of $\dom (\Delta_{j,\max}^{\psi}),\dom (\Delta_{j,\max}^{\phi}),j=0,2,$ 
or $\mathscr{D}(H^k_{i,\max}),i=0,1,$ 
$f$ and its derivative $f'$ are both locally absolutely continuous in $I$ with 
well-defined values at $x\in \partial I$. More 
precisely at $x=1$ in case $I=(0,1]$, and $x\in \{\e,1\}$ in case $I=[\e,1]$. 
Hence the following boundary conditions are well-defined 
\begin{equation*}
B_N^k(x)f:=f'(x)+(-1)^{k+1}c_{k}f(x), \quad B_D(x)f:=f(x), \quad x\in \partial I.
\end{equation*}
In case of $I=(0,1]$ boundary conditions at $x=0$ need to be imposed. 
By the well-known analysis, compare \cite{BruSee:AIT}, \cite{Che:OTS} and the 
basic discussion of the second author \cite{Ver:ZDF}, 
any solution $f\in L^2(0,1)$ to 
\begin{align}
 -\frac{d^2f}{dx^2}+\frac{1}{x^2}\left(\nu^2-\frac{1}{4}\right)f=g\in L^2(0,1),
\end{align}
admits an asymptotic expansion of the form
\begin{equation}
 f(x)\sim \left\{ 
\begin{array}{ll}
c_1(f)\sqrt{x}+c_2(f)\sqrt{x}\log(x)+O(x^{3/2}),  &\nu=0, \\
c_1(f)x^{\nu+1/2}+c_2(f)x^{-\nu+1/2}+O(x^{3/2}),  &\nu\in (0,1),\\
O(x^{3/2}), &\nu\geq 1,
\end{array}
\right. \ x\to 0,
\end{equation}
where the coefficients $c_1(f)$ and $c_2(f)$ depend only on $f$. Consequently 
the following boundary conditions at $x=0$ 
are well-defined
\begin{align}
 B_N(0)f:=
\left\{ 
\begin{array}{ll} c_1(f), & \nu \in [0,1), \\ 0, &\nu\geq 1,\end{array}\right.
\quad 
B_D(0)f:=
\left\{ 
\begin{array}{ll} c_2(f), & \nu \in [0,1), \\ 0, &\nu\geq 1.\end{array}\right.
\end{align}

\begin{prop}\label{rel-bc-prop} Let
 $\Delta_{j,\eta},\widetilde{\Delta}_{j,\eta}, j=0,1$, be the Laplacians
of the pair of subcomplexes \eqref{complex-1} and \eqref{complex-2}, and let  
$H^k_0,H^k_1,$ be the 
Laplacians  of the subcomplex \eqref{complex-3}. 
The domains of their relative self-adjoint extensions 
are given as follows. For $I=[\e,1]$
\begin{align*}
&\mathscr{D}(\Delta^{\textup{rel}}_{0,\eta})
=\{f\in\mathscr{D}(\Delta_{0,\eta}^{\max}) \mid B_D(\e)f=0, \ B_D(1)f=0\}=\mathscr{D}(\widetilde{\Delta}^{\textup{rel}}_{0,\eta}), \\ 
&\mathscr{D}(\Delta_{1,\eta}^{\textup{rel}})=\{f\in 
\mathscr{D}(\Delta_{1,\eta}^{\max})\mid B_N^{k+1}(\e)f=0, \ B_N^{k+1}(1)f=0\}, \\
&\mathscr{D}(\widetilde{\Delta}_{1,\eta}^{\textup{rel}})=\{f\in 
\mathscr{D}(\widetilde{\Delta}_{1,\eta}^{\max})\mid B_N^{n-k}(\e)f=0, \ B_N^{n-k}(1)f=0\}, \\
&\mathscr{D}(H^k_{1,\textup{rel}})=\{f\in \mathscr{D}(H^k_{1,\max})
\mid B_N^k(\e)f=0, \ B_N^k(1)f=0\}, \\
&\mathscr{D}(H^k_{0,\textup{rel}})=\{f\in \mathscr{D}(H^k_{0,\max})\mid B_D(\e)f=0, 
\ B_D(1)f=0\}.
\end{align*}
For $I=(0,1]$ the domains are given by
\begin{align*}
&\mathscr{D}(\Delta^{\textup{rel}}_{0,\eta})=\{f\in\mathscr{D}(\Delta_{0,\eta}^{\max})  \mid B_D(0)f=0, \ B_D(1)f=0\}=\mathscr{D}(\widetilde{\Delta}^{\textup{rel}}_{0,\eta}), \\ 
&\mathscr{D}(\Delta_{1,\eta}^{\textup{rel}})=\{f\in \mathscr{D}((\Delta_{1,\eta}^{\max})\mid B_D(0)f=0, \ B_N^{k+1}(1)f=0\}, \\
&\mathscr{D}(\widetilde{\Delta}_{1,\eta}^{\textup{rel}})=\{f\in \mathscr{D}(\widetilde{\Delta}_{1,\eta}^{\max})\mid B_D(0)f=0, \ B_N^{n-k}(1)f=0\}, \\
&\mathscr{D}(H^k_{1,\textup{rel}})=\{f\in \mathscr{D}(H^k_{1,\max})\mid B_N(0)f=0, \ B_N^k(1)f=0\}, \\
&\mathscr{D}(H^k_{0,\textup{rel}})=\{f\in \mathscr{D}(H^k_{0,\max})\mid B_D(0)f=0, \ B_D(1)f=0\}.
\end{align*}
\end{prop}

\begin{proof}
The choice of the boundary conditions at $x=1$ in case $I=(0,1]$, and 
at $x\in \{\e,1\}$ in case $I=[\e,1]$, 
follows for the individual relative self-adjoint extensions from 
\eqref{rel-bc}, the explicit form of the exterior differentials 
\eqref{derivative} and the fact that for any $x\in \partial I$ 
and the inclusion $\iota_x:\{x\}\times N \hookrightarrow C_I(N)$, we have 
$\iota^*_x(f_{k-1},f_k)=f_k(x)$ for any $(f_{k-1},f_k)\in \dom (\Delta_{\max})$ 
with $f_k$ continuous at $x$. The boundary conditions at $x=0$ in case 
$I=(0,1]$ have been determined in \cite[Corollary 2.14]{Ver:ZDF} and 
\cite[Proposition 3.6 and 3.7]{Ver:ATO}.
\end{proof}

\subsection{The Difference of Analytic Torsion for the  Truncated and the Full 
Cone}\label{difference}

Let $(N^n,g^N)$ be a closed Riemannian manifold of even dimension. Let 
$C(N):=(0,1]\times N$, equipped with the metric $g=dx^2\oplus x^2g^N$ and
for $\varepsilon>0$ let $C_\varepsilon(N):=[\varepsilon,1]\times N$, equipped
with the same metric.
Let $\Delta_k^{\textup{rel}}$ and $\Delta_{k,\e}^{\textup{rel}}$ denote the Laplacians 
with relative boundary conditions on $k$-forms associated to $(C(N),g)$ and 
$(C_{\e}(N),g)$, respectively. Put
\begin{align}
 T(\e,s):=\frac{1}{2} \, \sum_{k=1}^{\dim C(N)} (-1)^k\cdot k 
\cdot \left(\zeta(s,\Delta_{k,\e}^{\textup{rel}})
-\zeta(s,\Delta_{k}^{\textup{rel}})\right).
\end{align}
$T(\e,s)$ is related to the scalar analytic torsions of $(C(N),g)$ and 
$(C_{\e}(N),g)$ by
\begin{align}
 T'(\e,0)=\log T(C_{\e}(N),E,g)-\log T(C(N),E,g).
\end{align}
Consider the decomposition of the de Rham complex as described in Section 
\ref{decomposition}. 
For each fixed degree $k$, the subcomplexes \eqref{complex-1} and 
\eqref{complex-2} are determined 
by a coclosed eigenform $\psi \in \Omega^k(N,E_N)$ of the Laplacian 
$\Delta_{k,N}$ with eigenvalue $\eta>0$. For $k=0,\dots, n+1$ let
\[
E_k:=\textup{Spec}(\Delta_{k,ccl,N})\backslash \{0\}.
\]
For $\eta\in E_k$ the relative boundary conditions for the Laplacians 
$\Delta_{j,\eta}, \widetilde{\Delta}_{j,\eta}$, $j=0,1$, 
of the subcomplex-pair \eqref{complex-1} and \eqref{complex-2}, and 
the Laplacians $H^k_0,H^k_1,$ of the subcomplex \eqref{complex-3} 
are discussed in Proposition \ref{rel-bc-prop}. Here we distinguish operators 
on $(C_{\e}(N),g)$ by an additional $\e$-subscript. 
\begin{defn} \label{zetas} 
Denote by $\textup{m}(\eta)$ the multiplicity of $\eta \in E_k$. For 
$\textup{Re}(s)\gg 0$ put
\begin{equation}\begin{split}
&\zeta_{k,H}(s,\e):= \dim H^k(N,E_N) 
\left(\zeta(s, H^k_{0,\e,\textup{rel}}) - \zeta(s, H^k_{0,\textup{rel}})\right), \\
&\zeta_k(s,\e):=\sum_{\eta \in E_k} \textup{m}(\eta) 
\left(\zeta(s,\Delta_{1,\eta, \e}^{\textup{rel}})-
\zeta(s,\widetilde{\Delta}_{1,\eta,\e}^{\textup{rel}})\right)-\left(\zeta(s, 
\Delta_{1,\eta}^{\textup{rel}})-
\zeta(s, \widetilde{\Delta}_{1,\eta}^{\textup{rel}})\right).
\end{split}
\end{equation}
\end{defn}
We note that $T(\e,s)$ can be expressed in terms of $\zeta_{k,H}(s,\e)$ and 
$\zeta_k(s,\e)$ as follows
\begin{align}\label{T}
T(\e,s)=\frac{1}{2}\sum_{k=0}^{n/2-1}(-1)^k\zeta_k(s,\e) + 
\frac{1}{2}\sum_{k=0}^{n}(-1)^{k+1}\zeta_{k,H}(s,\e), \ \textup{Re}(s)\gg 0
\end{align}
(see  \cite[(4.3), (4.4)]{Ver:ATO}). Evaluation of $\zeta_k'(0,\e)$ requires 
application of the double summation method, which has been 
introduced by Spreafico in \cite{Spr:ZFA}, \cite{Spr:ZIF} and was applied by 
the second named author in \cite{Ver:ATO}, see Theorem \ref{BV-Theorem},
to derive the general formula for analytic torsion of a bounded cone in. 
Evaluation of $\zeta'_{k,H}(0,\e)$ reduces to an explicit computation of 
finitely many zeta-determinants and application of \cite{Les:DOR}.
We begin with the evaluation of $\zeta_k'(0,\e)$ for each fixed degree $k$ 
along the lines of \cite[Section 6]{Ver:ATO}.

\begin{prop}\label{N-prop}
For $c>0$ let
\[
\Lambda_c:=\{\lambda \in \C\colon |\textup{arg}(\lambda -c)|=\pi /4\}
\]
and assume that $\Lambda_c$ is oriented counter-clockwise. Put 
\begin{align*}
\A_k:=\frac{(n-1)}{2}-k, \quad 
\nu(\eta):=\sqrt{\eta + \A_k^2}, \ \eta \in \textup{Spec}.
\Delta_{k,ccl,N}\backslash \{0\}
\end{align*}
Let $c(\eta)=c_0/(2\nu(\eta)^2),$ where $c_0>0$ is a fixed positive number, 
smaller than the lowest non-zero eigenvalue of $\Delta^{\textup{rel}}_*$ and 
$\Delta^{\textup{rel}}_{*,\epsilon}$, such that $c(\eta)<1$ for all $\eta\in E_k$.
Then $\zeta_k(s,\e)$ 
admits the following integral representation for 
$\textup{Re}(s)\gg 0$
\begin{align*}
\zeta_k(s,\e)=\sum_{\eta \in E_k} \textup{m}(\eta) \nu(\eta)^{-2s} 
\frac{s^2}{\Gamma(s+1)}\int_0^{\infty}\frac{t^{s-1}}{2\pi i}
\int_{\wedge_{c(\eta)}}\frac{e^{-\lambda t}}{-\lambda}\, 
t_{\eta, \e}^{k}(\lambda)  d\lambda \, dt,
\end{align*}
where $t_{\eta,\e}^{k}(\lambda)$ is defined in terms of zeta-determinants by
\begin{equation}\label{tm}
\begin{split}
t_{\eta,\e}^{k}(\lambda)= - \log \frac{\det_{\zeta} 
\left(\Delta_{1,\eta,\e}^{\textup{rel}}-
\nu(\eta)^2\lambda\right)}{\det \left(\Delta_{1,\eta,\e}^{\textup{rel}}\right)}
+\log \frac{\det_{\zeta} \left(\widetilde{\Delta}_{1,\eta,\e}^{\textup{rel}}-
\nu(\eta)^2\lambda\right)}
{\det \left(\widetilde{\Delta}_{1,\eta,\e}^{\textup{rel}}\right)} \\
+ \log \frac{\det_{\zeta} \left(\Delta_{1,\eta}^{\textup{rel}}-
\nu(\eta)^2\lambda\right)}
{\det \left(\Delta_{1,\eta}^{\textup{rel}}\right)}
-\log \frac{\det_{\zeta} \left(\widetilde{\Delta}_{1,\eta}^{\textup{rel}}-
\nu(\eta)^2\lambda\right)}
{\det \left(\widetilde{\Delta}_{1,\eta}^{\textup{rel}}\right)}.
\end{split}
\end{equation}
and $\log$ denotes the main branch of the logarithm.
\end{prop}
\begin{proof}
Recall that the spectrum used to define $\zeta_k(s,\epsilon)$ is the union of 
the spectra for the Laplacians 
$\Delta_{1,\eta,\e}^{\textup{rel}}, \Delta_{1,\eta}^{\textup{rel}}$ and 
$\widetilde{\Delta}_{1,\eta,\e}^{\textup{rel}}, 
\widetilde{\Delta}_{1,\eta}^{\textup{rel}}$, where $\eta$ runs over $E_k$. 
For any choice 
\[
L(\eta)\in \left\{\Delta_{1,\eta,\e}^{\textup{rel}}, \Delta_{1,\eta}^{\textup{rel}}, 
\widetilde{\Delta}_{1,\eta,\e}^{\textup{rel}}, 
\widetilde{\Delta}_{1,\eta}^{\textup{rel}}\right\}, \eta\in E_k,
\]
the spectrum $\textup{Spec}\, L(\eta)\subset \R^+$ is strictly positive. 
Indeed, $\textup{Spec}\, L(\eta)$ is contained in
the spectrum of the non-negative Laplace operator on the truncated or full 
cone, and its zero eigenvalues arise in both cases only from harmonic 
forms $H^*(N,E_N)$. 
We note that the resolvent of $L(\eta)$ is a trace class operator 
\cite{Les:DOR}, and from Definition \ref{zetas} we infer for 
$\textup{Re}(s)\gg 0$
\begin{align}\label{integral}
\zeta_k(s,\e)= \sum_{\eta\in E_k} \textup{m}(\eta) \ \nu(\eta)^{-2s} 
\frac{1}{\Gamma(s)}\int_0^{\infty}t^{s-1}\frac{1}{2\pi i}
\int_{\wedge_{c(\eta)}}e^{-\lambda t}h_{\eta,\e}^{k}(\lambda) \,   d\lambda dt,
\end{align}
where
\begin{align*}
h_{\eta,\e}^{k}(\lambda) &=\textup{Tr}\left(\lambda - \nu(\eta)^{-2}
\Delta^{\psi(\eta)}_{2,\epsilon, \textup{rel}}\right)^{-1} 
- \textup{Tr}\left(\lambda - \nu(\eta)^{-2}
\Delta^{\phi(\eta)}_{2,\epsilon, \textup{rel}}\right)^{-1} \\
&-  \textup{Tr}\left(\lambda - \nu(\eta)^{-2}
\Delta_{1,\eta}^{\textup{rel}}\right)^{-1} 
+ \textup{Tr}\left(\lambda - \nu(\eta)^{-2}
\widetilde{\Delta}_{1,\eta}^{\textup{rel}}\right)^{-1}.
\end{align*}
For any choice of 
$$
L(\eta)\in \left\{\Delta_{1,\eta,\e}^{\textup{rel}}, \Delta_{1,\eta}^{\textup{rel}}, 
\widetilde{\Delta}_{1,\eta,\e}^{\textup{rel}}, \widetilde{\Delta}_{1,\eta}^{\textup{rel}}\right\}, \eta\in E_k,
$$
we find by \cite[Proposition 4.6]{Les:DOR} that, 
enumerating $\textup{Spec}\, L(\eta)=\{\lambda_i\}_{i=1}^{\infty}$
in increasing order, the series 
\begin{equation}\label{det-det}
\log \frac{\det_{\zeta} (L(\eta)-\nu(\eta)^{2}\lambda)}
{\det_{\zeta} L(\eta)} = \sum_{i=1}^{\infty} 
\log \left(1-\frac{\nu(\eta)^{2}\lambda}{\lambda_i}\right).
\end{equation}
converges and by the choice of the logarithm branch is holomorphic in 
$\lambda \in \C \backslash \{x\in \R \mid x > c(\eta)\}$. Moreover, 
\begin{align}\label{tr-det}
\textup{Tr} \left(\frac{L(\eta)}{\nu(\eta)^2}-\lambda\right)^{-1} 
=-\frac{d}{d\lambda} \log \frac{\det_{\zeta} 
(L(\eta)-\nu(\eta)^{2}\lambda)}{\det_{\zeta} L(\eta)}.
\end{align} 
By the definition of $c(\eta)>0$, \eqref{det-det} is holomorphic in an open 
neighborhood of 
the contour $\Lambda_{c(\eta)}$, and so we may integrate \eqref{integral} by 
parts 
first in $\lambda$ then in $t$, and obtain
\begin{align}
\zeta_k(s,\e)&=  \sum_{\eta\in E_k} \textup{m}(\eta) \ \nu(\eta)^{-2s}
\frac{1}{\Gamma(s)}\int_0^{\infty}t^{s-1}\frac{1}{2\pi i}
 \int_{\wedge_{c(\eta)}}e^{-\lambda t}h_{\eta,\e}^{k}(\lambda)  d\lambda dt 
\\&=  \sum_{\eta\in E_k} \textup{m}(\eta) \ \nu(\eta)^{-2s} 
\frac{s^2}{\Gamma(s+1)}\int_0^{\infty}t^{s-1}\frac{1}{2\pi i}
\int_{\wedge_{c(\eta)}}\frac{e^{-\lambda t}}{-\lambda}t_{\eta,\e}^{k}(\lambda) d\lambda dt,
\end{align}
where 
\begin{equation}
\begin{split}
t_{\eta,\e}^{k}(\lambda)= - \log \frac{\det_{\zeta} 
\left(\Delta_{1,\eta,\e}^{\textup{rel}}-
\nu(\eta)^2\lambda\right)}{\det \left(\Delta_{1,\eta,\e}^{\textup{rel}}\right)}
+\log \frac{\det_{\zeta} \left(\widetilde{\Delta}_{1,\eta,\e}^{\textup{rel}}-
\nu(\eta)^2\lambda\right)}
{\det \left(\widetilde{\Delta}_{1,\eta,\e}^{\textup{rel}}\right)} \\
+ \log \frac{\det_{\zeta} \left(\Delta_{1,\eta}^{\textup{rel}}-
\nu(\eta)^2\lambda\right)}
{\det \left(\Delta_{1,\eta}^{\textup{rel}}\right)}
-\log \frac{\det_{\zeta} \left(\widetilde{\Delta}_{1,\eta}^{\textup{rel}}-
\nu(\eta)^2\lambda\right)}
{\det \left(\widetilde{\Delta}_{1,\eta}^{\textup{rel}}\right)}.
\end{split}
\end{equation}
\end{proof}  

\begin{lemma}\label{det-Bessel}
For any $\nu>0$ and $z \in \C$ we have
\begin{equation*}
\begin{split}
\frac{\det_{\zeta} \left(\Delta_{1,\eta}^{\textup{rel}}+\nu^2z^2\right)}
{\det \left(\Delta_{1,\eta}^{\textup{rel}}\right)}&= \frac{2^{\nu}\Gamma(\nu)}
{(\nu z)^{\nu}(1+\A_k/\nu)}\left(\nu z I'_{\nu}(\nu z)
+\A_k I_{\nu}(\nu z)\right),\\ 
\frac{\det_{\zeta} \left(\widetilde{\Delta}_{1,\eta}^{\textup{rel}}+\nu^2z^2\right)}
{\det \left(\widetilde{\Delta}_{1,\eta}^{\textup{rel}}\right)}&
=\frac{2^{\nu}\Gamma(\nu)}
{(\nu z)^{\nu}(1-\A_k/\nu)}\left(\nu z I'_{\nu}(\nu z)
-\A_k I_{\nu}(\nu z)\right). \\
\frac{\det_{\zeta} \left(\Delta_{1,\eta,\e}^{\textup{rel}}+\nu^2z^2\right)}
{\det \left(\Delta_{1,\eta,\e}^{\textup{rel}}\right)}&= 
\frac{\left(\nu z I'_{\nu}(\nu z)+\A_k I_{\nu}(\nu z)\right) 
\left(\nu z \e K'_{\nu}(\nu z \e)+\A_k K_{\nu}(\nu z \e)\right) }
{(\nu^2-\A_k^2)(\epsilon^{\nu}-\epsilon^{-\nu})}
\\ \times 2\nu\sqrt{\epsilon}&
\left(1-\frac{ \nu z K'_{\nu}(\nu z)+\A_k K_{\nu}(\nu z)}
{ \nu z I'_{\nu}(\nu z)+\A_k I_{\nu}(\nu z)}
\cdot \frac{ \nu z \e I'_{\nu}(\nu z \e)+\A_k I_{\nu}(\nu z \e)}
{\nu z \e K'_{\nu}(\nu z \e)+\A_k K_{\nu}(\nu z \e)}\right),
\\
\frac{\det_{\zeta} \left(\widetilde{\Delta}_{1,\eta,\e}^{\textup{rel}}
+\nu^2z^2\right)}
{\det \left(\widetilde{\Delta}_{1,\eta,\e}^{\textup{rel}}\right)}&=
\frac{\left(\nu z I'_{\nu}(\nu z)-\A_k I_{\nu}(\nu z)\right) 
\left(\nu z \e K'_{\nu}(\nu z \e)-\A_k K_{\nu}(\nu z \e)\right) }
{(\nu^2-\A_k^2)(\epsilon^{\nu}-\epsilon^{-\nu})}
\\ \times 2\nu\sqrt{\epsilon}&
\left(1-\frac{ \nu z K'_{\nu}(\nu z)-\A_k K_{\nu}(\nu z)}
{ \nu z I'_{\nu}(\nu z)-\A_k I_{\nu}(\nu z)}
\cdot \frac{ \nu z \e I'_{\nu}(\nu z \e)-\A_k I_{\nu}(\nu z \e)}
{\nu z \e K'_{\nu}(\nu z \e)-\A_k K_{\nu}(\nu z \e)}\right)
\end{split}
\end{equation*}
\end{lemma}

\begin{proof}
We evaluate the zeta-determinants using their explicit relation with the 
normalized solutions of the operators, established by Lesch in 
\cite[Theorem 1.2]{Les:DOR}. The first two equations have been evaluated in 
\cite[Corollary 6.3]{Ver:ATO}:
\begin{equation}\label{det1}
\begin{split}
\frac{\det_{\zeta} \left(\Delta_{1,\eta}^{\textup{rel}}+\nu^2z^2\right)}
{\det \left(\Delta_{1,\eta}^{\textup{rel}}\right)}= \frac{2^{\nu}\Gamma(\nu)}
{(\nu z)^{\nu}(1+\A_k/\nu)}\left(\nu z I'_{\nu}(\nu z)
+\A_k I_{\nu}(\nu z)\right),\\ 
\frac{\det_{\zeta} \left(\widetilde{\Delta}_{1,\eta}^{\textup{rel}}+\nu^2z^2\right)}
{\det \left(\widetilde{\Delta}_{1,\eta}^{\textup{rel}}\right)}
=\frac{2^{\nu}\Gamma(\nu)}
{(\nu z)^{\nu}(1-\A_k/\nu)}\left(\nu z I'_{\nu}(\nu z)-\A_k I_{\nu}(\nu z)\right).
\end{split}
\end{equation}
In order to evaluate zeta determinants of $\Delta_{1,\eta,\e}^{\textup{rel}}$ and 
$\widetilde{\Delta}_{1,\eta,\e}^{\textup{rel}}$, consider solutions 
$f_{\psi, \nu}(\cdot, z)$ and $f_{\phi,\nu}(\cdot, z)$
of $(\Delta_{1,\eta,\e}^{\textup{rel}}+z^2)u=0$ and 
$(\widetilde{\Delta}_{1,\eta,\e}^{\textup{rel}}+z^2)v=0$, 
respectively, normalized at $x=1$. By definition 
(see \cite[(1.38a), (1.38b)]{Les:DOR} these are solutions of the 
respective operators, 
satisfying relative boundary conditions at $x=1$ and are normalized by  
$f_{\psi, \nu}(1, z)=1$ and $f_{\phi,\nu}(1, z)=1$, i.e. 
\begin{align*}
\begin{array}{lll}
(\Delta_{1,\eta,\e}+z^2)f_{\psi, \nu}(\cdot, z)=0, & 
f'_{\psi, \nu}(1, z) + (-1)^kc_{k+1}f_{\psi, \nu}(1, z)=0, & 
f_{\psi, \nu}(\cdot, z)=1, \\ 
(\widetilde{\Delta}_{1, \eta,\e}+z^2)f_{\phi, \nu}(\cdot, z)=0, 
& f'_{\phi, \nu}(1, z) + (-1)^{n-k+1}c_{n-k}f_{\phi, \nu}(1, z)=0, & 
f_{\phi, \nu}(\cdot, z)=1.
\end{array}
\end{align*}
Normalized solutions are uniquely determined and explicit computations lead to 
the following expressions
\begin{equation}\label{f-mu}
\begin{split}
f_{\psi, \nu}(x, z)&=(zI'_{\nu}(z)+\A_kI_{\nu}(z))\sqrt{x}K_{\nu}(zx) 
- (zK'_{\nu}(z)+\A_kK_{\nu}(z))\sqrt{x}I_{\nu}(zx), \\
f_{\phi, \nu}(x, z)&=(zI'_{\nu}(z)-\A_kI_{\nu}(z))\sqrt{x}K_{\nu}(zx) 
- (zK'_{\nu}(z)-\A_kK_{\nu}(z))\sqrt{x}I_{\nu}(zx), \\
f_{\psi, \nu}(x, 0)&=\frac{1}{2\nu} (\nu-\A_k)x^{\nu+1/2} + \frac{1}{2\nu} 
(\nu +\A_k)x^{-\nu+1/2}, \\
f_{\phi, \nu}(x, 0)&=\frac{1}{2\nu} (\nu+\A_k)x^{\nu+1/2} + \frac{1}{2\nu} 
(\nu -\A_k)x^{-\nu+1/2},
\end{split}
\end{equation}
where we use 
\begin{align}
K_{\nu}(z)I'_{\nu}(z)-K'_{\nu}(z)I_{\nu}(z)=\frac{1}{z}.
\end{align}
In view of \cite[Theorem 1.2]{Les:DOR} we find
\begin{equation}\label{det2}
\begin{split}
\frac{\det_{\zeta} \left(\Delta_{1,\eta,\e}^{\textup{rel}}+\nu^2z^2\right)}
{\det \left(\Delta_{1,\eta,\e}^{\textup{rel}}\right)}= 
\frac{f'_{\psi, \nu}(\e, \nu z) +
 (-1)^kc_{k+1}f_{\psi, \nu}(\e, \nu z)}{f'_{\psi, \nu}(\e, 0) 
+ (-1)^kc_{k+1}f_{\psi, \nu}(\e, 0)}, \\
\frac{\det_{\zeta} \left(\widetilde{\Delta}_{1,\eta,\e}^{\textup{rel}}+\nu^2z^2\right)}
{\det \left(\widetilde{\Delta}_{1,\eta,\e}^{\textup{rel}}\right)}
=\frac{f'_{\phi, \nu}(\e, \nu z) + 
(-1)^kc_{k+1}f_{\phi, \nu}(\e, \nu z)}{f'_{\phi, \nu}(\e, 0) 
+ (-1)^kc_{k+1}f_{\phi, \nu}(\e, 0)}.
\end{split}
\end{equation}
We note that in the non-singular case this is due to 
Burghelea-Friedlander-Kappeler in \cite{BFK:OTD}. 
Plugging in the expressions \eqref{f-mu} we obtain the lemma.
\end{proof}

In particular, applying Lemma \ref{det-Bessel} several cancellations lead to a 
representation of $t_{\eta,\e}^{k}(\lambda)$ in terms of Bessel functions with
$\nu\equiv \nu(\eta)$ and $z=\sqrt{-\lambda}$, where we use the main branch of 
logarithm in $\C\backslash \R^-$, extended by continuity to one of the 
of the cut
\begin{equation}\label{t-bessel}
\begin{split}
t_{\eta,\e}^{k}(\lambda)=
-&\log \left(\nu z \e K'_{\nu}(\nu z \e)+\A_k K_{\nu}(\nu z \e)\right)
-\log\left(1+\frac{\A_k}{\nu}\right)\\
+&\log \left(\nu z \e K'_{\nu}(\nu z \e)-\A_k K_{\nu}(\nu z \e)\right)  
+\log\left(1-\frac{\A_k}{\nu}\right)\\
- &\log \left(1-\frac{ \nu z K'_{\nu}(\nu z)+\A_k K_{\nu}(\nu z)}
{ \nu z I'_{\nu}(\nu z)+\A_k I_{\nu}(\nu z)}
\cdot \frac{ \nu z \e I'_{\nu}(\nu z \e)+\A_k I_{\nu}(\nu z \e)}
{\nu z \e K'_{\nu}(\nu z \e)+\A_k K_{\nu}(\nu z \e)}\right) \\
+&\log \left(1-\frac{ \nu z K'_{\nu}(\nu z)-\A_k K_{\nu}(\nu z)}
{ \nu z I'_{\nu}(\nu z)-\A_k I_{\nu}(\nu z)}
\cdot \frac{ \nu z \e I'_{\nu}(\nu z \e)-\A_k I_{\nu}(\nu z \e)}
{\nu z \e K'_{\nu}(\nu z \e)-\A_k K_{\nu}(\nu z \e)}\right).
\end{split}
\end{equation}

For the arguments below we need to summarize some facts about Bessel functions.
We consider expansions of Bessel-functions for large arguments and fixed order, 
(see \cite[p.377]{AbrSte:HOM}).
 For the modified Bessel functions of first kind we have
\begin{equation}\label{large-arg-I}
\begin{split}
I_{\nu}(z)=\frac{e^z}{\sqrt{2\pi z}}\left(1+O\left(\frac{1}{z}\right)\right), \\
 I'_{\nu}(z)=\frac{e^z}{\sqrt{2\pi z}}\left(1+O\left(\frac{1}{z}\right)\right), 
\end{split}
\quad |z|\to \infty.
\end{equation}
Expansions for modified Bessel functions of second kind are
\begin{equation}\label{large-arg-K}
\begin{split}
K_{\nu}(z)=\sqrt{\frac{\pi}{2 z}}e^{-z}\left(1+O\left(\frac{1}{z}\right)\right),
 \\
K'_{\nu}(z)=-\sqrt{\frac{\pi}{2 z}}e^{-z}\left(1+O\left(\frac{1}{z}\right)\right),
\end{split}
\quad |z|\to \infty.
\end{equation}
The expansions \eqref{large-arg-I} and \eqref{large-arg-K} hold in  
$|\textup{arg}(z)|<\pi /2$, in particular they hold for $z=\sqrt{-\lambda}$ 
with $\lambda \in \Lambda_c$ large. 
For small arguments and positive orders $\nu>0$ we have the following expansions
\begin{equation}\label{small}
\begin{split}
I_{\nu}(z)\sim \frac{z^{\nu}}{2^{\nu}\Gamma (\nu+1)}, 
\quad K_{\nu}(z)\sim 2^{\nu-1}\frac{\Gamma(\nu)}{z^{\nu}}&, \\ 
I'_{\nu}(z)\sim \frac{z^{\nu-1}}{2^{\nu}\Gamma(\nu)}, 
\quad K'_{\nu}(z)\sim-2^{\nu-1}\frac{\Gamma(\nu+1)}{z^{\nu+1}}&,
\end{split}
\quad \textup{as} \ |z|\to 0.
\end{equation}
Next recall the expansions of Bessel-functions for large order $\nu>0$
(see \cite[Section 7]{Olv:AAS}). 
For any $z\in\{w\in \C\colon|\textup{arg}(w)|<\pi /2\}\cup \{ix|x\in (-1,1)\},$ 
put 
\[
t:=(1+z^2)^{-1/2}\quad \textup{and}\quad \xi:=1/t+\log (z/(1+1/t)).
\]
For the modified Bessel functions of first kind we then have
\begin{equation}\label{large-nu-I}
\begin{split}
I_{\nu}(\nu z)  = \frac{1}{\sqrt{2\pi \nu}}\frac{e^{\nu \xi}}{(1+z^2)^{1/4}} 
\left[1+\sum_{r=1}^{N-1}\frac{u_r(t)}{\nu^r} + \frac{\eta_{N,1}(\nu, z)}{\nu^{N}} 
\right], & \\
I'_{\nu}(\nu z)  = \frac{1}{\sqrt{2\pi \nu}
} \frac{e^{\nu \xi}}{z (1+z^2)^{-1/4}}
\left[1+\sum_{r=1}^{N-1}\frac{v_r(t)}{\nu^r} + \frac{\eta_{N,2}(\nu, z)}{\nu^{N}}
 \right]. &
\end{split}
\end{equation} 
Expansions for modified Bessel functions of second kind are
\begin{equation}\label{large-nu-K}
\begin{split}
K_{\nu}(\nu z) = \sqrt{\frac{\pi}{2 \nu}} \frac{e^{-\nu \xi}}{(1+z^2)^{1/4}} 
\left[1+\sum_{r=1}^{N-1}\frac{u_r(t)}{(-\nu)^r} + \frac{\eta_{N,3}(\nu, z)}
{(-\nu)^{N}} \right], & \\
K'_{\nu}(\nu z) =  -\sqrt{\frac{\pi}{2 \nu}} \frac{e^{-\nu \xi}}
{z (1+z^2)^{-1/4}}
\left[1+\sum_{r=1}^{N-1}\frac{v_r(t)}{(-\nu)^r} + \frac{\eta_{N,4}(\nu, z)}
{(-\nu)^{N}} \right]. &
\end{split}
\end{equation}
The error terms $\eta_{N,i}(\nu, z)$ are bounded for large $\nu$ 
uniformly in any compact subset of 
$\{z\in \C\colon |\textup{arg}(z)|<\pi /2\}\cup \{ix|x\in (-1,1)\}$. For this 
fact see the analysis of the validity regions for the expansions 
\eqref{large-nu-I} and \eqref{large-nu-K} in \cite[Section 8]{Olv:AAS}. 
For $\lambda \in \Lambda_c$ with $0<c<1$, the induced $z=\sqrt{-\lambda}$ 
is contained in that region of validity, 
where we use the main branch of logarithm in $\C\backslash \R^-$, extended by 
continuity to one of the sides of the cut. The 
coefficients $u_r(t),v_r(t)$ are polynomial in $t$ and defined via a recursive 
relation (see \cite[(7.10)]{Olv:AAS}).

As in \cite[(3.15)]{BKD:HKA} we have for any fixed $\A\in \R$ 
the following expansion as $\nu\to \infty$ 
\begin{equation}\label{polynom2}
\begin{split}
&\log  \left(1+\sum_{r=1}^{N}\frac{u_r(t)}{(\pm\nu)^r}+ O(\nu^{-N-1})\right) 
\sim \sum_{r=1}^{\infty}\frac{D_r(t)}{(\pm\nu)^r} + O(\nu^{-N-1}), \\
&\log \left[ \left(1+\sum_{k=1}^{N}\frac{v_r(t)}{(\pm\nu)^r} \right)
+ \frac{\A}{(\pm\nu)}t\left(1+\sum_{r=1}^{N-1}\frac{u_r(t)}{(\pm\nu)^r}\right)
+ O(\nu^{-N-1})\right] 
 \\ &\sim \sum_{r=1}^{N}\frac{M_r(t, \A)}{(\pm\nu)^r} + O(\nu^{-N-1}).
\end{split}
\end{equation}
The coefficients $D_r(t)$ and $M_r(t,\A)$ are polynomial in $t$ of the form
\begin{align}\label{MD-polynom}
D_r(t)=\sum_{b=0}^{r}x_{r,b}t^{r+2b}, \quad 
M_r(t, \A)=\sum_{b=0}^{r}z_{r,b}(\A)t^{r+2b}.
\end{align}
This follows from the fact that the $u_r(t)$'s and $v_r(t)$'s are polynomials.
See also \cite[(3.7), (3.16)]{BKD:HKA}. 
As a consequence of \cite[(4.24)]{BGKE:ZFD} we have
\begin{align}\label{DM}
M_{r}(1,\A) = D_{r}(1)-\frac{(-\A)^{r}}{r}.
\end{align}
\begin{prop}\label{contour-change}
There exist $\e,c>0$  such that for 
$\textup{Re}(s)\gg 0$ we have
\begin{align*}
\zeta_k(s,\e) &=
\frac{s^2}{\Gamma(s+1)}\int_0^{\infty}\frac{t^{s-1}}{2\pi i}
\int_{\wedge_{c}}\frac{e^{-\lambda t}}{-\lambda}\, 
T^k_{\e}(s,\lambda) d\lambda \, dt, \\
T^k_{\e}(s,\lambda) &=\sum_{\eta \in E_k} \textup{m}(\eta) 
\ t_{\eta, \e}^{k}(\lambda)  \, \nu(\eta)^{-2s} 
\end{align*}
\end{prop}
\begin{proof}
Consider the expression of $t_{\eta, \e}^{k}(\lambda)$ in \eqref{t-bessel}
in terms of Bessel functions. We need to investigate its behavior for 
large $\eta$, or equivalently for large $\nu(\eta)$.
Let $z\in\{w\in \C\colon|\textup{arg}(w)|<\pi /2\}\cup \{ix|x\in (-1,1)\}$ 
and $t_{\e}:=(1+(\e z)^2)^{-1/2}$. By \eqref{large-nu-K} we find
\begin{equation}\label{log1}
\begin{split}
-\log &\left(\nu z \e K'_{\nu}(\nu z \e)+\A_k K_{\nu}(\nu z \e)\right) +
\log \left(\nu z \e K'_{\nu}(\nu z \e)-\A_k K_{\nu}(\nu z \e)\right) \\
= \, &-\log \left[ \left(1+\sum_{k=1}^{N-1}\frac{v_r(t_{\e})}{(-\nu)^r} 
\right)+ \frac{\A_k}{(-\nu)}t_{\e}\left(1+\sum_{r=1}^{N-2}
\frac{u_r(t_{\e})}{(-\nu)^r}\right) + 
\frac{\kappa_{N,1}(\nu,z\e)}{(-\nu)^N}\right] \\
&+\log \left[ \left(1+\sum_{k=1}^{N-1}\frac{v_r(t_{\e})}{(-\nu)^r} \right)- 
\frac{\A_k}{(-\nu)}t_{\e}\left(1+\sum_{r=1}^{N-2}\frac{u_r(t_{\e})}{(-\nu)^r}
\right) +
\frac{\kappa_{N,2}(\nu,z\e)}{(-\nu)^N}\right] ,
\end{split}
\end{equation}
where the error terms
\begin{equation}
\begin{split}
\kappa_{N,1}(\nu,z\e)&=
\eta_{N,4}(\nu,z\e) + (\A_kt_{\e}) \eta_{N-1,3}(\nu,z\e) \\
\kappa_{N,2}(\nu,z\e)&=
\eta_{N,4}(\nu,z\e) - (\A_kt_{\e}) \eta_{N-1,3}(\nu,z\e)
\end{split}
\end{equation}
are bounded for large $\nu$ 
uniformly in any compact subset of 
$\{z\in \C\colon|\textup{arg}(z)|<\pi /2\}\cup \{ix|x\in (-1,1)\}$.
Employing \eqref{large-nu-I} and \eqref{large-nu-K} we find with
$\xi:=1/t+\log (z/(1+1/t))$ and $\xi_{\e}:=1/t_{\e}+\log (\e z/(1+1/t_{\e}))$
\begin{equation}\label{log2}
\begin{split}
&\log \left( 1- \frac{ \nu z K'_{\nu}(\nu z)\pm\A_k 
K_{\nu}(\nu z)}{ \nu z I'_{\nu}(\nu z)\pm
\A_k I_{\nu}(\nu z)}\cdot \frac{ \nu z \e I'_{\nu}(\nu z \e)\pm\A_k I_{\nu}
(\nu z \e)}{\nu z \e K'_{\nu}(\nu z \e)\pm\A_k K_{\nu}(\nu z \e)}\right) = \\
&\log \left( 1- e^{2\nu(\xi_{\e}-\xi)}(1+\kappa(\nu,z))\right),
\end{split}
\end{equation}
where the error term $\kappa(\nu,z)$ is again bounded for large $\nu$ 
uniformly in any compact subset of 
$\{z\in \C\colon|\textup{arg}(z)|<\pi /2\}\cup \{ix|x\in (-1,1)\}$.
We need to consider the difference $(\xi_{\e}-\xi)$ in detail.
\begin{align*}
\xi_{\e}-\xi=\sqrt{1+(\e z)^2}-\sqrt{1+z^2}
+\log\left(\frac{\e z}{1+\sqrt{1+(\e z)^2}}\right) - 
\log\left(\frac{ z}{1+\sqrt{1+z^2}}\right)&\\
=\sqrt{1+(\e z)^2}\left[1-\frac{1}{\e}
\sqrt{\frac{\e^2+(\e z)^2}{1+(\e z)^2}}\right] +
\log\left(\frac{\e z}{1+\sqrt{1+(\e z)^2}}\right) 
- \log\left(\frac{ z}{1+\sqrt{1+z^2}}\right)&.
\end{align*}
We are interested in the asymptotic behavior of $(\xi_{\e}-\xi)$ as 
$\e \to 0$, which is possibly
non-uniform in $z$. Hence, we consider $(\xi_{\e}-\xi)$ under three asymptotic 
regimes, 
$|\e z|\to \infty, |\e z|\to 0$ and $|\e z|\sim \textup{const}$.  We find by 
straightforward estimates
\begin{equation}\begin{split}
&\textup{Re}\,(\xi_{\e}-\xi) \sim \e \, \textup{Re}(z)(1-1/\e) = 
\textup{Re}(z)(\e-1), \ \textup{as} \ |\e z|\to \infty, \ \e \to 0,\\
&\textup{Re}\,(\xi_{\e}-\xi) \sim \log |\e z|-\textup{Re}\sqrt{1+z^2}, \ 
\textup{as} \ |\e z|\to 0, \ \e \to 0,\\
&\textup{Re}\,(\xi_{\e}-\xi) \sim -C\e^{-1}, \ \textup{as} \ |\e z|\sim 
\textup{const}, \ \e \to 0,
\end{split}\end{equation}
for some constant $C>0$. For 
$\{z\in \C\colon|\textup{arg}(z)|<\pi /2\}\cup \{z=ix|x\in (-1,1)\}$, 
we have $\textup{Re}\sqrt{1+z^2}>0$, and $\textup{Re}(z)>0$ as $|z|\to \infty$. 
Consequently, for $\e>0$ sufficiently small
$\textup{Re}(\xi_{\e}-\xi)<\delta <0$
for some fixed $\delta<0$ and hence $\exp(2\nu(\xi_{\e}-\xi))$ 
vanishes as $\nu\to \infty$, uniformly in any compact subset of 
$\{z\in \C\colon|\textup{arg}(z)|<\pi /2\}\cup \{ix|x\in (-1,1)\}$. \medskip

The uniform expansions above show 
that in \eqref{log1} and \eqref{log2} the arguments of the logarithms stay away 
from the branch cut $\C\backslash \R^-$ for $\nu$ large enough and $\e>0$ 
sufficiently small, uniformly in any compact subset of 
$\{z\in \C\colon|\textup{arg}(z)|<\pi /2\}\cup \{ix|x\in (-1,1)\}$. 
Consequently, in view of the expression
\eqref{t-bessel}, $t_{\eta, \e}^{k}(\lambda)$ is in particular holomorphic 
in an open neighborhood of $\{\lambda \in [0,c']\}\subset \C$
for some $0\!<\!c'\!<\!1$ and $\nu(\eta)>\nu_0$. Moreover,
for any $\eta\in E_k$, $t_{\eta, \e}^{k}(\lambda)$ is  holomorphic 
in $\lambda \in \C\backslash \{x\in \R\mid x>c(\eta)\}$. Thus, setting 
$c:=\min \{c', c(\eta)\mid \eta \in E_k, \nu(\eta)
\leq \nu_0\},$ we deduce for $\e>0$ 
sufficiently small
\begin{align*}
\zeta_k(s,\e)=\sum_{\eta \in E_k}  \textup{m}(\eta) \ \nu(\eta)^{-2s} 
\frac{s^2}{\Gamma(s+1)}\int_0^{\infty}\frac{t^{s-1}}{2\pi i}
\int_{\wedge_{c}}\frac{e^{-\lambda t}}{-\lambda}\, 
 t_{\eta, \e}^{k}(\lambda)  d\lambda \, dt,
\end{align*}
The deforming of the contour of integration  from $\Lambda_{c(\eta)}$ to 
$\Lambda_c$ is permissible, as the deformation is performed within the region 
of regularity for each 
$t_{\eta, \e}^{k}(\lambda), \eta\in E_k$. Employing again the expansions 
\eqref{large-nu-I} and \eqref{large-nu-K} we find that 
\begin{align}
\sum_{\eta \in E_k}  \textup{m}(\eta) \ t_{\eta, \e}^{k}(\lambda) \nu(\eta)^{-2s}, \ 
\textup{Re}(s)\gg 0,
\end{align}
converges uniformly in $\lambda \in \Lambda_c$ and hence by the 
uniform convergence of the integrals and series we arrive 
at the statement of the proposition.
\end{proof}

\begin{prop}\label{large-nu} 
Let the notation be as in Proposition \ref{N-prop} and \ref{contour-change}. 
Let $\lambda \in \Lambda_c$ 
and $t_{\e}(\lambda):=(1-(\e^2 \lambda))^{-1/2}$. 
Then for $\e>0$ sufficiently small we have the following 
asymptotic expansion for large $\nu(\eta), \eta\in E_k$
\begin{align*}
t_{\eta,\e}^{k}(\lambda) \sim  \sum_{r=1}^{\infty}\frac{1}{(-\nu(\eta))^r}\left(M_r(t_{\e}(\lambda) ,-\A_k)-
M_r(t_{\e}(\lambda) ,\A_k) + \left(\frac{\A_k^r-(-\A_k)^r}{r}\right) \right).
\end{align*}
\end{prop}

\begin{proof}
Consider expansions of the individual terms in the expression \eqref{t-bessel} for $t_{\eta,\e}^{k}(\lambda)$. 
For $\lambda \in \Lambda_c$, $z=\sqrt{-\lambda}$ lies in the region of validity of the 
expansions \eqref{large-nu-I} and \eqref{large-nu-K}.
Combining \eqref{large-nu-K} and \eqref{polynom2} we compute for large $\nu\equiv \nu(\eta)$
\begin{align*}
-\log &\left(\nu z \e K'_{\nu}(\nu z \e)+\A_k K_{\nu}(\nu z \e)\right) +
\log \left(\nu z \e K'_{\nu}(\nu z \e)-\A_k K_{\nu}(\nu z \e)\right) \\
\sim \, &-\log \left[ \left(1+\sum_{k=1}^{\infty}\frac{v_r(t_{\e}(\lambda))}{(-\nu)^r} \right)+ \frac{\A_k}{(-\nu)}t_{\e}(\lambda)\left(1+\sum_{r=1}^{\infty}\frac{u_r(t_{\e}(\lambda))}{(-\nu)^r}\right)\right] \\
&+\log \left[ \left(1+\sum_{k=1}^{\infty}\frac{v_r(t_{\e}(\lambda))}{(-\nu)^r} \right)- 
\frac{\A_k}{(-\nu)}t_{\e}(\lambda)\left(1+\sum_{r=1}^{\infty}\frac{u_r(t_{\e}(\lambda))}{(-\nu)^r}\right)\right] \\ 
&\sim \sum_{r=1}^{\infty}\frac{1}{(-\nu)^r} \left(-M_r(t_{\e}(\lambda), \A_k) 
+ M_r(t_{\e}(\lambda), -\A_k)\right), 
\ \textup{as} \ \nu\to \infty.
\end{align*}
The standard expansion of the logarithm yields for  large $\nu$
\begin{align*}
-\log\left(1+\frac{\A_k}{\nu}\right)+\log\left(1-\frac{\A_k}{\nu}\right)=
\sum_{r=1}^{\infty}\frac{1}{(-\nu)^r}\left(\frac{\A_k^r-(-\A_k)^r}{r}\right).
\end{align*}
This already gives all the terms in the stated asymptotic expansion of 
$t_{\eta,\e}^{k}(-z^2)$. 
Thus we need to check that the remaining terms indeed have no asymptotic 
contribution. 
Using \eqref{large-nu-I}, \eqref{large-nu-K}, and putting 
$\xi_{\e}:=1/t_{\e}(\lambda)+\log (\e z/(1+1/t_{\e}(\lambda)))$, 
the remaining terms are estimated as follows
\begin{align}\label{eta-difference}
\frac{ \nu z K'_{\nu}(\nu z)\pm\A_k K_{\nu}(\nu z)}{ \nu z I'_{\nu}(\nu z)\pm
\A_k I_{\nu}(\nu z)}\cdot \frac{ \nu z \e I'_{\nu}(\nu z \e)\pm\A_k I_{\nu}
(\nu z \e)}{\nu z \e K'_{\nu}(\nu z \e)\pm\A_k K_{\nu}(\nu z \e)}\sim 
O(e^{2\nu(\xi_{\e}-\xi)}), \ \nu\to\infty.
\end{align}
The difference $(\xi_{\e}-\xi)$ has been considered in detail
in Proposition \ref{contour-change}.
For $\e$ sufficiently small, $\textup{Re}(\xi_{\e}-\xi)<0$ and hence 
the remainder term $O(e^{2\nu(\xi_{\e}-\xi)})$ 
in \eqref{eta-difference} does not contribute to the asymptotic expansion for 
large $\nu$.
\end{proof}

Next we introduce a (shifted) zeta-function by
\begin{align}\label{shifted-zeta}
\zeta_{k,N}(s):=\sum_{\eta \in E_k} \textup{m}(\eta) \ \nu(\eta)^{-s}=
\zeta\left(\frac{s}{2},\, \Delta_{k,ccl,N}+\A_k^2\right), \, \textup{Re}(s)>n,
\end{align}
where as before $\textup{m}(\eta)$ denotes the multiplicity of $\eta \in E_k$.
The heat trace expansions for $(\Delta_{k,ccl,N}+\A_k^2)$ and $\Delta_{k,ccl,N}$ 
have the same exponents, and hence $\zeta_{k,N}(s)$ extends meromorphically to 
$\C$ 
with simple poles at 
$\{(n-2k)\mid k\in \N\}$. Consequently, the terms $\nu(\eta)^{-r}$ in the 
asymptotic expansion of $t_{\eta,\e}^{k}(\lambda)$ 
with $r=n-2k,k\in \N$, may lead to singular behavior of 
$T^k_{\e}(s,\lambda)$ at $s=0$. 
We regularize $T^k_{\e}(s,\lambda)$ by subtracting off these terms from 
$t_{\eta,\e}^{k}(\lambda)$, and define
\begin{equation}\label{P-k}
\begin{split}
f_{r,\e}^{k}(\lambda):=&(-1)^r\left(M_r(t_{\e}(\lambda),-\A_k)-
M_r(t_{\e}(\lambda),\A_k) + \left(\frac{\A_k^r-(-\A_k)^r}{r}\right) \right), \\
p_{\eta,\e}^{k}(\lambda):=&t_{\eta,\e}^{k}(\lambda)-\sum_{r=1}^{n}\nu(\eta)^{-r} 
f_{r,\e}^{k}(\lambda), \quad P^k_{\e}(s,\lambda):= \sum_{\eta\in E_k} \textup{m} 
\ (\eta)p_{\eta,\e}^{k}(\lambda)
\nu(\eta)^{-2s}.
\end{split}
\end{equation}
By construction, $P^k_{\e}(s,\lambda)$ is regular at $s=0$. The contribution of 
the terms $f_{r,\e}^{k}(\lambda)$ 
is computed in terms of the polynomials $M_r(t,\A)$ in \eqref{MD-polynom}. 
The computation uses special integrals evaluated already by 
Spreafico \cite{Spr:ZIF}.

\begin{prop}\label{ff}
\begin{align*}
\int_0^{\infty}t^{s-1}\frac{1}{2\pi i}\int_{\wedge_c}\frac{e^{-\lambda t}}
{-\lambda}\, f_{r,\e}^{k}(\lambda)\, d\lambda \, dt = (-1)^r\sum_{b=0}^{r}
(z_{r,b}(-\A_k)-z_{r,b}(\A_k))\frac{\Gamma(s+b+r/2)}{s\, \Gamma (b+r/2)}\, \e^{2s}.
\end{align*}
\end{prop}

\begin{proof} 
The $\lambda$-independent part of $f_{r,\e}^{k}(\lambda)$ vanishes after integration in $\lambda$. 
The coefficients $M_{r}(t_{\e}(\lambda),\pm\A_k)$ in the definition of $f_{r,\e}^{k}(\lambda)$ are polynomial 
in $t_{\e}(\lambda)=(1-\e^2\lambda)^{-1/2}$. Hence we compute, by substituting first $\mu=\e^2\lambda$, 
and then $\tau=t/\e^2$
\begin{align*}
\int_0^{\infty}t^{s-1}\frac{1}{2\pi i}\int_{\wedge_c}\frac{e^{-\lambda t}}
{-\lambda}\frac{1}{(1-\e^2\lambda)^a}\, d\lambda \, dt 
= \, \e^{2s}\, \int_0^{\infty}\tau^{s-1}\frac{1}{2\pi i}\int_{\wedge_{\e^2c}}
\frac{e^{-\mu \tau}}{-\mu}\frac{1}{(1-\mu)^a}\, d\mu \, d\tau 
\end{align*}
For the inner integral we obtain by substituting $z=\tau(\mu-1)$
\begin{align*}
\frac{1}{2\pi i}\int_{\wedge_{\e^2c}}\frac{e^{-\mu \tau}}{-\mu}\frac{1}{(1-\mu)^a}
d\mu= -\frac{1}{2\pi i}e^{-\tau} \tau^{a}\int_{\wedge}\frac{e^{-z}}{z+\tau}
(-z)^{-a}dz,
\end{align*}
where $\wedge \equiv \wedge_{\tau(\e^2c-1)}$. The contour of integration  
encircles 
a pole singularity $z=0$ of the integrand and the second pole at $z=-\tau$ 
lies outside the contour of integration. Hence we can deform the contour to 
start at infinity of the real axis, 
continue along real axis to some $\delta>0$, continue along the circle of 
radius $\delta$ around  the origin counter-clockwise, and then continue from 
$\delta$ back to infinity along the real axis.

The deformation does not change the value of the integral, and the deformed 
contour shall 
be denoted by $\mathscr{C}_{\delta}$, with its three components 
$\mathscr{C}_{\delta}^j,j=1,2,3,$ as in Figure \ref{contour} below. 
We can now evaluate the integral along each of these three components.

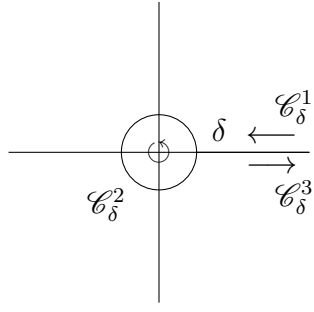
\begin{figure}[h]
\begin{center}
\begin{tikzpicture}
\draw (-2,0) -- (2,0);
\draw (0,-2) -- (0,2);

\draw (0.5,0) -- (2,0);
\draw (0,0) circle (0.5cm);

\node at (1.5,0.2) {$\longleftarrow$};
\node at (1.5,-0.2) {$\longrightarrow$};
\node at (0,0) {$\circlearrowleft$};

\node at (0.8,0.3) {$\delta$};
\node at (1.8,0.6) {$\mathscr{C}_{\delta}^1$};
\node at (1.8,-0.6) {$\mathscr{C}_{\delta}^3$};
\node at (-0.7,-0.7) {$\mathscr{C}_{\delta}^2$};

\end{tikzpicture}
\end{center}
\label{contour}
\caption{The deformed integration contour $\mathscr{C}_{\delta}$.}
\end{figure}

Note that the many-valued function $(-z)^{-a}$ is made definite by the 
convention 
$$(-z)^{-a}=e^{-a\log (-z)},$$
with the main branch of the logarithm. Along $\mathscr{C}_{\delta}^1$ we have 
$\arg (-z)=-\pi$, 
and along $\mathscr{C}_{\delta}^3$ we have $\arg (-z)=\pi$. Consequently, we find 
\begin{align*}
-\frac{1}{2\pi i}e^{-\tau} \tau^{a}\int\limits_{\mathscr{C}_{\delta}^1\, \cup \, \, \mathscr{C}_{\delta}^2}\frac{e^{-z}}{z+\tau} (-z)^{-a}dz
=\frac{\sin (a\pi)}{\pi}e^{-\tau} \tau^{a}\int_{\delta}^{\infty}\frac{e^{-z}}{z+\tau} z^{-a} dz 
\end{align*}
Assuming $\textup{Re}(a)<1$, the limits as $\delta \to 0$ for the integrals 
along each of the three components $\mathscr{C}_{\delta}^j,j=1,2,3,$ are well defined 
and in fact the integral along $\mathscr{C}_{\delta}^2$ vanishes in the limit. Consequently we obtain 
using \cite[8.353.3]{GraRyz:TOI} and assuming $\textup{Re}(a)<1$
\begin{align}\label{int1}
-\frac{1}{2\pi i}e^{-\tau} \tau^{a}\int_{\wedge}\frac{e^{-z}}{z+\tau}
(-z)^{-a}dz &= \frac{\sin (a\pi)}{\pi}e^{-\tau} \tau^{a}\int_{0}^{\infty}\frac{e^{-z}}{z+\tau} z^{-a} dz  \\
&= \frac{\sin (a\pi)}{\pi}\Gamma (a,\tau) \Gamma (1-a).\label{int2}
\end{align}
Since the left integral in \eqref{int1} and the expression \eqref{int2} are both analytic in 
$a\in \C$, the equality between the two in fact holds for any $a\in \C$ and the statement 
follows finally from the relation between the incomplete Gamma function and the probability integral
\begin{align}
\int_0^{\infty}\tau^{s-1}\, \frac{\Gamma(a,\tau)}{\Gamma (a)}\, d\tau=
\frac{\Gamma(s+a)}{s\, \Gamma (a)}.
\end{align} 
\end{proof}
Consequently we arrive at the intermediate 
representation of $\zeta_k(s,\e)$ for $\textup{Re}(s)\gg 0$
\begin{equation}\label{zeta-intermediate}
\begin{split}
\zeta_k(s,\e)&=
\frac{s^2}{\Gamma(s+1)}\int_0^{\infty}\frac{t^{s-1}}{2\pi i}
\int_{\wedge_{c}}\frac{e^{-\lambda t}}{-\lambda}\, 
P^k_{\e}(s,\lambda) d\lambda \, dt \\
&+ \sum_{r=1}^{n}\zeta_{k,N}(2s+r)\frac{(-1)^r s}{\Gamma (s+1)}
\sum_{b=0}^{r}(z_{r,b}(-\A_k)-z_{r,b}(\A_k))
\frac{\Gamma(s+b+r/2)}{\Gamma (b+r/2)}\, \e^{2s}
\end{split}
\end{equation}

While the second summand in \eqref{zeta-intermediate} 
extends meromorphically to $\C$, 
it still remains to derive an analytic extension to $s=0$ 
for the first summand.

\begin{prop}\label{AB} 
Let the notation be as in Proposition \ref{N-prop} and \eqref{P-k}. 
Then for large arguments $\lambda\in \Lambda_c$ and fixed order $\eta \in E_k$ we have 
the following asymptotics
\begin{align*}
p_{\eta,\e}^{k}(\lambda)= b^k_{\eta,\e}
+O\left((-\lambda)^{-1/2}\right),
\end{align*}
where 
\begin{align*}
b^k_{\eta, \e}=\log \left(1-\frac{\A_k}{\nu(\eta)}\right)-
\log \left(1+\frac{\A_k}{\nu(\eta)}\right) -\sum_{r=1}^n\frac{1}{(-\nu(\eta))^r} 
\left(\frac{\A_k^r-(-\A_k)^r}{r}\right).
\end{align*}
\end{prop}
\begin{proof}
The function $p_{\eta,\e}^{k}(\lambda)$ is given by the following expression
\begin{align*}
p_{\eta,\e}^{k}(\lambda)=t_{\eta,\e}^{k}(\lambda)-\sum_{r=1}^{n} 
\frac{1}{(-\nu(\eta))^r}\left(M_r(t_{\e}(\lambda),-\A_k)-
M_r(t_{\e}(\lambda),\A_k) + \left(\frac{\A_k^r-(-\A_k)^r}{r}\right) \right).
\end{align*}
The polynomials $M_{2r}(t_{\e}(\lambda),\pm \A_k)$ have no constant terms, and hence 
$M_{2r}(t_{\e}(\lambda),\pm \A_k)\sim O\left((-\lambda)^{-1/2}\right), \lambda \to \infty,$ since
\begin{align}
t_{\e}(\lambda)=\frac{1}{\sqrt{1-\e^2\lambda}}=O\left((-\lambda)^{-1/2}\right), 
\quad  \lambda \to \infty.
\end{align} 
By \eqref{large-arg-I} and \eqref{large-arg-K}, setting $\nu\equiv \nu(\eta)$
\begin{align*}
-&\log \left(\nu z \e K'_{\nu}(\nu z \e)+\A_k K_{\nu}(\nu z \e)\right)+
\log \left(\nu z \e K'_{\nu}(\nu z \e)-\A_k K_{\nu}(\nu z \e)\right) \\
\sim &\log \left(1+\frac{\A_k}{\nu z \e}\right) - \log \left(1-\frac{\A_k}
{\nu z \e}\right) + O\left((-\lambda)^{-1/2}\right) \sim 
O\left((-\lambda)^{-1/2}\right), \quad \lambda \to \infty.
\end{align*}
Moreover 
\begin{align*}
\frac{ \nu z K'_{\nu}(\nu z)\pm\A_k K_{\nu}(\nu z)}{ \nu z I'_{\nu}(\nu z)
\pm\A_k I_{\nu}(\nu z)}\cdot \frac{ \nu z \e I'_{\nu}(\nu z \e)\pm\A_k 
I_{\nu}(\nu z \e)}{\nu z \e K'_{\nu}(\nu z \e)\pm\A_k K_{\nu}(\nu z \e)}
\sim O(e^{2\nu z(\e-1)}), \lambda \to \infty.
\end{align*}
$(\e-1)<0$ and Re$(z)>0$ for large $z=\sqrt{-\lambda}, \lambda \in \Lambda_c$. Consequently 
$O(e^{2\nu z(\e-1)})$ is in particular of $O\left((-\lambda)^{-1/2}\right)$ asymptotics 
for $\lambda \to \infty, \lambda \in \Lambda_c$. By the explicit expression for $t_{\eta,\e}^{k}(\lambda)$ 
in \eqref{t-bessel} the statement follows.
\end{proof}

\begin{defn} \label{AB1} Define for $\textup{Re}(s)>n$ in notation of Proposition \ref{AB}
\begin{align}
B^k_{\e}(s):=\sum_{\eta \in E_k} \textup{m}(\eta) \ b^k_{\eta,\e}\, \nu(\eta)^{-2s}.
\end{align}
\end{defn}
$B^k_{\e}(s)$ converges at $s=0$ by construction, since $\zeta_{k,N}(s)$ converges 
for $Re(s)>n$.
\begin{prop}\label{P}
Let $P^k_{\e}(s,\lambda)$ be defined by \eqref{P-k}. Then
\begin{align*}
P^k_{\e}(s,0)=0.
\end{align*}
\end{prop}
\begin{proof}
By \eqref{DM}
\begin{align}
M_{r}(1,-\A_k)- M_{r}(1,\A_k)=\frac{(-\A_k)^r-\A_k^r}{r}.
\end{align}
For any fixed $\e>0$ clearly $\lambda \to 0$ implies that $t=(1-\e^2\lambda)^{-1/2}$ tends to $1$. 
Hence 
\begin{align*}
f_{r,\e}^{k}(\lambda)=(-1)^r\left(M_r(1,-\A_k)-
M_r(1,\A_k) + \left(\frac{\A_k^r-(-\A_k)^r}{r}\right) \right)=0.
\end{align*}
Moreover, by \eqref{small}
\begin{align*}
-\log \left(\nu z \e K'_{\nu}(\nu z \e)+\A_k K_{\nu}(\nu z \e)\right)+
\log \left(\nu z \e K'_{\nu}(\nu z \e)-\A_k K_{\nu}(\nu z \e)\right) \\
\sim \log \left(1+\frac{\A_k}{\nu}\right) - \log \left(1-\frac{\A_k}
{\nu}\right), \  \textup{as} \ \lambda \to 0.
\end{align*}
Moreover we have
\begin{align*}
\frac{ \nu z K'_{\nu}(\nu z)\pm\A_k K_{\nu}(\nu z)}{ \nu z I'_{\nu}(\nu z)\pm\A_k I_{\nu}
(\nu z)}\cdot \frac{ \nu z \e I'_{\nu}(\nu z \e)\pm\A_k I_{\nu}(\nu z \e)}{\nu z \e 
K'_{\nu}(\nu z \e)\pm\A_k K_{\nu}(\nu z \e)}\sim \e^{2\nu}, \   \textup{as} \ \lambda \to 0.
\end{align*}
By the explicit expression for $t_{\eta,\e}^{k}(\lambda)$ in \eqref{t-bessel} the statement 
follows. Note that the $\e-$dependence cancels.
\end{proof}

We can now put everything together and write down the meromorphic 
continuation to $s=0$
of the zeta-function $\zeta_k(s,\e)$, introduced in Proposition \ref{N-prop}. 
By the arguments of \cite[Section 4.1]{Spr:ZFA} we have
\begin{equation}\label{zeta-expression}
\begin{split}
\zeta_k(s,\e) &=  \frac{s}{\Gamma (s+1)}\, \left(
P^k_{\e}(s,0) - B^k_{\e}(s)\right)\\ 
& + \sum_{r=1}^{n}\zeta_{k,N}(2s+r)\frac{(-1)^r s}{\Gamma (s+1)}
\sum_{b=0}^{r}(z_{r,b}(-\A_k)-z_{r,b}(\A_k))
\frac{\Gamma(s+b+r/2)}{\Gamma (b+r/2)}\, \e^{2s} \\ &
+ \frac{s^2}{\Gamma (s+1)}\, h(s,\e),
\end{split}
\end{equation}
where $h(s,\e)$ vanishes with its derivative at $s=0$. Note that all the terms 
are regular
at $s=0$. Inserting the results of Proposition \ref{ff}, Proposition \ref{AB}, 
Proposition \ref{P} together with Definition \ref{AB1} into the expression 
\eqref{zeta-expression} we obtain the following

\begin{prop}\label{zeta-total} 
Let $(E_N,\nabla_N,h_N)$ be a flat Hermitian vector bundle over an 
even-dimensional 
oriented closed Riemannian manifold $(N^n,g^N)$.
Denote by $\Delta_{k,ccl,N}$ the Laplacian on coclosed $k-$differential forms 
$\Omega^k_{\textup{ccl}}(N,E_N)$. Let the notation be as in \eqref{MD-polynom} and 
\eqref{shifted-zeta}. Put
\begin{align*}
\A_k:=\frac{(n-1)}{2}-k, \quad 
\nu(\eta)=\sqrt{\eta + \A_k^2}, \ \textup{for} \ 
\eta \in E_k=\textup{Spec}\Delta_{k,ccl,N}\backslash \{0\}.
\end{align*}
Let $\textup{m}(\eta)$ denote the multiplicity of $\eta \in E_k$.
Then for $\e>0$ sufficiently small, $\zeta_k(s,\e)$ defined in Definition \ref{zetas}
admits an analytic continuation to $s=0$ of the form
\begin{equation}\label{zeta-continued}
\begin{split}
\zeta_k(s,\e)&=\frac{s}{\Gamma (s+1)}\left( \sum_{\eta \in E_k} \textup{m}(\eta) \ \nu(\eta)^{-2s}
\left(\log\left(1+\frac{\A_k}{\nu(\eta)}\right) +\sum_{r=1}^n\frac{(-\A_k)^r}{r\nu(\eta)^r}\right)\right) \\
&- \frac{s}{\Gamma (s+1)}\left( \sum_{\eta \in E_k} \textup{m}(\eta) \ \nu(\eta)^{-2s}\left(\log\left(1-\frac{\A_k}{\nu(\eta)}\right) + 
\sum_{r=1}^n \frac{\A_k^r}{r\nu(\eta)^r}\right)\right) \\
&+\sum_{r=1}^{n}\zeta_{k,N}(2s+r)\frac{(-1)^r s}{\Gamma (s+1)}\sum_{b=0}^{r}(z_{r,b}(-\A_k)-z_{r,b}(\A_k)) \\
&\times \frac{\Gamma(s+b+r/2)}{\Gamma (b+r/2)}\, \e^{2s} + \frac{s^2}{\Gamma (s+1)}h(s),
\end{split}
\end{equation}
where $h(s)$ vanishes with its derivative at $s=0$.
\end{prop}

Note the full analogy (up to computationally irrelevant, but geometrically crucial sign differences) 
to the corresponding result in \cite[Proposition 6.10]{Ver:ATO}. An ad verbatim repetition of the 
arguments in the proof of \cite[Corollary 6.1]{Ver:ATO} leads to the final formula.
\begin{cor}\label{total-contribution-1}
Let $(E_N,\nabla_N,h_N)$ be a flat Hermitian vector bundle over an even-dimensional 
oriented closed Riemannian manifold $(N^n,g^N)$.
Denote by $\Delta_{k,ccl,N}$ the Laplacian on coclosed $k-$differential forms 
$\Omega^k_{\textup{ccl}}(N,E_N)$ and put
\begin{align*}
&\A_k:=\frac{(n-1)}{2}-k, \quad 
\nu(\eta)=\sqrt{\eta + \A_k^2}, \ \textup{for} \ \eta \in E_k=\textup{Spec}\Delta_{k,ccl,N}\backslash \{0\}, \\
&\zeta_{k,N}(s)=\sum_{\eta \in E_k} \textup{m}(\eta) \ \nu(\eta)^{-s},\quad \zeta_{k,N}(s, \pm \A_k):=
\sum_{\eta\in E_k} \textup{m}(\eta) \ (\nu(\eta)\pm \A_k)^{-s}, \quad Re(s)\gg0,
\end{align*}
where $\textup{m}(\eta)$ denotes the multiplicity of $\eta \in E_k$.
Then we find in notation of \eqref{MD-polynom} for $\e>0$ sufficiently small
\begin{align*}
\zeta_k'(0,\e) =& \,
\frac{1}{2}\sum_{r=1}^{n/2}\underset{s=2r}{\textup{Res}}\, \zeta_{k,N}(s)\sum_{b=0}^{2r}
\left(z_{2r,b}(-\A_k)-z_{2r,b}(\A_k)\right)\frac{\Gamma'(b+r)}{\Gamma (b+r)} \\
+&\, \zeta_{k,N}'(0,-\A_k)-\zeta_{k,N}'(0,\A_k).
\end{align*}
\end{cor}

\begin{proof}
We follow the approach of \cite[Section 11]{BKD:HKA}. Define
\begin{align*}
K(s,\pm \A_k):=\sum_{\eta \in E_k} \textup{m}(\eta) \ \nu(\eta)^{-2s}\left(-\log\left(1\pm\frac{\A_k}{\nu(\eta)}\right)-
\sum_{r=1}^n\frac{(\mp\A_k)^r}{r\nu(\eta)^r}\right).
\end{align*}
The series $K(0,\pm \A_k)$ converges absolutely, since $\zeta_{k,N}(s)$ converges absolutely for 
$Re(s)\geq n$. In order to evaluate  $K(0,\pm \A_k)$, define
\begin{equation}\label{exp0}
\begin{split}
K_0(s,\pm \A_k)&:= \sum_{\eta \in E_k} \textup{m}(\eta) \ \int_0^{\infty}t^{s-1}e^{-\nu(\eta) t}\left(e^{\mp\A_k t}-
\sum_{r=0}^n\frac{(\mp\A_kt)^r}{r!}\right) dt\\
&=\Gamma(s)\, \zeta_{k,N}(s,\pm\A_k)-\sum_{r=0}^n\frac{(\mp\A_k)^r}{r!}\, 
\Gamma (s+r)\, \zeta_{k,N}(s+r).
\end{split}
\end{equation}
$K_0(s,\pm \A_k)$ is an absolutely convergent sum at $s=0$, since $\zeta_{k,N}(s)$ converges absolutely for 
$Re(s)\geq n$. By construction 
\begin{align}
K(0,\pm \A_k)=K_0(0,\pm \A_k).
\end{align}
Furthermore we find from \eqref{exp0} and regularity of $K_0(s,\pm \A_k)$ at $s=0$
\begin{align}\label{zeta-exp}
\zeta_{k,N}(0,\pm\A_k)= \zeta_{k,N}(0) + \sum_{r=1}^n\frac{(\mp\A_k)^r}{r}
\underset{s=r}{\textup{Res}}\, \zeta_{k,N}(s).
\end{align}
Using \eqref{zeta-exp} we obtain following expansion at $s=0$
\begin{equation}\label{exp1}
\begin{split}
\Gamma(s)\, \zeta_{k,N}(s,\pm\A_k)&\sim  \left(\frac{1}{s}-\gamma+O(s)\right)
\left(\,  \zeta_{k,N}(s,\pm\A_k)- \zeta_{k,N}(0,\pm\A_k)
\, \right) \\  &+\left(\frac{1}{s}-\gamma+O(s)\right) \zeta_{k,N}(0,\pm\A_k) \sim \zeta'_{k,N}(0,\pm\A_k) \\
&+ \left(\frac{1}{s}-\gamma\right)
\left(\zeta_{k,N}(0) + \sum_{r=1}^n\frac{(\mp\A_k)^r}{r}
\underset{s=r}{\textup{Res}}\, \zeta_{k,N}(s)\right) + O(s).
\end{split}
\end{equation}
Similarly we find
\begin{align}\label{exp2}
\Gamma(s)\, \zeta_{k,N}(s)\sim \zeta'_{k,N}(0)+ \left(\frac{1}{s}-\gamma\right) \zeta_{k,N}(0) + O(s), \ \textup{as} \ s\to 0.
\end{align}
Moreover, denoting by PP$\zeta_{k,N}(r)$ the constant term in the asymptotics of $\zeta_{k,N}(s)$ near the 
pole singularity $s=r$, we compute
\begin{equation}\label{exp3}
\begin{split}
\sum_{r=1}^n\frac{(\mp\A_k)^r}{r!}\, \Gamma (s+r)\, \zeta_{k,N}(s+r) &\sim 
\sum_{r=1}^n\frac{(\mp\A_k)^r}{r!}\, \frac{\Gamma (s+r)}{s}\, \underset{s=r}{\textup{Res}}\, \zeta_{k,N}(s) \\
&+ \sum_{r=1}^n\frac{(\mp\A_k)^r}{r} \, \textup{PP}\zeta_{k,N}(r) + O(s), \ \textup{as} \ s\to 0.
\end{split}
\end{equation}
Plugging \eqref{exp1}, \eqref{exp2} and \eqref{exp3} into \eqref{exp0} we arrive at the following
\begin{equation}\label{exp4}
\begin{split}
K(0,\pm \A_k)&=K_0(0,\pm \A_k)=\zeta_{k,N}'(0,\A)
-\zeta_{k,N}'(0) \\ &- \sum_{r=1}^n\frac{(\mp\A_k)^r}{r}\left(\underset{s=r}{\textup{Res}}\,\zeta_{k,N}(s)
\left(\!\gamma +\frac{\Gamma'(r)}{\Gamma(r)}\!\, \right)+\textup{PP}\zeta_{k,N}(r)\right),
\end{split}
\end{equation}
This result corresponds to the result obtained in \cite[p. 388]{BKD:HKA}, up to certain factors due to 
a different notation. Furthermore, we compute straightforwardly
\begin{align*}
\left.\frac{d}{ds}\right|_{0}\!\zeta_{k,N}(2s+r)\frac{s}{\Gamma (s+1)}
\frac{\Gamma\left(s+b+\frac{r}{2}\right)}{\Gamma (b+\frac{r}{2})}=\frac{1}{2}\,
\underset{s=r}{\textup{Res}}\, \zeta_{k,N}(s)\left[\frac{\Gamma'\left(b+\frac{r}{2}\right)}
{\Gamma \left(b+\frac{r}{2}\right)}+
\gamma \right]+\textup{PP}\zeta_{k,N}(r).
\end{align*}
Finally, note by \eqref{MD-polynom} and \eqref{DM} 
\begin{align}\label{zrb}
\sum_{b=0}^{r}(z_{r,b}(-\A_k)-z_{r,b}(\A_k))=M_r(1,-\A_k)-M_r(1,\A_k)=\frac{(-\A_k)^r-\A_k^r}{r}.
\end{align}
Differentiating \eqref{zeta-continued}, we 
arrive at the result; note that $\e-$dependence cancels, since \eqref{zrb} vanishes for $r$ even, whereas on 
the even dimensional closed Riemannian manifold $(N^n,g^N)$ the residue 
$\underset{s=r}{\textup{Res}}\, \zeta_{k,N}(s)$ vanishes for $r$ odd. 
\end{proof}

\begin{prop}\label{h-truncated}
Let $(E_N,\nabla_N,h_N)$ be a flat Hermitian vector bundle over an even-dimensional 
oriented closed Riemannian manifold $(N^n,g^N)$. Denote the Euler characteristic of 
$(N,E_N)$ by $\chi(N,E_N)$ and the Betti numbers by $b_k:=\dim H^k(N,E_N)$. Then 
in notation of Definition \ref{zetas} we find
\begin{equation}
\begin{split}
\sum_{k=0}^{n}\frac{(-1)^{k+1}}{2}\, \zeta'_{k,H}(0,\e)&=
\sum_{k=0}^n\frac{(-1)^{k}}{2} \, b_k \, \log \left( \frac{1-\e^{n-2k+1}}{n-2k+1}\right)
\\ &+\sum_{k=0}^{n/2-1}(-1)^kb_k \sum_{l=0}^{n/2-k-1}\log (2l+1)\\
&+\sum_{k=0}^{n/2-1}\frac{(-1)^{k}}{2}\, b_k\log (n-2k+1)
\end{split}
\end{equation}
\end{prop}

\begin{proof}
By Definition \ref{zetas} we can write
\begin{equation}
\begin{split}
\sum_{k=0}^{n}\frac{(-1)^{k+1}}{2}\, \zeta_{k,H}(s,\e)&=
\sum_{k=0}^n\frac{(-1)^{k+1}}{2} \, b_k \, \zeta(s, H^k_{0,\e,\textup{rel}})\\ 
&-\sum_{k=0}^n\frac{(-1)^{k+1}}{2} \, b_k \, \zeta(s, H^k_{0,\textup{rel}})=:H(s,\e)-H(s).
\end{split}
\end{equation}
$H'(0)$ has been evaluated in \cite[Theorem 7.8]{Ver:ATO} with 
\begin{equation}\label{H1}
\begin{split}
\sum_{k=0}^n\frac{(-1)^{k+1}}{2}\, b_k \, \zeta'(0, H^k_{0,\textup{rel}})=& \,
\frac{\log 2}{2}\chi(N,E_N)-\sum_{k=0}^{n/2-1}\frac{(-1)^k}{2}\, b_k\, \log (n-2k+1)\\
&-\sum_{k=0}^{n/2-1}(-1)^k\, b_k \sum_{l=0}^{n/2-k-1}\log (2l+1).
\end{split}
\end{equation}
We evaluate $H'(0,\e)$ using \cite[Theorem 1.2]{Les:DOR}, which relates the zeta determinants 
to the normalized solutions of the operators, satisfying the corresponding boundary conditions. 
The boundary conditions for $H^k_{0,\e,\textup{rel}}$ 
have been determined in Proposition \ref{rel-bc-prop}, and are given by the Dirichlet boundary conditions. 
By the formula \cite[Theorem 1.2]{Les:DOR} we then find
\begin{align}
\det\nolimits_{\zeta} \left(H^k_{0,\e,\textup{rel}}\right) = 
\frac{\sqrt{\e}}{|\A_k|}\left(\e^{-|\A_k|}-\e^{|\A_k|}\right).
\end{align}
Taking logarithms and employing Poincare duality on $(N,g^N)$ we find
\begin{equation}
\begin{split}
H'(0,\e) &= \sum_{k=0}^n\frac{(-1)^{k}}{2}\, b_k \, \log \det\nolimits_{\zeta} \left(H^k_{0,\e,\textup{rel}}\right) \\
&=\sum_{k=0}^n\frac{(-1)^{k}}{2} \, b_k \, \log \left(\frac{\sqrt{\e}}{|\A_k|}
\left(\e^{-|\A_k|}-\e^{|\A_k|}\right) \right)\\
&=\sum_{k=0}^n\frac{(-1)^{k}}{2} \, b_k \, \log \left(\frac{\sqrt{\e}}{|\A_k+1|}
\left(\e^{-|\A_k+1|}-\e^{|\A_k+1|}\right) \right).
\end{split}
\end{equation}
Obviously, we can replace $|\A_k+1|$ by $(\A_k+1)$ in the expression above, and find 
after straightforward cancellations
\begin{equation}\label{H2}
\begin{split}
H'(0,\e) &= \sum_{k=0}^n\frac{(-1)^{k}}{2} \, b_k \, \log \left(\frac{\sqrt{\e}}{(\A_k+1)}
\left(\e^{-(\A_k+1)}-\e^{(\A_k+1)}\right) \right)\\
&=\sum_{k=0}^n\frac{(-1)^{k}}{2} \, b_k \, \log \left(2\e^{k-n/2}
\left(\frac{1-\e^{n-2k+1}}{n-2k+1}\right) \right)\\
&=\sum_{k=0}^n\frac{(-1)^{k}}{2} \, b_k \, \log 
\left( \frac{1-\e^{n-2k+1}}{n-2k+1}\right) + \frac{\log 2}{2} \chi(N,E_N).
\end{split}
\end{equation}
The statement follows by combination of \eqref{H1} and \eqref{H2}.
\end{proof}

Summing up the expressions in Corollary \ref{total-contribution-1} and Proposition \ref{h-truncated}, 
we arrive at the following result.
\begin{thm}\label{t-difference}
Let $(C(N)=(0,1]\times N, g=dx^2\oplus x^2g^N)$ be an odd-dimensional 
bounded cone over a closed oriented Riemannian manifold $(N^n,g^N)$. 
Denote by $(C_{\e}(N)=[\e,1]\times N, g)$ 
its truncation. Let $(E,\nabla, h^E)$ be a flat complex Hermitian vector bundle 
over $(C(N),g)$ and $(E_N,\nabla_N,h_N)$ its restriction 
to the cross-section $N$. Denote by $\chi(N,E_N)$ the Euler characteristic and by $b_k:=\dim H^k(N,E_N)$ 
the Betti numbers of $(N,E_N)$. Denote by $\Delta_{k,ccl,N}$ the Laplacian on coclosed $k-$differential forms 
$\Omega^k_{\textup{ccl}}(N,E_N)$ and put
\begin{align*}
&\A_k:=\frac{(n-1)}{2}-k, \quad 
\nu(\eta)=\sqrt{\eta + \A_k^2}, \ \textup{for} \ \eta \in E_k=\textup{Spec}\Delta_{k,ccl,N}\backslash \{0\}, \\
&\zeta_{k,N}(s)=\sum_{\eta \in E_k} \textup{m}(\eta) \ \nu(\eta)^{-s},\quad \zeta_{k,N}(s, \pm \A_k):=
\sum_{\eta\in E_k} \textup{m}(\eta) \ (\nu(\eta)\pm \A_k)^{-s}, \quad Re(s)\gg0,
\end{align*}
where $\textup{m}(\eta)$ denotes the multiplicity of $\eta \in E_k$.
Then the difference of the scalar analytic torsions for $(C(N),g)$ and 
$(C_{\e}(N), g)$ is given by the following explicit expression
\begin{align*}
\log &T(C_{\e}(N), E, g)-\log T(C(N), E, g)= 
\sum_{k=0}^n\frac{(-1)^{k}}{2} \, b_k \, \log \left( \frac{1-\e^{n-2k+1}}{n-2k+1}\right)
\\ &+\sum_{k=0}^{n/2-1}(-1)^kb_k \sum_{l=0}^{n/2-k-1}\log (2l+1) + \sum_{k=0}^{n/2-1}\frac{(-1)^k}{2}\, b_k \log (n-2k+1)
\\&+\sum_{k=0}^{n/2-1}\frac{(-1)^k}{4}\sum_{r=1}^{n/2}\underset{s=2r}{\textup{Res}}\, \zeta_{k,N}(s)\sum_{b=0}^{2r}\left(z_{2r,b}(-\A_k)-z_{2r,b}(\A_k)\right)\frac{\Gamma'(b+r)}{\Gamma (b+r)} \\
&+ \sum_{k=0}^{n/2-1}\frac{(-1)^k}{2}
\left( \zeta_{k,N}'(0,-\A_k)-\zeta_{k,N}'(0,\A_k)\right).
\end{align*}
\end{thm}

Comparison of Theorem \ref{t-difference} and Theorem \ref{BV-Theorem} yields the following
\begin{cor}\label{t1}
Let $(C_{\e}(N)=[\e,1]\times N, g,g_0),\e>0,$ be an odd-dimensional cylinder over 
a closed Riemannian manifold $(N,g^N)$, with a pair of Riemannian 
metrics $g=dx^2\oplus x^2g^N$ and $g_0=dx^2\oplus g^N$. 
The Riemannian manifold $(C_{\e}(N),g)$ 
is a truncated cone, while $(C_{\e}(N),g_0)$ is an exact cylinder. 
Fix a flat complex Hermitian vector bundle $(E,\nabla,h^E)$. Then 
\begin{equation}\label{scalar-T}
\begin{split}
&\log T(C_{\e}(N), E, g)=\sum_{k=0}^n\frac{(-1)^{k}}{2} \, b_k \, \log 
\left( \frac{1-\e^{n-2k+1}}{n-2k+1}\right) + \frac{\log 2}{2} \chi(N,E_N) \\
&+\sum_{k=0}^{n/2-1}\frac{(-1)^k}{2}\sum_{r=1}^{n/2}\underset{s=2r}{\textup{Res}}\, \zeta_{k,N}(s)\sum_{b=0}^{2r}\left(z_{2r,b}(-\A_k)-z_{2r,b}(\A_k)\right)\frac{\Gamma'(b+r)}{\Gamma (b+r)},
\end{split}
\end{equation}
and the analytic torsion norms of $(C_{\e}(N), g)$ and 
$(C_{\e}(N), g_{0})$ are related as follows
\begin{equation}\label{e-independence}
\begin{split}
\log \left(\frac{\|\cdot \|^{RS}_{(C_{\e}(N), E, g)}}{\|\cdot \|^{RS}_{(C_{\e}(N), E,
g_0)}}\right)&=\sum_{k=0}^{n/2-1}\frac{(-1)^k}{2}\sum_{r=1}^{n/2}\underset{s=2r}{\textup{Res}}\, 
\zeta_{k,N}(s) \\ & \times \sum_{b=0}^{2r}\left(z_{2r,b}(-\A_k)-z_{2r,b}(\A_k)\right)\frac{\Gamma'(b+r)}{\Gamma (b+r)}.
\end{split}
\end{equation}
\end{cor}

\begin{proof}
The first relation \eqref{scalar-T} follows by a direct comparison of 
Theorem \ref{t-difference} and Theorem \ref{BV-Theorem}. For the second relation note, that 
by definition of analytic torsion norms we have by \eqref{scalar-T}
\begin{align*}
\log \left(\frac{\|\cdot \|^{RS}_{(C_{\e}(N), E; g)}}{\|\cdot \|^{RS}_{(C_{\e}(N), E;
g_0)}}\right)&=\sum_{k=0}^n\frac{(-1)^{k}}{2} \, b_k \, \log 
\left( \frac{1-\e^{n-2k+1}}{n-2k+1}\right) - \frac{1}{2}\chi(N,E_N) \log (1-\e) \\
&+\sum_{k=0}^{n/2-1}\frac{(-1)^k}{2}\sum_{r=1}^{n/2}\underset{s=2r}{\textup{Res}}\, \zeta_{k,N}(s)\sum_{b=0}^{2r}\left(z_{2r,b}(-\A_k)-z_{2r,b}(\A_k)\right)\frac{\Gamma'(b+r)}{\Gamma (b+r)}
\\ &+ \log \left(\frac{\|\cdot \|_{\det H^*(C_{\e}(N), E), g}}
{\|\cdot \|_{\det H^*(C_{\e}(N), E), g_0}}\right)
\end{align*}
The quotient between the norms on $\det H^*(C_{\e}(N), E)$, induced by the $L^2(g,h^E)$ and 
$L^2(g_{0},h^E)$ norms on harmonic forms, amounts by a straightforward computation to 
\begin{align*}
\log \left(\frac{\|\cdot \|_{\det H^*(C_{\e}(N), E), g}}
{\|\cdot \|_{\det H^*(C_{\e}(N), E), g_0}}\right)=\frac{1}{2}\chi(N,E_N) \log (1-\e)
-\sum_{k=0}^n\frac{(-1)^{k}}{2} \, b_k \, \log 
\left( \frac{1-\e^{n-2k+1}}{n-2k+1}\right).
\end{align*}
Hence overall we arrive at the following
\begin{equation*}
\begin{split}
\log \left(\frac{\|\cdot \|^{RS}_{(C_{\e}(N), E; g)}}{\|\cdot \|^{RS}_{(C_{\e}(N), E;
g_0)}}\right)=\sum_{k=0}^{n/2-1}\frac{(-1)^k}{2}\sum_{r=1}^{n/2}\underset{s=2r}{\textup{Res}}\, 
\zeta_{k,N}(s)\sum_{b=0}^{2r}\left(z_{2r,b}(-\A_k)-z_{2r,b}(\A_k)\right)\frac{\Gamma'(b+r)}{\Gamma (b+r)}.
\end{split}
\end{equation*}
\end{proof}

\section{Metric Anomaly at the Regular Boundary of the Cone}\label{section-anomaly}

Consider a truncated cone $C_{\e}(N)=[\e,1]\times N, g=dx^2\oplus x^2 g^N$ 
over a closed oriented Riemannian manifold $(N^n,g^N)$. The Levi-Civita connection $\nabla^{TC_{\e}(N)}$, 
induced by $g$, defines secondary classes $B_{\e}(\nabla^{TC_{\e}(N)})$ and $B_{1}(\nabla^{TC_{\e}(N)})$ 
at the left  $\{x=\e\}\times N$ and the right  $\{x=1\}\times N$ boundary components of $C_{\e}(N)$, respectively.

Introducing new coordinates $y=\log (1/x)$ near the right boundary component $\{x=1\}\times N$ of $C_{\e}(N)$, 
and $z=\log (x/\e)$ near the left boundary component $\{x=\e\}\times N$, 
we can write for $\delta >0$ small the Riemannian metric $g$ as follows
\begin{equation}
\begin{split}
g=e^{-2y}\left(dy^2+g^N\right), \ y \in [0,\delta), \ \textup{near $\{x=1\}\times N$ of $C_{\e}(N)$}, \\
g=\e^2e^{2z}\left(dz^2+g^N\right), \ z \in [0,\delta), \ \textup{near $\{x=\e\}\times N$ of $C_{\e}(N)$}. 
\end{split}
\end{equation}
By Proposition \ref{Scaling} and in view of \eqref{RS2} and the explicit formulae in \eqref{RS3}, we deduce
\begin{align}\label{2B}
B_{\e}(\nabla^{TC_{\e}(N)})=B_{1}(\nabla^{TC_{\e}(N)})=:B_1(g^N),
\end{align}
independent of $\e>0$. Note, that the metric anomalies at $x=\e$ and $x=1$ are defined 
with respect to different inward unit normal vectors at the boundary, which amounts to an additional 
sign only in case of an odd-dimensional cross section. In our setup of 
even-dimensional cross-section, 
both anomalies coincide. Since the corresponding secondary classes induced 
by the product metric on $C_{\e}(N)$ are zero, we arrive by 
\eqref{BM-thm} at the following proposition.

\begin{prop}\label{BM+}
Let $(C_{\e}(N)=[\e,1]\times N, g,g_0),\e>0,$ be an odd-dimensional cylinder 
over 
a closed Riemannian manifold $(N,g^N)$, with a pair of Riemannian 
metrics $g=dx^2\oplus x^2g^N$ and $g_0=dx^2\oplus g^N$. 
The Riemannian manifold $(C_{\e}(N),g)$ 
is a truncated cone, while $(C_{\e}(N),g_0)$ is an exact cylinder. 
Fix a flat complex Hermitian vector bundle $(E,\nabla,h^E)$. 
Then the analytic torsion norms are related as follows
\begin{align}
\log \left(\frac{\|\cdot \|^{RS}_{(C_{\e}(N), E; g)}}
{\|\cdot \|^{RS}_{(C_{\e}(N), E; g_0)}}\right)=-\textup{rank}(E) \int_N B_{1}(g^N).
\end{align}
\end{prop}

Comparing Corollary \ref{t1} and Proposition \ref{BM+}, we arrive at the following result.
\begin{cor}\label{t2}
Let $(C(N)\cong (0,1)\times N, g=dx^2\oplus x^2g^N)$ be an 
odd-dimensional bounded cone over a closed oriented Riemannian manifold 
$(N^n,g^N)$. Let $(E,\nabla, h^E)$ be a flat complex Hermitian vector bundle and $(E_N,\nabla_N,h_N)$
its restriction to the cross-section $N$ over $C(N)$.
Then, in the notation of Section \ref{section-truncated}, the integral of the secondary class
$B_{1}(g^N)$ may be expressed as follows
\begin{align*}
\textup{rank}(E) \int_N B_{1}(g^N) = \sum_{k=0}^{n/2-1}\frac{(-1)^{k+1}}{2}
\sum_{r=1}^{n/2}\underset{s=2r}{\textup{Res}}\, \zeta_{k,N}(s)
\sum_{b=0}^{2r}\left(z_{2r,b}(-\A_k)-z_{2r,b}(\A_k)\right)\frac{\Gamma'(b+r)}{\Gamma (b+r)}.
\end{align*}
\end{cor}

\begin{remark}
The quotient of analytic torsion norms in \eqref{e-independence} is independent of $\e>0$, which 
also implies scaling invariance of the secondary class $B_{1}(\nabla^{TC_{\e}(N)})$ and the statement 
of Proposition \ref{BM+} follows, independently of the general result in Proposition \ref{Scaling}.
\end{remark}

\begin{example}
The following example might be illuminating for the identity in Corollary \ref{t2}.
Consider the special case of $N=T^2$ being the two-dimensional flat torus, and 
$E$ a trivial line bundle. Then one can verify Corollary \ref{t2} by direct computations.
Indeed, from \cite[(3.6), (3.7)]{BGKE:ZFD} we infer
\begin{equation}
\begin{split}
M_2(t,\A)=\sum_{b=0}^1 z_{2,b}(\A)t^{2+2b}= 
\left(-\frac{3}{16}+ \frac{\A}{2}-\frac{\A^2}{2}\right)t^2 +
\left(\frac{5}{8}-\frac{\A}{2}\right)t^4 - \frac{7}{16}t^6.
\end{split}
\end{equation}
Consequently, with $n=2$ and $\A_0=1/2$, we find
\begin{equation}
\begin{split}
&\left(z_{2,0}(-1/2)-z_{2,0}(1/2)\right)\frac{\Gamma'(1)}{\Gamma (1)}=
 -\frac{1}{2}\frac{\Gamma'(1)}{\Gamma (1)}=\frac{\gamma}{2}, \\
&\left(z_{2,1}(-1/2)-z_{2,1}(1/2)\right)\frac{\Gamma'(2)}{\Gamma (2)}= 
\frac{1}{2}\frac{\Gamma'(2)}{\Gamma (2)}=\frac{1}{2}(1-\gamma), \\
&\left(z_{2,2}(-1/2)-z_{2,2}(1/2)\right)\frac{\Gamma'(3)}{\Gamma (3)}= 0.
\end{split}
\end{equation}
Moreover, the heat trace expansion for the flat two-dimensional torus implies
$$
\underset{s=2}{\textup{Res}}\, \zeta_{0,N}(s) = 
2 \cdot \underset{s=1}{\textup{Res}}\, \zeta(s, \Delta_{0,N})
= 2 \cdot (4\pi)^{-1}\textup{Vol}(T^2).
$$
In total we find for the right hand side of the equality 
in Corollary \ref{t2}
\begin{equation}\label{RHS}
\begin{split}
\sum_{k=0}^{n/2-1}\frac{(-1)^{k+1}}{2}&\sum_{r=1}^{n/2}\underset{s=2r}{\textup{Res}}\, \zeta_{k,N}(s)
\sum_{b=0}^{2r}\left(z_{2r,b}(-\A_k)-z_{2r,b}(\A_k)\right)\frac{\Gamma'(b+r)}{\Gamma (b+r)} \\
&= -\frac{1}{8\pi}\textup{Vol}(T^2).
\end{split}
\end{equation}
On the other hand, the expression for the integral of $B_{1}(g^N)$ follows from \cite[(4.43)]{BruMa:AAF},
which in our special case reduces to 
\begin{align}\label{LHS}
\int_N B_{1}(g^N)=-\frac{1}{2\pi}\left(\frac{1}{2}\right)^2\textup{Vol}(T^2)=-\frac{1}{8\pi}\textup{Vol}(T^2).
\end{align}
Both, \eqref{RHS} and \eqref{LHS} agree, as asserted by Corollary \ref{t2}.
\end{example}

Corollary \ref{t2} together with Theorem \ref{BV-Theorem} 
leads to the main result of this section, announced in Theorem \ref{main2}. 
\begin{thm}\label{main1+}
Let $(C(N)\cong (0,1)\times N, g=dx^2\oplus x^2g^N)$ be an 
odd-dimensional bounded cone over a closed oriented Riemannian manifold 
$(N^n,g^N)$. Let $(E,\nabla, h^E)$ be a flat complex Hermitian vector bundle and $(E_N,\nabla_N,h_N)$
its restriction to the cross-section $N$ over $C(N)$.
Let $b_k=\dim H^k(N,E_N)$ be the Betti numbers and $\chi(N,E_N)$ the 
Euler characteristic of $(N,E_N)$. Denote by $\Delta_{k,ccl,N}$ the 
Laplacian on coclosed $k-$differential forms $\Omega^k_{\textup{ccl}}(N,E_N)$. Put $\A_k=(n-1)/2-k$ and define
\begin{align*}
&\nu(\eta)=\sqrt{\eta + \A_k^2}, \ \textup{for} \ \eta \in E_k=\textup{Spec}\Delta_{k,ccl,N}\backslash \{0\}, \\
&\zeta_{k,N}(s, \pm \A_k):=\sum_{\eta\in E_k} \textup{m}(\eta) \ (\nu(\eta)\pm \A_k)^{-s}, \quad Re(s)\gg0,
\end{align*}
where $\textup{m}(\eta)$ denotes the multiplicity of $\eta \in E_k$.
Then the logarithm of the scalar analytic torsion of $(C(N), g)$, is given by\begin{align*}
\log T(C(N), E, g)&= 
\sum_{k=0}^{n/2-1}\frac{(-1)^{k+1}}{2}\, b_k\left(\sum_{l=0}^{n/2-k-1}\log (2l+1)^2+\log (n-2k+1)\right) \\
& +\sum_{k=0}^{n/2-1}\frac{(-1)^k}{2} \left(\zeta_{k,N}'(0,\A_k)-\zeta_{k,N}'(0,-\A_k)\right)\\
& +\frac{\log 2}{2}\, \chi(N,E_N) - \frac{1}{2}\, \textup{rank}(E) \int_N B_{1}(g^N).
\end{align*}
\end{thm}
Theorem \ref{main1+} identifies the residual term in the formula for analytic torsion of a bounded 
cone in Theorem \ref{BV-Theorem} in terms of the metric anomaly of analytic torsion at the regular boundary of 
the cone. This identifies the actual contribution of the conical singularity to the analytic torsion, clearing up the formula 
in Theorem \ref{BV-Theorem} of the contributions from the regular boundary. 
We can rewrite Theorem \ref{main1+} in terms of the analytic torsion norm.
\begin{cor}\label{main2+}
Let $C(N)=(0,1]\times N, g=dx^2\oplus x^2g^N$ be an odd-dimensional bounded 
cone over a closed oriented Riemannian manifold $(N^n,g^N)$. Consider $(C(N),g_0)$, where 
$g_{\textup{pr}}$ coincides with $g$ near the singularity at $x=0$ and is product $dx^2\oplus g^N$ 
near the boundary $\{x=1\}\times N$. 
Let $(E,\nabla, h^E)$ be a flat complex Hermitian vector bundle and $(E_N,\nabla_N,h_N)$
its restriction to the cross-section $N$ over $C(N)$.
Then the quotient of analytic torsion norm for $(C(N),g_{\textup{pr}})$ and the 
$L^2(g,h^E)-$induced norm on $\det H^k(C(N),E)$ is given by
\begin{align*}
\log \left(\frac{\|\cdot \|^{RS}_{(C(N), E; g_{\textup{pr}})}}{\|\cdot \|_{\det H^k(C(N),E), g}}\right) = 
\sum_{k=0}^{n/2-1}(-1)^{k+1}\frac{b_k}{2}\left(\sum_{l=0}^{n/2-k-1}\log (2l+1)^2+\log (n-2k+1)\right)\\ 
+\, \frac{\log 2}{2}\, \chi(N,E_N) +\sum_{k=0}^{n/2-1}\frac{(-1)^k}{2} \left(\zeta_{k,N}'(0,\A_k)-\zeta_{k,N}'(0,-\A_k)\right).
\end{align*}
\end{cor}

\begin{proof}
The metric variation for analytic torsion is local by the gluing formula in \cite{Les:GFA} 
and hence the metric anomaly formula of Br\"uning-Ma in\cite{BruMa:AAF} holds also in case of 
manifolds with isolated conical singularities away from the variation region, so that 
\eqref{BM-thm} applies to $(C(N), g, g_{\textup{pr}})$. 
\begin{align*}
\log \left(\frac{\|\cdot \|^{RS}_{(C(N), E; g_{\textup{pr}})}}{\|\cdot \|_{\det H^k(C(N),E), g}}\right)
&= \log \left(\frac{\|\cdot \|^{RS}_{(C(N), E; g_{\textup{pr}})}}{\|\cdot \|^{RS}_{(C(N), E;
g)}}\right) +\log T(C(N), E; g) \\
&= \log T(C(N), E; g)+ \frac{1}{2}\, \textup{rank}(E) \int_N B_{1}(g^N).
\end{align*}
The claim follows by Theorem \ref{main1+}.
\end{proof}

\section{Asymptotics of the New Torsion-Like Spectral Invariant}
\label{section-torsion-like}

Analytic torsion defines a topological invariant of an odd-dimensional closed oriented Riemannian manifold 
with a flat Hermitian vector bundle. In even dimensions, analytic torsion is trivial as a consequence of Poincare duality. 
Theorem \ref{main1+} identifies the contribution of a conical singularity to analytic torsion in terms of
a new torsion-like spectral invariant of the even-dimensional cross-section, which is non-trivial despite Poincare 
duality and deserves an independent definition.

\begin{defn}
Let $(N^n,g^N)$ be an even-dimensional closed oriented Riemannian manifold and $(E_N,\nabla_N,h)$ a flat 
flat complex Hermitian vector bundle. Denote by $\Delta_{k,ccl,N}$ the corresponding
Laplacian on coclosed $k-$differential forms $\Omega^k_{\textup{ccl}}(N,E_N)$. Put $\A_k=(n-1)/2-k$ and define
\begin{align*}
&\nu(\eta)=\sqrt{\eta + \A_k^2}, \ \textup{for} \ \eta \in E_k=\textup{Spec}\Delta_{k,ccl,N}\backslash \{0\}, \\
&\zeta_{k,N}(s, \pm \A_k):=\sum_{\eta\in E_k} \textup{m}(\eta) \ (\nu(\eta)\pm \A_k)^{-s}, \quad Re(s)\gg0,
\end{align*}
where $\textup{m}(\eta)$ denotes the multiplicity of $\eta \in E_k$.
Then the "torsion-like" invariant $\textup{Tors}(N,E_N;g^N)$ is defined as
\begin{align*}
 \textup{Tors}(N,E_N; g^N):=\frac{1}{2}\sum_{k=0}^{n/2-1}(-1)^k \left(\zeta_{k,N}'(0,\A_k)-
\zeta_{k,N}'(0,-\A_k)\right)=\frac{1}{2}\sum_{k=0}^{n-1}(-1)^k \zeta_{k,N}'(0,\A_k).
\end{align*}
\end{defn} 

We study the asymptotic behavior of $\textup{Tors}(N,E_N; g^N)$ under scaling of $g^N$, 
employing the perturbation analysis by Gambia-Mushiest-Solo min in \cite{GMS:OPT}. 
Let $A$ be an elliptic non-negative self-adjoint differential operator of second order over a closed 
Riemannian manifold $(N, g^N)$ acting on sections of a Hermitian vector bundle $E$. Consider 
the square root $A_{\A}:=\sqrt{A+\A^2}$, which is a self-adjoint pseudo-differential operator 
of first order. Following the arguments of \cite[Corollary 1]{GMS:OPT} we arrive at the following result.

\begin{thm}\label{perturbation1}
The zeta-function of $(\Ae+\A)$ is analytic  for $\textup{Re}(s)>-(\dim N+2)$
with the following expansion \textup{(} $K=\dim N+2$ \textup{)}
$$\zeta(s,\Ae+\A)=\zeta(s,\Ae)+\sum_{k=1}^{K-1}\A^k\mathscr{T}_k(s,\A)+\A^{K}\mathscr{R}_K(s,\A),$$
where $\mathscr{R}_K(s,\A)$ is holomorphic and $\mathscr{R}_K, \partial_s\mathscr{R}_K$ are locally uniformly bounded in $\A$, and
$$\mathscr{T}_k(s,\A)=\frac{(-1)^k}{k!}\zeta (s+k, \Ae)\prod_{j=0}^{k-1}(s+j).$$
\end{thm}  

We now can prove the main result of this subsection.

\begin{cor}\label{zeta-convergence}
The zeta-functions of $(\Ae\pm\A)$ are regular at $s=0$ and 
$$\left.\frac{d}{ds}\right|_{0}\zeta(s,\Ae+\A)-\left.\frac{d}{ds}\right|_{0}\zeta(s,\Ae-\A)=O(\A), \quad \A\to 0.$$ 
\end{cor}
\begin{proof}
Theorem \ref{perturbation1} implies for $K=\dim N+2$
\begin{equation*}
  \zeta(s,\Ae+\A)-\zeta(s,\Ae-\A)=\sum_{k=0, \, \textup{odd}}^{K-1}\! 2 \, \A^{k} \mathscr{T}_{k}(s,\A) + 
\A^K\left(\mathscr{R}_K(s,\A)-\mathscr{R}_K(s,-\A))\right).
\end{equation*}
The term $\left(\mathscr{R}_K(s,\A)-\mathscr{R}_K(s,-\A))\right)$, and also its $\partial_s$-differential, are both regular at $s=0$ 
and locally uniformly bounded in $\A\in \R$. Hence it remains to analyze the terms $\mathscr{T}_k(s,\A)$.
Repeating the arguments of Theorem \ref{perturbation1} now for $(A+\A^2)$ as a perturbation of $A$, we find
\begin{equation}\label{zrm}
 \begin{split}
  \mathscr{T}_k(s,\A)
&=\zeta \left(\frac{(s+k)}{2}, A+\A^2\right)\left(\frac{(-1)^k}{k!}\prod_{j=0}^{k-1}(s+j)\right) \\
&=\left(\frac{(-1)^k}{k!}\prod_{j=0}^{k-1}(s+j)\right)\left[\zeta\left(\frac{(s+k)}{2}, A\right) + 
\sum_{p=1}^{F-1}\A^{2p}\zeta\left(\frac{(s+k)}{2}+p, A\right) \right. \\
&\left.\times \left(\frac{(-1)^p}{p!}\prod_{j=0}^{p-1}\left((\frac{(s+k)}{2}+j\right)\right)
+\widehat{\mathscr{R}}_F((s+n)/2,\A)\right],
 \end{split}
\end{equation}
Again, the term $\widehat{\mathscr{R}}_F$ and also its $\partial_s-$differential, are both regular at $s=0$ and locally 
uniformly bounded in $\A\in \R$. The zeta-functions $\zeta((s+k)/2+j,A)$ are meromorphic possibly with simple poles at $s=0$, 
canceled by the additional $s$-factor in \eqref{zrm}, so that $\mathscr{T}_k(s,\A)$ is regular at $s=0$ 
and its derivative at $s=0$ is locally uniformly bounded in $\A\in \R$. The statement now follows.
\end{proof}

We close the section by deriving the scaling behavior of the torsion-like invariant (Theorem \ref{main4}), 
combining of Theorem \ref{perturbation1} and Corollary \ref{zeta-convergence}.

\begin{thm}
Let $(N^n,g^N)$ be a closed oriented Riemannian manifold of even dimension and $(E_N,\nabla_N,h)$ a
flat complex Hermitian vector bundle. Let $g^N(\mu):=\mu^{-2}g^N, \mu\in \R^+$ and 
denote by $\Delta_{k,ccl,N}(\mu)$ the corresponding Laplacian on coclosed $k-$differential 
forms $\Omega^k_{\textup{ccl}}(N,E_N)$. Put $\A_k=(n-1)/2-k$ and define
\begin{align*}
&\nu(\eta)=\sqrt{\eta + \A_k^2}, \ \textup{for} \ \eta \in E_k=\textup{Spec}\Delta_{k,ccl,N}\backslash \{0\}, \\
&\zeta_{k,N}(s, \pm \A_k):=\sum_{\eta\in E_k} \textup{m}(\eta) \ (\nu(\eta)\pm \A_k)^{-s}, \quad Re(s)\gg0,
\end{align*}
where $\textup{m}(\eta)$ denotes the multiplicity of $\eta \in E_k$.
Then the associated family of "torsion-like" invariants $\textup{Tors}(N,E_N; g^N(\mu))$ 
admits the following asymptotic behavior as $\mu \to \infty$
\begin{align*}
\textup{Tors}(N,E_N; g^N(\mu))=\frac{1}{2}\sum_{k=0}^{n/2-1}(-1)^k \left(\zeta_{k,N}'(0,\A_k, \mu)-
\zeta_{k,N}'(0,-\A_k, \mu)\right)=O\left(\frac{\log \mu}{\mu}\right).
\end{align*}
\end{thm}

\begin{proof}
Write $A$ for the coclosed Laplacian on $\Omega^k_{\textup{ccl}}(N,E_N)$ associated to $g^N$.
Put $\A(\mu):= \A_k\mu^{-1}$ and consider $A_{\A(\mu)}$ in the notation as fixed before.
By definition, $\Delta_{k,ccl,N}(\mu)=\mu^2 A $, and consequently, $\textup{Spec}\Delta_{k,ccl,N}(\mu)=\mu^2 \textup{Spec}A$. Hence
$$\zeta_{k,N}(s,\pm\A_k, \mu)=\mu^{-s}\zeta(s, A_{\A(\mu)} \pm \A(\mu)).$$
Taking derivatives at $s=0$ we obtain
$$\left.\frac{d}{ds}\right|_{s=0}\zeta_{k,N}(0,\pm\A_k, \mu)=-\log \mu \cdot \zeta(0, A_{\A(\mu)}\pm \A(\mu))+\zeta'\left(0, A_{\A(\mu)}\pm \A(\mu)\right). $$
Theorem \ref{perturbation1} and Corollary \ref{zeta-convergence} assert that as $\mu\to \infty$
\begin{align*}
\zeta'(0, A_{\A(\mu)}+ \A(\mu))-\zeta'(0, A_{\A(\mu)}- \A(\mu))&=O\, \left(\frac{1}{\mu}\right),\\
\zeta(0, A_{\A(\mu)}+ \A(\mu))-\zeta(0, A_{\A(\mu)}- \A(\mu))&=O\,\left(\frac{1}{\mu}\right).
\end{align*}
This proves the statement.
\end{proof}

\bibliography{localbib.bib}
\bibliographystyle{amsalpha-lmp}

\listoffigures
\end{document}